\documentclass[10pt]{amsart}
\usepackage{amsmath, amssymb, latexsym}
\usepackage{mathrsfs}
\usepackage[all, knot]{xy}
\usepackage{color}
\xyoption{arc}

\newtheorem{Thm}{Theorem}[section]
\newtheorem{Prop}[Thm]{Proposition}
\newtheorem{Cor}[Thm]{Corollary}
\newtheorem{Lem}[Thm]{Lemma}

\theoremstyle{definition}
\newtheorem{Def}[Thm]{Definition}
\newtheorem{Exa}[Thm]{Example}

\numberwithin{equation}{section}

\newenvironment{red}
{\relax\color{red}}
{\hspace*{.5ex}\relax}

\newcommand{\ber}{\begin{red}}
\newcommand{\er}{\end{red}}

\newcommand{\Z}{\mathbb{Z}}
\newcommand{\Q}{\mathbb{Q}}

\newcommand{\g}{\mathfrak{g}}

\newcommand{\Span}{{\rm Span}}

\newcommand{\wt}{{\rm wt}}

\newcommand{\ind}{{\rm Ind}}

\newcommand{\res}{{\rm Res}}
\newcommand{\soc}{{\rm soc}}
\newcommand{\hd}{{\rm hd}}
\newcommand{\Hom}{{\rm Hom}}
\newcommand{\HOM}{{\rm HOM}}

\newcommand{\fMod}{ \mathrm{fmod}}

\newcommand{\ch}{{\rm ch}}
\newcommand{\qch}{{\rm ch}_{q}}
\newcommand{\qdim}{ \dim_q }
\newcommand{\infl}{ {\rm infl} }
\newcommand{\pr}{ {\rm pr} }

\newcommand{\Ire}{ I^{\rm re} }
\newcommand{\Iim}{ I^{\rm im} }
\newcommand{\seq}{ {\rm Seq} }
\newcommand{\seqd}{ {\rm Seqd} }
\newcommand{\F} { \bR }  
\newcommand{\pol} { { Pol}}
\newcommand{\sym} { { Sym}}
\newcommand{\ie} { 1}
\newcommand{\A} {\mathbb{A}}
\newcommand{\ep} { \varepsilon}
\newcommand{\ph} { \varphi}
\newcommand{\ke} { \tilde{e}}
\newcommand{\kf} { \tilde{f}}
\newcommand{\Bklr}[1] { \mathfrak{B}(#1)}
\newcommand{\sg} { S}
\newcommand{\End} { {\rm End}}
\newcommand{\id} { {\rm id}}

\newcommand{\bse} {e}
\newcommand{\bsf} {f}

\newcommand{\bR} {\mathbb{F}} 

\newcommand{\genOne}[3] 
{
\fontsize{7}{7}\selectfont
\xy
(0,6)*{}; (0,-6)*{} **\dir{-};
(4,0)*{\cdots};
(8,6)*{}; (8,-6)*{} **\dir{-};
(12,0)*{\cdots};
(16,6)*{}; (16,-6)*{} **\dir{-};
(0,-8)*{#1}; (8,-8)*{#2}; (16,-8)*{#3};
\endxy
\fontsize{10}{10}\selectfont
}

\newcommand{\genX}[3] 
{
\fontsize{7}{7}\selectfont
\xy
(0,6)*{}; (0,-6)*{} **\dir{-};
(4,0)*{\cdots};
(8,6)*{}; (8,-6)*{} **\dir{-};
(8.1,0)*{\bullet}; (12,0)*{\cdots};
(16,6)*{}; (16,-6)*{} **\dir{-};
(0,-8)*{#1}; (8,-8)*{#2}; (16,-8)*{#3};
\endxy
\fontsize{10}{10}\selectfont
}

\newcommand{\genTau}[4] 
{
\fontsize{7}{7}\selectfont
\xy
(0,6)*{}; (0,-6)*{} **\dir{-};
(4,0)*{\cdots};
(5,6)*{}; (11,-6)*{} **\dir{-};
(11,6)*{}; (5,-6)*{} **\dir{-};
(12,0)*{\cdots};
(16,6)*{}; (16,-6)*{} **\dir{-};
(0,-8)*{#1}; (5,-8)*{#2}; (11,-8)*{#3}; (16,-8)*{#4};
\endxy
\fontsize{10}{10}\selectfont
}

\newcommand{\twoStrands}[2] 
{\fontsize{7}{7}\selectfont
\xy
(0,6)*{}; (0,-6)*{} **\dir{-};
(6,6)*{}; (6,-6)*{} **\dir{-};
(0,-8)*{#1}; (6,-8)*{#2};
\endxy
\fontsize{10}{10}\selectfont
}

\newcommand{\twoDotStrandsL}[3] 
{ \fontsize{7}{7}\selectfont
\xy
(0,6)*{}; (0,-6)*{} **\dir{-};
(6,6)*{}; (6,-6)*{} **\dir{-};
(0,0)*{\bullet}; (-4,0)*{#1}; (0,-8)*{#2}; (6,-8)*{#3};
\endxy
\fontsize{10}{10}\selectfont
}

\newcommand{\twoDotStrandsR}[3] 
{\fontsize{7}{7}\selectfont
\xy
(0,6)*{}; (0,-6)*{} **\dir{-};
(6,6)*{}; (6,-6)*{} **\dir{-};
(6,0)*{\bullet}; (9.5,0)*{#1}; (0,-8)*{#2}; (6,-8)*{#3};
\endxy
\fontsize{10}{10}\selectfont}

\newcommand{\smalltwoDotStrandsL}[2] 
{ \fontsize{7}{7}\selectfont
\xy
(0,6)*{}; (0,-6)*{} **\dir{-};
(5,6)*{}; (5,-6)*{} **\dir{-};
(0,0)*{\bullet}; (0,-8)*{#1}; (5,-8)*{#2};
\endxy
\fontsize{10}{10}\selectfont
}

\newcommand{\smalltwoDotStrandsR}[2] 
{\fontsize{7}{7}\selectfont
\xy
(0,6)*{}; (0,-6)*{} **\dir{-};
(5,6)*{}; (5,-6)*{} **\dir{-};
(5,0)*{\bullet}; (0,-8)*{#1}; (5,-8)*{#2};
\endxy
\fontsize{10}{10}\selectfont}

\newcommand{\dCross}[2] 
{\fontsize{7}{7}\selectfont
\xy
(0,9)*{}="T1"; (5,9)*{}="T2";
(0,-9)*{}="B1"; (5,-9)*{}="B2";
"T1"; "B1" **\crv{(10, 0)};
"T2"; "B2" **\crv{(-5,0)};
(0,-11)*{#1}; (5,-11)*{#2};
\endxy
\fontsize{10}{10}\selectfont}

\newcommand{\CrossUL}[3] 
{\fontsize{7}{7}\selectfont
\xy
(0,6)*{}; (9,-6)*{} **\dir{-} \POS?(.25)="x";
(9,6)*{}; (0,-6)*{} **\dir{-};
"x"*{\bullet}; "x"+(-2,2)*{#1}; (0,-8)*{#2}; (9,-8)*{#3}; \endxy
\fontsize{10}{10}\selectfont}

\newcommand{\CrossUR}[3] 
{\fontsize{7}{7}\selectfont
\xy
(0,6)*{}; (9,-6)*{} **\dir{-};
(9,6)*{}; (0,-6)*{} **\dir{-} \POS?(.25)="x";
"x"*{\bullet}; "x"+(2,2)*{#1}; (0,-8)*{#2}; (9,-8)*{#3}; \endxy
\fontsize{10}{10}\selectfont}

\newcommand{\CrossDR}[3] 
{\fontsize{7}{7}\selectfont
\xy
(0,6)*{}; (9,-6)*{} **\dir{-} \POS?(.75)="x";
(9,6)*{}; (0,-6)*{} **\dir{-};
"x"*{\bullet}; "x"+(2,2)*{#1}; (0,-8)*{#2}; (9,-8)*{#3}; \endxy
\fontsize{10}{10}\selectfont}

\newcommand{\CrossDL}[3] 
{\fontsize{7}{7}\selectfont
\xy
(0,6)*{}; (9,-6)*{} **\dir{-};
(9,6)*{}; (0,-6)*{} **\dir{-} \POS?(.75)="x";
"x"*{\bullet}; "x"+(-2,2)*{#1}; (0,-8)*{#2}; (9,-8)*{#3}; \endxy
\fontsize{10}{10}\selectfont}

\newcommand{\BraidL}[3]  
{\fontsize{7}{7}\selectfont
\xy
(0,9)*{}="T1"; (6,9)*{}="T2";  (12,9)*{}="T3";
(0,-9)*{}="B1"; (6,-9)*{}="B2"; (12,-9)*{}="B3";
"T1"; "B3" **\dir{-}; "T3"; "B1" **\dir{-};
"T2"; "B2" **\crv{(-7,0)};
(0,-11)*{#1}; (6,-11)*{#2};(12,-11)*{#3};
\endxy
\fontsize{10}{10}\selectfont}

\newcommand{\BraidR}[3]  
{\fontsize{7}{7}\selectfont
\xy
(0,9)*{}="T1"; (6,9)*{}="T2";  (12,9)*{}="T3";
(0,-9)*{}="B1"; (6,-9)*{}="B2"; (12,-9)*{}="B3";
"T1"; "B3" **\dir{-}; "T3"; "B1" **\dir{-};
"T2"; "B2" **\crv{(17,0)};
(0,-11)*{#1}; (6,-11)*{#2};(12,-11)*{#3};
\endxy
\fontsize{10}{10}\selectfont}

\newcommand{\threeDotStrands}[5] 
{ \fontsize{7}{7}\selectfont
\xy
(0,5)*{}; (0,-5)*{} **\dir{-};
(5,5)*{}; (5,-5)*{} **\dir{-};
(10,5)*{}; (10,-5)*{} **\dir{-};
(0,0)*{\bullet};(10,0)*{\bullet}; (-2,0)*{#1};(19,0)*{#2}; (0,-7)*{#3}; (5,-7)*{#4};(10,-7)*{#5};
\endxy
\fontsize{10}{10}\selectfont
}

\newcommand{\CrossL}[3]   
{\fontsize{7}{7}\selectfont
\xy
(0,5)*{}; (5,-5)*{} **\dir{-};
(5,5)*{}; (0,-5)*{} **\dir{-};
(8,5)*{}; (8,-5)*{} **\dir{-};
(0,-7)*{#1}; (5,-7)*{#2}; (8,-7)*{#3}; \endxy
\fontsize{10}{10}\selectfont}

\newcommand{\CrossR}[3]   
{\fontsize{7}{7}\selectfont
\xy
(-3,5)*{}; (-3,-5)*{} **\dir{-};
(0,5)*{}; (5,-5)*{} **\dir{-};
(5,5)*{}; (0,-5)*{} **\dir{-};
(-3,-7)*{#1}; (0,-7)*{#2}; (5,-7)*{#3};\endxy
\fontsize{10}{10}\selectfont}

\newcommand{\smallthreeDotStrandsL}[3] 
{ \fontsize{7}{7}\selectfont
\xy
(0,6)*{}; (0,-6)*{} **\dir{-};
(4,6)*{}; (4,-6)*{} **\dir{-};
(8,6)*{}; (8,-6)*{} **\dir{-};
(0,0)*{\bullet}; (0,-8)*{#1}; (4,-8)*{#2}; (8,-8)*{#3};
\endxy
\fontsize{10}{10}\selectfont
}

\newcommand{\smallthreeDotStrandsM}[3] 
{\fontsize{7}{7}\selectfont
\xy
(0,6)*{}; (0,-6)*{} **\dir{-};
(4,6)*{}; (4,-6)*{} **\dir{-};
(8,6)*{}; (8,-6)*{} **\dir{-};
(4,0)*{\bullet}; (0,-8)*{#1}; (4,-8)*{#2}; (8,-8)*{#3};
\endxy
\fontsize{10}{10}\selectfont}

\newcommand{\smallthreeDotStrandsR}[3] 
{\fontsize{7}{7}\selectfont
\xy
(0,6)*{}; (0,-6)*{} **\dir{-};
(4,6)*{}; (4,-6)*{} **\dir{-};
(8,6)*{}; (8,-6)*{} **\dir{-};
(8,0)*{\bullet}; (0,-8)*{#1}; (4,-8)*{#2}; (8,-8)*{#3};
\endxy
\fontsize{10}{10}\selectfont}
\newcommand{\alphaabmi}
{
\fontsize{7}{7}\selectfont
\xy
(-13,14)*{\overbrace{ \quad \quad \quad \quad \quad \quad }^{a}};
(13,14)*{\overbrace{ \quad \quad  \quad  \quad \quad \quad }^{b}};
(-20,12)*{}; (0,0)*{} **\crv{(-19,6) & (-0,6.5)};
(-15,12)*{}; (-20,0)*{} **\crv{(-15,6) & (-22,6.5)};
(-10,12)*{}; (-15,0)*{} **\crv{(-10,6) & (-17,6.5)};
(-5,12)*{}; (-10,0)*{} **\crv{(-5,6) & (-12,6.5)};
(-20,-10)*{}; (-20,-20)*{} **\dir{-}; (-20,-22)*{i};
(-15,-10)*{}; (-15,-20)*{} **\dir{-}; (-15,-22)*{i};
(-10,-10)*{}; (-10,-20)*{} **\dir{-}; (-10,-22)*{i};
(-22,0)*{}; (-8,0)*{} **\dir{-}; (-22,0)*{}; (-22,-10)*{} **\dir{-};
(-15,-5)*{1_{i,a-1}};
(-22,-10)*{}; (-8,-10)*{} **\dir{-}; (-8,0)*{}; (-8,-10)*{} **\dir{-};
(5,12)*{}; (5,0)*{} **\dir{-}; (10,12)*{}; (10,0)*{} **\dir{-};
(15,12)*{}; (15,0)*{} **\dir{-}; (20,12)*{}; (20,0)*{} **\dir{-};
(0,12)*{}; (-5,-20) **\crv{(0,6) & (-6.5,6.5)}; (-5,-22)*{j};
(-2,0)*{}; (22,0)*{} **\dir{-}; (-2,0)*{}; (-2,-10)*{} **\dir{-};
(9,-5)*{1_{i,b+1}};
(22,0)*{}; (22,-10)*{} **\dir{-}; (-2,-10)*{}; (22,-10)*{} **\dir{-};
(0,-10)*{}; (0,-20)*{} **\dir{-}; (0,-22)*{i};
(5,-10)*{}; (5,-20)*{} **\dir{-}; (5,-22)*{i};
(10,-10)*{}; (10,-20)*{} **\dir{-}; (10,-22)*{i};
(15,-10)*{}; (15,-20)*{} **\dir{-}; (15,-22)*{i};
(20,-10)*{}; (20,-20)*{} **\dir{-}; (20,-22)*{i};
(-14.5,-25)*{\underbrace{ \quad \quad \quad \quad   }_{a-1}};
(10,-25)*{\underbrace{ \quad \quad \quad \quad \quad \quad \quad}_{b+1}};

\endxy
\fontsize{10}{10}\selectfont
}

\newcommand{\alphaABmi}
{\xy (0,0)*++{\alphaabmi};
\endxy}

\newcommand{\alphaabpl}{
\fontsize{7}{7}\selectfont
\xy
(-12,14)*{\overbrace{ \quad \quad \ }^{a}};
(10,14)*{\overbrace{ \quad \quad  \quad \quad  \quad \quad \quad  }^{b}};
(-5,12)*{}; (0,-20) **\crv{(-5,6) & (0,6.5)}; (0,-22)*{j};
(20,12)*{}; (-5,0)*{} **\crv{(19,6) & (-5,6.5)};
(0,12)*{}; (5,0)*{} **\crv{(0,6) & (7,6.5)};
(5,12)*{}; (10,0)*{} **\crv{(5,6) & (12,4)};
(10,12)*{}; (15,0)*{} **\crv{(10,6) & (17,4)};
(15,12)*{}; (20,0)*{} **\crv{(13,6) & (20.5,4)};
(-10,12)*{}; (-10,0)*{} **\dir{-}; (-15,12)*{}; (-15,0)*{} **\dir{-};
(-17,0)*{}; (-3,0)*{} **\dir{-}; (-17,0)*{}; (-17,-10)*{} **\dir{-};
(-10,-5)*{1_{i,a+1}};
(-17,-10)*{}; (-3,-10)*{} **\dir{-}; (-3,0)*{}; (-3,-10)*{} **\dir{-};
(3,0)*{}; (22,0)*{} **\dir{-}; (3,0)*{}; (3,-10)*{} **\dir{-};
(12,-5)*{1_{i,b-1}};
(22,0)*{}; (22,-10)*{} **\dir{-}; (3,-10)*{}; (22,-10)*{} **\dir{-};
(5,-10)*{}; (5,-20)*{} **\dir{-}; (5,-22)*{i};
(10,-10)*{}; (10,-20)*{} **\dir{-}; (10,-22)*{i};
(15,-10)*{}; (15,-20)*{} **\dir{-}; (15,-22)*{i};
(20,-10)*{}; (20,-20)*{} **\dir{-}; (20,-22)*{i};
(-5,-10)*{}; (-5,-20)*{} **\dir{-}; (-5,-22)*{i};
(-10,-10)*{}; (-10,-20)*{} **\dir{-}; (-10,-22)*{i};
(-15,-10)*{}; (-15,-20)*{} **\dir{-}; (-15,-22)*{i};
(-10,-25)*{\underbrace{ \quad \quad \quad \quad    }_{a+1}};
(13,-25)*{\underbrace{ \quad \quad \quad \quad \quad \quad}_{b-1}};
\endxy
\fontsize{10}{10}\selectfont
}

\newcommand{\alphaABpl}
{\xy (0,0)*++{\alphaabpl};
\endxy}

\begin{document}

\title[Categorification of Quantum Generalized Kac-Moody Algebras and Crystal Bases]
{Categorification of Quantum Generalized Kac-Moody Algebras and Crystal Bases}
\author[Seok-Jin Kang]{Seok-Jin Kang $^{1,2}$}
\thanks{$^1$ This work was supported by KRF Grant \# 2007-341-C00001.}
\thanks{$^2$ This work was supported by NRF Grant \# 2010-0010753.}
\address{Department of Mathematical Sciences and Research Institute of Mathematics,
Seoul National University, 599 Gwanak-ro, 
Gwanak-gu, Seoul 151-747, Korea} \email{sjkang@math.snu.ac.kr}
\author[Se-jin Oh]{Se-jin Oh $^{3,4}$ }
\thanks{$^3$ This work was supported by NRF Grant \# 2010-0019516.}
\thanks{$^4$ This work was supported by BK21 Mathematical Sciences Division.}
\address{Department of Mathematical Sciences, Seoul National University,
599 Gwanak-ro, 
Gwanak-gu, Seoul 151-747, Korea} \email{sj092@snu.ac.kr}
\author[Euiyong Park]{Euiyong Park $^{1,2}$}

\address{School of Mathematics, Korea Institute for Advanced Study,
87 Hoegiro, Dongdaemun-gu, Seoul 130-722, Korea}
\email{eypark@kias.re.kr}

\subjclass[2000]{05E10, 16G99, 81R10} \keywords{categorification,
Khovanov-Lauda-Rouquier algebras, crystal bases, quantum generalized
Kac-Moody algebras}

\begin{abstract}
We construct and investigate the structure of the
Khovanov-Lauda-Rouquier algebras $R$ and their cyclotomic quotients
$R^\lambda$ which give a categorification of quantum generalized
Kac-Moody algebras. Let $U_\A(\g)$ be the integral form of the
quantum generalized Kac-Moody algebra associated with a
Borcherds-Cartan matrix $A=(a_{ij})_{i,j \in I}$ and let $K_0(R)$ be
the Grothendieck group of finitely generated projective graded
$R$-modules. We prove that there exists an injective algebra
homomorphism $\Phi: U_\A^-(\g) \rightarrow K_0(R)$ and that $\Phi$
is an isomorphism if $a_{ii}\ne 0$ for all $i\in I$. Let $B(\infty)$
and $B(\lambda)$ be the crystals of $U_q^-(\g)$ and $V(\lambda)$,
respectively, where $V(\lambda)$ is the irreducible highest weight
$U_q(\g)$-module. We denote by $\Bklr{\infty}$ and $\Bklr{\lambda}$
the isomorphism classes of irreducible graded modules over $R$ and
$R^\lambda$, respectively. If $a_{ii}\ne 0$ for all $i\in I$, we
define the $U_q(\g)$-crystal structures on $\Bklr{\infty}$ and
$\Bklr{\lambda}$, and show that there exist crystal isomorphisms
$\Bklr{\infty} \simeq B(\infty)$ and $\Bklr{\lambda} \simeq
B(\lambda)$. One of the key ingredients of our approach is the
perfect basis theory for generalized Kac-Moody algebras.
\end{abstract}

\maketitle


\section*{Introduction}

In \cite{KL09, KL11} and \cite{R08}, Khovanov-Lauda and Rouquier
independently introduced a new family of graded algebras $R$ which
gives a {\it categorification} of quantum groups associated with
symmetrizable Kac-Moody algebras. More precisely, let $U_q(\g)$ be
the quantum group associated with a symmetrizable Kac-Moody algebra
and let $U_{\A}(\g)$ be the integral form of $U_q(\g)$, where $\A =
\Z[q, q^{-1}]$. Then it was shown
that the Grothendieck group $K_{0}(R)$ of finitely generated graded
projective $R$-modules is isomorphic to $U_{\A}^{-}(\g)$, the
negative part of $U_{\A}(\g)$. Furthermore, for symmetric Kac-Moody
algebras, Varagnolo and Vasserot proved that the isomorphism classes
of principal indecomposable $R$-modules correspond to Lusztig's {\it
canonical basis} (or Kashiwara's {\it lower global basis}) under
this isomorphism \cite{VV09}. The algebra $R$ is called the {\it
Khovanov-Lauda-Rouquier algebra} associated with $\g$.

For each dominant integral weight $\lambda \in P^{+}$, the algebra
$R$ has a special quotient $R^{\lambda}$ which is called the {\it
cyclotomic quotient}. It was conjectured that the cyclotomic
quotient $R^{\lambda}$ gives a categorification of the irreducible
highest weight module $V(\lambda)$ \cite{KL09}. For type $A_\infty$
and $A_n^{(1)}$, this conjecture was proved in \cite{BK08,BK09}.
In \cite{KK11}, Kang and Kashiwara proved Khovanov-Lauda categorification
conjecture for {\it all} symmetrizable Kac-Moody algebras.
Webster also gave a proof of this conjecture by a completely different method \cite{Webster10}. 
In \cite{LV09}, the crystal version of this conjecture was proved.
That is, in \cite{LV09}, Lauda
and Vazirani investigated the crystal structure on the set of
isomorphism classes of irreducible graded modules over $R$ and
$R^\lambda$, and showed that these crystals are isomorphic to the
crystals $B(\infty)$ and $B(\lambda)$, respectively.

The purpose of this paper is to extend the study of
Khovanov-Lauda-Rouquier algebras to the case of {\it generalized
Kac-Moody algebras}. The generalized Kac-Moody algebras were
introduced by Borcherds in his study of Monstrous Moonshine
\cite{Bor88}, and they form an important class of algebraic
structure behind many research areas such as algebraic geometry,
number theory and string theory (see, for example, \cite{Bor92,
FRS97, GN98a, GN98b,HM96, HM98,  KangKwon00, Moore98, Nai95, Sch04,
Sch06}). In particular, the {\it Monster Lie algebra}, a special
example of generalized Kac-Moody algebras, played a crucial role in
proving the Moonshine conjecture \cite{Bor92}.
Moreover, the generalized Kac-Moody algebras draw more and more attention among mathematical
physicists due to their connection with string theory and other related topics.
The quantum
deformations of generalized Kac-Moody algebras and their integrable
highest weight modules were constructed in \cite{Kang95} and the
crystal basis theory for quantum generalized Kac-Moody algebras was
developed in \cite{JKK05, JKKS07}. In \cite{KO06}, the canonical
bases for quantum generalized Kac-Moody algebras were realized as
certain semisimple perverse sheaves, and in \cite{KKO09a,KKO09b}, a
geometric construction of crystals $B(\infty)$ and $B(\lambda)$ was
given using Lusztig's and Nakajima's quiver varieties, respectively.

In this paper, we construct and investigate the structure of
Khovanov-Lauda-Rouquier algebras $R$ and their cyclotomic quotients
$R^{\lambda}$ which give a categorification of quantum generalized
Kac-Moody algebras. Let $U_q(\g)$ be the quantum generalized
Kac-Moody algebra associated with a Borcherds-Cartan matrix
$A=(a_{ij})_{i,j\in I}$. We first define the Khovanov-Lauda-Rouquier
algebra $R$ in terms of generators and relations.
A big contrast with the case of Kac-Moody algebras is that the nil Hecke
algebras corresponding to the {\it imaginary} simple roots with norm $\le 0$
may have nonconstant twisting factors for commutation and braid relations.
In this work,
we choose any homogeneous polynomials ${\mathcal
P}_i(u,v)$ of degree $1 - \dfrac{a_{ii}}{2}$ and their variants
$\overline{\mathcal P}_{i}'$ and $\overline{\mathcal P}_{i}''$ $(i \in I)$ as these twisting
factors (see Definition \ref{def:KLR}).
When $a_{ii}=2$, we are reduced to the case of Kac-Moody algebras.
The role of these twisting factors is still mysterious.
For convenience, we also
give a diagrammatic presentation of the algebra $R$.

Next, we show that there exists an injective algebra homomorphism
$\Phi: U_{\A}^{-}(\g) \longrightarrow K_{0}(R)$, where $K_{0}(R)$ is
the Grothendieck group of finitely generated graded projective
$R$-modules (Theorem \ref{Thm:Phi is injective}). Thus $\text{Im}
\Phi$ gives a categorification of $U_q^{-}(\g)$.
To do this, we need to show that the quantum Serre relations are preserved by the map $\Phi$.
In general, $\Phi$
is not surjective
even for the case $A = (0)$. The whole Grothendieck group seems rather large and nontrivial.
However, if $a_{ii} \neq 0$ for all $i \in I$, we
can show that $\Phi$ is an isomorphism (Theorem \ref{Thm:iso of K0
and Uq}). As in the case of Kac-Moody algebras, we conjecture that,
if the Borcherds-Cartan matrix $A=(a_{ij})_{i,j \in I}$ is symmetric
and $a_{ii} \neq 0$ for all $i \in I$, then the isomorphism classes
of graded projective indecomposable $R$-modules correspond to
canonical basis elements under the isomorphism $\Phi$. We will
investigate this conjecture in a forthcoming paper following the
framework given in \cite{KO06, VV09}.

Now we focus on the crystal structures.
We would like to emphasize that
one of the key ingredients
of our approach is the {\it perfect basis theory} for generalized
Kac-Moody algebras and it can be applied to
the Kac-Moody algebras setting as well.
Our work is different from \cite{LV09} in this respect.
In \cite{BerKaz07}, Berenstein and Kazhdan
introduced the notion of perfect bases for integrable highest weight
modules $V(\lambda)$ $(\lambda \in P^{+})$ over Kac-Moody algebras.
They showed that the colored oriented graphs arising from perfect
bases are all isomorphic to the crystal $B(\lambda)$. Their work was
extended to the integrable highest weight modules over generalized
Kac-Moody algebras in \cite{KOP09}. In this work, we define the
notion of perfect bases for $U_q^{-}(\g)$ as a module over the {\it
quantum boson algebra} $B_q(\g)$. The existence of perfect basis for
$U_q^{-}(\g)$ is provided by constructing the {\it upper global
basis} (or {\it dual canonical basis}) of $U_q^{-}(\g)$. We also
show that the crystal arising from any perfect basis of $U^-_q(\g)$ is
isomorphic to the crystal $B(\infty)$ (Theorem \ref{Thm: uniqueness
of perfect graphs}).

With perfect basis theory at hand, we construct the crystal
$\Bklr{\infty}$ as follows. Let $G_0(R)$ be the Grothendieck group of
finite-dimensional graded $R$-modules and set $G_0(R)_{\Q(q)} =
\Q(q) \otimes_\A G_0(R)$. We denote by $\Bklr{\infty}$ the set of
isomorphism classes of irreducible graded $R$-modules and define the
crystal operators using induction and restriction functors.
Moreover, we show that $G_0(R)_{\Q(q)}$ has a $B_q(\g)$-module
structure and that if $a_{ii} \neq 0$ for all $i\in I$, then
$\Bklr{\infty}$ is a perfect basis of $G_0(R)_{\Q(q)}$. Therefore,
by the main theorem of perfect basis theory, we obtain a crystal
isomorphism (Theorem \ref{Thm: B(infty)}):
$$\Bklr{\infty} \simeq
B(\infty).$$

For a dominant integral weight $\lambda \in P^{+}$, we define the
{\it cyclotomic Khovanov-Lauda-Rouquier algebra} $R^\lambda$ to be
the quotient of $R$ by a certain two-sided ideal depending on
$\lambda$. Let $\Bklr{\lambda}$ denote the set of isomorphism
classes of irreducible graded $R^\lambda$-modules and define the
crystal operators using induction/restriction functors and
projection/inflation functors. It was shown in \cite{JKKS07} that
there exists a strict crystal embedding
$$ B(\lambda) \hookrightarrow
B(\infty) \otimes T_\lambda \otimes C.$$ If $a_{ii} \neq 0$ for all
$i \in I$, using the above crystal embedding, we construct a crystal
isomorphism (Theorem \ref{Thm: B(lambda)}):
$$ \Bklr{\lambda} \simeq B(\lambda).$$

In \cite{KKO11}, after this work was completed, Khovanov-Lauda cyclotomic conjecture
was proved for all symmetrizable generalize Kac-Moody algebras.

This paper is organized as follows. Section 1 contains a brief
review of quantum generalized Kac-Moody algebras and crystal bases.
In Section 2, we define the Khovanov-Lauda-Rouquier algebra $R$
associated with a Borcherds-Cartan matrix $A=(a_{ij})_{i,j \in I}$,
and investigate its algebraic structure and representation theory.
We construct a faithful polynomial representation of $R(\alpha)$ and
prove the Khovanov-Lauda-Rouquier algebra version of the quantum
Serre relations. In Section 3, we show that the algebra $R$ gives a
categorification of $U_\A^-(\g)$. We define a twisted bialgebra
structure on $K_0(R)$ using induction and restriction functors, and
show that there exists an injective algebra homomorphism $\Phi:
U_\A^-(\g) \longrightarrow K_0(R)$. In particular, we prove that
$U_\A^-(\g) \simeq K_0(R)$ when $a_{ii} \ne 0$ for all $i\in I$.
Section 4 is devoted to the theory of perfect bases. We define the
notion of perfect bases for $U_q^-(\g)$ as a $B_q(\g)$-module and
show that $U_q^{-}(\g)$ has a perfect basis by constructing the
upper global basis of $U_q^-(\g)$. The main theorem in Section 4
asserts that the crystals arising from perfect bases are all
isomorphic to $B(\infty)$. In Section 5, we study the crystal
structures on $\Bklr{\infty}$ and $\Bklr{\lambda}$. Using the theory
of perfect bases, we prove that there exists a crystal isomorphism
$\Bklr{\infty} \simeq B(\infty)$ when $a_{ii} \ne 0$ for $i\in I$.
Furthermore, we define the cyclotomic quotient $R^\lambda$ of $R$,
and investigate the basic properties of irreducible
$R^\lambda$-modules. Combining the isomorphism $\Bklr{\infty} \simeq
B(\infty)$ with the strict embedding $ B(\lambda) \hookrightarrow
B(\infty) \otimes T_\lambda \otimes C$, we obtain a crystal
isomorphism $\Bklr{\lambda} \simeq B(\lambda)$.

\vskip 3em

\section{Quantum generalized Kac-Moody algebras} \label{Sec:GKM}

Let $I$ be a countable (possibly infinite) index set. A matrix
$A=(a_{ij})_{i,j \in I}$ with $a_{ij} \in \Z$ is called an {\it even
integral Borcherds-Cartan matrix} if it satisfies (i) $a_{ii} = 2
\text{ or } a_{ii} \in 2 \Z_{\le 0}$, (ii) $a_{ij} \le 0 \text{ for
} i \neq j$, (iii) $a_{ij}=0 \text{ if and only if } a_{ji}=0$. For
$i \in I$, $i$ is said to be \emph{real} if $a_{ii}=2$ and is said
to be \emph{imaginary} otherwise. We denote by $\Ire$ the set of all
real indices and by $\Iim$ the set of all imaginary indices. In this
paper, we assume that $A$ is {\em symmetrizable}; i.e., there is a
diagonal matrix $D={\rm diag}( s_i \in \Z_{> 0} | i \in I)$ such that $DA$ is
symmetric.

A \emph{Borcherds-Cartan datum} $(A,P,\Pi,\Pi^{\vee})$ consists of
\begin{enumerate}
\item[(1)] a Borcherds-Cartan matrix $A$,
\item[(2)] a free abelian group $P$, the \emph{weight lattice},
\item[(3)] $\Pi= \{ \alpha_i \in P \mid \ i \in I \}$, the set of \emph{simple roots},
\item[(4)] $\Pi^{\vee}= \{ h_i \ | \ i \in I  \} \subset P^{\vee}:=\Hom(P,\Z)$, the set of \emph{simple coroots},
\end{enumerate}
satsifying the following properties:
\begin{enumerate}
\item[(a)] $\langle h_i,\alpha_j \rangle = a_{ij}$ for all $i,j \in I$,
\item[(b)] $\Pi$ is linearly independent,
\item[(c)] for any $i \in I$, there exists $\Lambda_i \in P$ such that
           $\langle h_j ,\Lambda_i \rangle =\delta_{ij}$ for all $j \in I$.
\end{enumerate}

Let $\mathfrak{h} = \Q \otimes_\Z P^{\vee}$. Since $A$ is symmetrizable, there is a symmetric biliear form $( \ | \ )$ on $\mathfrak{h}^*$ satisfying
$$ (\alpha_i | \alpha_j) = s_i a_{ij} \quad (i,j \in I). $$
We denote by $P^{+} := \{ \lambda \in P | \lambda(h_i) \in \Z_{\ge
0}, i \in I \}$ the set of \emph{dominant integral weights}. The
free abelian group $Q= \oplus_{i \in I} \Z \alpha_i$ is called the
\emph{root lattice}. Set $Q^{+}= \sum_{i \in I} \Z_{\ge 0}
\alpha_i$. For $\alpha = \sum k_i \alpha_i \in Q^{+}$, we denote by
$|\alpha|$ the {\it height} of $\alpha$: $|\alpha|=\sum k_i$.

Let $q$ be an indeterminate and $m,n \in \Z_{\ge 0}$. Set $c_i =
-\frac{1}{2}a_{ii}$ and $q_i = q^{s_i}$ for $i\in I$. If $i \in
\Ire$, define
\begin{equation*}
 \begin{aligned}
 \ &[n]_i =\frac{ q^n_{i} - q^{-n}_{i} }{ q_{i} - q^{-1}_{i} },
 \ &[n]_i! = \prod^{n}_{k=1} [k]_i ,
 \ &\left[\begin{matrix}m \\ n\\ \end{matrix} \right]_i=  \frac{ [m]_i! }{[m-n]_i! [n]_i! }.
 \end{aligned}
\end{equation*}
 If $a_{ii} <0$, we define
\begin{equation*}
 \begin{aligned}
 \ &\{n\}_i =\frac{ q^{c_i  n}_{i} - q^{-c_i  n}_{i} }{ q^{c_i}_{i} - q^{-c_i}_{i} } ,
 \ &\{n\}_i! = \prod^{n}_{k=1} \{k\}_i ,
 \ &\left\{ \begin{matrix}m \\ n\\ \end{matrix} \right\}_i=  \frac{ \{m\}_i! }{\{m-n\}_i! \{n\}_i! }.
\end{aligned}
\end{equation*}
If $a_{ii}=0$, we define
\begin{equation*}
 \begin{aligned}
 \ \ &\{n\}_i =n ,
 \ \ & \{n\}_i! = n!,
 \ \ &\left\{ \begin{matrix}m \\ n\\ \end{matrix} \right\}_i=\left( \begin{matrix}m \\ n\\ \end{matrix} \right).
 \end{aligned}
\end{equation*}

\begin{Def} \label{Def: GKM}
The {\em quantum generalized Kac-Moody algebra} $U_q(\g)$ associated
with a Borcherds-Cartan datum $(A,P,\Pi,\Pi^{\vee})$ is the associative
algebra over $\Q(q)$ with ${\bf 1}$ generated by $e_i,f_i$ $(i \in I)$ and
$q^{h}$ $(h \in P^{\vee})$ satisfying following relations:
\begin{enumerate}
  \item  $q^0=1, q^{h} q^{h'}=q^{h+h'} $ for $ h,h' \in P^{\vee},$
  \item  $q^{h}e_i q^{-h}= q^{\langle h, \alpha_i \rangle} e_i,
          \ q^{h}f_i q^{-h} = q^{-\langle h, \alpha_i \rangle }f_i$ for $h \in P^{\vee}, i \in I$,
  \item  $e_if_j - f_je_i =  \delta_{ij} \dfrac{K_i -K^{-1}_i}{q_i- q^{-1}_i }, \ \ \mbox{ where } K_i=q_i^{ h_i},$
  \item  $\displaystyle \sum^{1-a_{ij}}_{r=0} (-1)^r \left[\begin{matrix}1-a_{ij} \\ r\\ \end{matrix} \right]_i e^{1-a_{ij}-r}_i
         e_j e^{r}_i =0 \quad \text{ if } i\in \Ire \text{ and } i \ne j, $
  \item $\displaystyle \sum^{1-a_{ij}}_{r=0} (-1)^r \left[\begin{matrix}1-a_{ij} \\ r\\ \end{matrix} \right]_i f^{1-a_{ij}-r}_if_j
        f^{r}_i=0 \quad \text{ if } i \in \Ire \text{ and } i \ne j, $
  \item $ e_ie_j - e_je_i=0,\ f_if_j-f_jf_i =0  \ \ \mbox{ if }a_{ij}=0.$
\end{enumerate}
\end{Def}

Let $U_q^{+}(\g)$ (resp.\ $U_q^{-}(\g)$) be the subalgebra of $U_q(\g)$ generated by the elements $e_i$
(resp.\ $f_i$), and let $U^0_q(\g)$ be the subalgebra of $U_q(\g)$ generated
by $q^{h}$ $(h \in P^{\vee})$. Then we have the \emph{triangular decomposition}
$$ U_q(\g) \cong U^{-}_q(\g) \otimes U^{0}_q(\g) \otimes U^{+}_q(\g),$$
and the {\em root space decomposition}
$$U_q(\g) = \bigoplus_{\alpha \in Q} U_q(\g)_{\alpha},$$
where $U_q(\g)_{\alpha}:=\{ x \in U_q(\g) \mid q^{h}x q^{-h}=q^{\langle h, \alpha \rangle}x \text{ for any } h \in P^{\vee} \}$.
Define a $\Q$-algebra automorphism \ $\bar {}  : U^{-}_q(\g) \to U^{-}_q(\g)$ by
\begin{equation}\label{Eq:bar involution}
e_i \mapsto e_i, \ \ f_i \mapsto f_i, \ \ q^h \mapsto q^{-h}, \ \ q
\mapsto q^{-1}.
\end{equation}

Let $\A =\Z[q,q^{-1}]$. For $n \in \Z_{>0}$, set
$$ e_i^{(n)} = \begin{cases} \dfrac{e_i^{n}}{[n]_i!} \ \  & \text{if} \ i \in \Ire ,\\
                             e_i^{n} \ \  & \text{if} \ i \in \Iim, \end{cases}
\quad \quad f_i^{(n)} = \begin{cases} \dfrac{f_i^{n}}{[n]_i!} \ \  & \text{if} \ i \in \Ire, \\
                             f_i^{n} \ \  & \text{if} \ i \in \Iim, \end{cases}
$$
and denote by $U^{-}_{\A}(\g)$ (resp.\ $U^{+}_{\A}(\g)$) the $\A$-sualgebra of $U_q^{-}(\g)$ generated by $f_i^{(n)}$
(resp.\ $e_i^{(n)}$).

Define a twisted algebra structure on  $U^{-}_q(\g) \otimes U^{-}_q(\g)$ as follows:
$$ (x_1 \otimes x_2)(y_1 \otimes y_2)= q^{-( \beta_2 | \gamma_1)}(x_1 y_1 \otimes x_2 y_2 ),$$
where $x_i \in U^-_q(\g) _{\beta_i}$ and $ y_i \in U^-_q(\g) _{\gamma_i} $ ($i=1,2$).
Then there is an algebra homomorphism $\Delta_0: U^{-}_q(\g) \to U^{-}_q(\g) \otimes U^{-}_q(\g)$
satisfying
\begin{align} \label{Eq:def of Delta 0}
 \Delta_0(f_i) := f_i \otimes {\bf 1} + {\bf 1} \otimes f_i\ (i\in I).
\end{align}

Fix $i \in I$. For any $P \in U^{-}_q(\g)$, there exist unique elements $Q,R \in U^{-}_q(\g)$ such that
$$ e_i P - P e_i = \frac{K_i Q - K_i^{-1}R}{q_i -q_i^{-1}}.$$
We define the endomorphisms ${e_i'},\bse_i'': U^{-}_q(\g) \to U^{-}_q(\g) $ by
$$ {\bse_i'}(P)=R,\ \  \mathrm{e''_i}(P)=Q .$$
Consider ${\bsf_i}$ as the endomorphism of $U^{-}_q(\g)$ defined by left multiplication by $f_i$.
Then we have
\begin{equation} \label{eq: special commute}
\begin{aligned}
{\bse_i'} {\bsf_j} =  \delta_{ij} + q_i^{-a_{ij}}{\bsf_j}{\bse_i'}.
 \end{aligned}
\end{equation}

\begin{Def}
The {\it quantum boson algebra}  $B_{q}(\g)$ associated with a
Borcherds-Cartan matrix $A$ is the associative algebra over $\Q(q)$
generated by ${\bse_i'},{\bsf_i}$ $(i \in I)$ satisfying the
following relations:
\begin{enumerate}
  \item ${\bse_i'} {\bsf_j} = q_i^{-a_{ij}}{\bsf_j}{\bse_i'} + \delta_{ij}$,
  \item $\displaystyle \sum_{r=0}^{1-a_{ij}} (-1)^{r} \left[\begin{matrix}1-a_{ij} \\ r\\ \end{matrix} \right]_i
        {\bse_i'}^{1-a_{ij}-r}{\bse_j'}{\bse_i'}^{r}=0$ $\quad$ if $i \in \Ire$, $i \neq j$,
  \item $\displaystyle \sum_{r=0}^{1-a_{ij}} (-1)^{r} \left[\begin{matrix}1-a_{ij} \\ r\\ \end{matrix} \right]_i
        {\bsf_i}^{1-a_{ij}-r}{\bsf_j}{\bsf_i}^{r}=0$ $\quad$ if $i \in \Ire$, $i \neq j$,
  \item ${\bse_i'}{\bse_j'}-{\bse_j'}{\bse_i'}=0$, ${\bsf_i}{\bsf_j}-{\bsf_j}{\bsf_i}=0$ $\quad$ if $a_{ij}=0$.
\end{enumerate}
\end{Def}

The algebra $U^{-}_q(\g)$ has a $B_q(\g)$-module structure from the equation
$\eqref{eq: special commute}$ (\cite{JKK05,Kash91}).

\begin{Prop} \ \label{Prop:Highest vector 1}
\begin{enumerate}
\item If $x \in U^{-}_q(\g)$ and $\bse_i' x=0$ for all $i\in I$, then $x$ is a constant multiple of ${\bf 1}$.
\item $U^{-}_q(\g)$ is a simple $B_q(\g)$-module.
\end{enumerate}
\end{Prop}

\begin{proof}
The proof is almost the same as in \cite[Lemma 3.4.7, Corollary 3.4.9]{Kash91}.
\end{proof}

Consider the anti-automorphism $\varphi$ on $B_q(\g)$ defined by
$$ \varphi(\bse_i') = \bsf_i \ \text{ and } \ \varphi(\bsf_i) = \bse_i' .$$
We define the symmetric bilinear forms $( \ , \ )_K$ and $( \ , \
)_L$ on $U_q^{-}(\g)$ as follows (cf. \cite[Propostion
3.4.4]{Kash91}, \cite[Chapter 1]{Lus93}):
\begin{equation} \label{Eq:def of ()K and ()L}
\begin{aligned}
&  ( {\bf 1},{\bf 1} )_K=1,\ \ (b x,y)_K=(x, \varphi(b) y)_K, \\
&  ( {\bf 1},{\bf 1} )_L=1, \ \ ( f_i,f_j )_L =\delta_{ij}(1-q_i^{2})^{-1},
   \ \ (x,yz)_L=(\Delta_0(x), y \otimes z)_L
\end{aligned}
\end{equation}
for $ x,y,z \in U_q^{-}(\g)$ and $b \in B_q(\g)$.

\begin{Lem} \ \label{Lem:nondegenerate pairing in GKM}
\begin{enumerate}
\item The bilinear form $(\ ,\ )_K$ on $U_q^{-}(\g)$ is nondegenerate.
\item  For homogeneous elements $x \in U_q^{-}(\g)_{-\alpha}$ and
$y \in U_q^{-}(\g)_{-\beta}$, we have
$$ (x,y)_L = \prod_{i\in I} \dfrac{1}{(1-q_{i}^{2})^{k_i}} (x,y)_K,$$
where $\alpha=\sum_{i\in I}k_i\alpha_{i} \in Q^+$. Hence $(\ ,\ )_L$ is nondegenerate.
\item  For any $x, y \in U_q^{-}(\g)$, we have
$$(\bse'_i x ,y)_L = (1-q^2_i)(x, \bsf_i y)_L. $$
\end{enumerate}
\end{Lem}
\begin{proof}
The assertion (1) is proved in \cite{JKK05}.

It was shown in \cite[(2.4)]{SV01} that the bilinear form $(\ , \
)_K$ satisfies
$$ (x,yz)_K=\sum_n(x^{(1)}_{n},y )_K( x^{(2)}_{n},z)_K,$$
where $\Delta_0(x)=\sum_n x^{(1)}_{n} \otimes x^{(2)}_{n}$. Then the
assertion (2) can be proved by induction on $|\alpha|$.

To prove the assertion (3), without loss of generality, we may
assume that $x \in U_q^{-}(\g)_{-\alpha}$, where $\alpha = -\sum_{i}
k_i \alpha_i \in -Q^{+}$. Then by (2) and the definition of $(\ , \
)_K$, we have
\begin{align*}
(\bse'_i x,y)_L
& = \dfrac{1}{(1-q_i^{2})^{k_i -1}}  \prod_{j \neq i } \dfrac{1}{(1-q_j^{2})^{k_j}} (\bse'_i x,y)_K \\
&=  \dfrac{1-q^2_i}{(1-q_i^{2})^{k_i }}  \prod_{j \neq i } \dfrac{1}{(1-q_j^{2})^{k_j}} (x, \bsf_i y)_K \\
&= (1-q^2_i) (x,\bsf_i y)_L,
\end{align*}
which proves the assertion (3).
\end{proof}

We now briefly review the crystal basis theory of quantum
generalized Kac-Moody algebras which was
developed in \cite{JKK05,JKKS07}. 
For any homogeneous element $u \in U^{-}_q(\g)$, $u$ can be expressed uniquely as
\begin{align} \label{eq: lowerpart i-string decomposition}
u = \sum_{l \ge 0} f_i^{(l)}u_l,
\end{align}
where ${\bse_i'} u_l=0$ for every $l \ge 0$ and $u_l=0$ for $l \gg 0$.
We call it the {\it i-string decomposition} of $u$ in $U_q^{-}(\g)$.
We define the {\it lower Kashiwara operators} $\tilde{e}_i$, $\tilde{f}_i$ $(i \in I)$
of $U_q^{-}(\g)$ by
$$ \tilde{e}_i u = \sum_{k \ge 1}f_i^{(k-1)}u_k, \ \ \tilde{f}_i u = \sum_{k \ge 0}f_i^{(k+1)}u_k. $$\
\
Let $\A_0=\{f/g \in \Q(q) \mid f,g \in \Q[q],g(0) \neq 0 \}$.
\begin{Def}
A {\it lower crystal basis} of $U^{-}_q(\g)$ is a pair $(L,B)$
satisfying the following conditions:
\begin{enumerate}
\item $L$ is a free $\A_0$-module of $U^{-}_q(\g)$ such that $U^{-}_q(\g)=\Q(q) \otimes_{\A_0} L$ and
      $L = \bigoplus_{\alpha \in Q^+} L_{-\alpha}$, where $L_{-\alpha} := L \cap
      U^{-}_q(\g)_{-\alpha}$,
\item $B$ is a $\Q$-basis of $L/ q L$ such that
$B = \bigsqcup_{\alpha \in Q^+} B_{-\alpha}$, where $B_{-\alpha} := B \cap (L_{-\alpha}/ q
L_{-\alpha})$,
\item $\tilde{e}_i B \subset B \sqcup \{0\}, \  \tilde{f}_i B \subset B $ for all $i \in
I$,
\item For $ b,b'\in B$ and $i \in I,$ $ b' = \tilde{f}_i b$ if and only if $b = \tilde{e}_i b'$.
\end{enumerate}

\end{Def}

 \begin{Prop} \cite[Theorem 7.1]{JKK05} \label{Prop: crystal bases of lowerpart }
Let $L(\infty)$ be the free $\A_0$-module of $U_q^-(\g)$
generated by $\{\tilde{f}_{i_1} \cdots \tilde{f}_{i_r} {\bf 1} \mid
r \ge 0 , i_k \in I\}$ and let
$$B(\infty) = \{\tilde{f}_{i_1} \cdots \tilde{f}_{i_r}{\bf 1}+q L(\infty)
\mid r \ge 0 , i_k \in I\}\setminus\{0\}.$$
Then the pair $(L(\infty),B(\infty))$ is a unique lower crystal basis of $U^{-}_q(\g)$.
\end{Prop}

Let $\mathcal{O}_{int}$ be the abelian category of $U_q(\g)$-modules
defined in \cite[Definition 3.1]{JKK05}. For each $\lambda \in P^+$,
let $V(\lambda)$ denote the irreducible highest weight
$U_q(\g)$-module with highest weight $\lambda$. It is generated by a
unique highest weight vector $v_{\lambda}$ with defining relations:
\begin{equation} \label{eq:hw module}
\begin{aligned}
& q^h v_{\lambda} = q^{\langle h, \lambda \rangle} v_{\lambda} \ \
\text{for all} \ h \in P^{\vee}, \\
& e_i v_{\lambda} = 0 \ \ \text{for all} \ i \in I, \\
& f_i^{\langle h_i, \lambda \rangle + 1} v_{\lambda} =0 \ \
\text{for} \ i \in \Ire, \\
& f_i v_{\lambda} = 0 \ \ \text{for} \ i\in \Iim \ \text{with} \
\langle h_i, \lambda \rangle =0.
\end{aligned}
\end{equation}
It was proved in \cite[Theorem 3.7]{JKK05} that the category
$\mathcal{O}_{int}$ is semisimple and that all the irreducible
objects have the form $V(\lambda)$ for $\lambda \in P^{+}$.

Let $M$ be a $U_q(\g)$-module in the category $\mathcal{O}_{int}$. For any $i \in I$ and $u \in M_{\mu}$,
the element $u$ can be expressed uniquely as
$$
u = \sum_{k \ge 0} f_i^{(k)}u_k,
$$
where $u_k \in M_{\mu + k \alpha_i}$ and $e_i u_k=0$. We call it the {\it i-string decomposition} of $u$.
We define
the {\it lower Kashiwara operators} $\tilde{e}_i,\tilde{f}_i \ (i \in I)$ by
$$ \tilde{e}_i u = \sum_{k \ge 1}f_i^{(k-1)}u_k, \ \ \tilde{f}_i u = \sum_{k \ge 0}f_i^{(k+1)}u_k. $$

\begin{Def}
A {\it lower crystal basis} of $U_q(\g)$-module $M$ is a pair
$(L,B)$ satisfying the following conditions:
\begin{enumerate}
\item $L$ is a free $\A_0$-module of $M$ such that $M = \Q(q) \otimes_{\A_0} L$
and $L = \bigoplus_{\lambda \in P}L_{\lambda}$, where $L_{\lambda}
:= L \cap M_\lambda$,
\item $B$ is $\Q$-basis of $L/qL$ such that $B = \bigsqcup_{\lambda \in P} B_{\lambda}$,
where $B_{\lambda} := B \cap
L_{\lambda}/qL_{\lambda}$,
\item $\tilde{e}_i B \subset B \sqcup \{ 0 \} $, \ $\tilde{f}_i B \subset B \sqcup \{ 0 \} $ for all $i \in I$,
\item For $b,b' \in B$ and $i\in I$, $b'=\tilde{f}_i b$ if and only if $b= \tilde{e}_i b'$.
\end{enumerate}
\end{Def}

\begin{Prop} \cite[Theorem 7.1]{JKK05} \label{Prop: crystal bases of integrable module }
For $\lambda \in P^+$, let $L(\lambda)$ be the free $\A_0$-module
of $V(\lambda)$ generated by $\{ \tilde{f_{i_1}} \cdots
\tilde{f_{i_r}}v_{\lambda} \mid r \ge 0 , i_k \in I\}$ and let
$$B(\lambda) = \{\tilde{f}_{i_1} \cdots \tilde{f}_{i_r}v_{\lambda}+qL(\lambda) \mid r \ge 0 , i_k \in I\}\setminus\{0\}.$$
Then the pair $(L(\lambda),B(\lambda))$ is a unique lower crystal basis of $V(\lambda)$.
\end{Prop}

\vskip 3em

\section{Khovanov-Lauda-Rouquier algebra $R$} \label{Sec:KLR}

In this section, we construct the Khovanov-Lauda-Rouquier algebra $R$ associated with a Borcherds-Cartan matrix $A$,
and investigate its algebraic structure and representation theory.

\subsection{The algebras $R(\alpha)$}\

Let $\F$ be a field. For $\alpha \in Q^+$ with
$|\alpha|=d$, set
\begin{align*}
\seq(\alpha) &= \{ \mathbf{i}=(i_1 \ldots i_d) \in I^d \mid \alpha_{i_1} + \cdots + \alpha_{i_d} = \alpha \},\\
\seqd(\alpha) &=\{ \mathbf{i}=(i_1^{(d_1)} \ldots i_r^{(d_r)}) \in I^d \mid d_1\alpha_{i_1} + \cdots + d_r\alpha_{i_r} = \alpha \}.
\end{align*}
Then the symmetric group $\sg_d = \langle r_i \mid i =1, \ldots d-1
\rangle$ acts naturally on $\seq(\alpha)$. For
$\mathbf{i}=(i_1\ldots i_d) \in \seq(\alpha),\ \mathbf{j}=(j_1\ldots
{j_{d'}} ) \in \seq(\beta)$, we denote by $\mathbf{i} * \mathbf{j}$
the concatenation of $\mathbf{i}$ and $\mathbf{j}$:
$$ \mathbf{i} * \mathbf{j} := (i_1\ldots i_d j_1\ldots {j_{d'}}) \in \seq(\alpha+ \beta). $$
The symmetric group  $S_d$ acts on the polynomial ring
$\F[x_1,\ldots,x_d]$ by
\begin{align*}
w \cdot f(x_1,\ldots,x_d) = f(x_{w(1)}, \ldots,x_{w(d)})
\quad\text{for $w\in S_d$ and $f(x_1,\ldots,x_d) \in \F[x_1,\ldots,x_d]$.}
\end{align*}
For $t=1,\ldots, d-1$, define the
operator $\partial_t$ on $\F[x_1,\ldots,x_d]$ by
$$ \partial_t(f) = \frac{r_t f - f}{ x_{t} - x_{t+1} } $$
for $f \in \F[x_1,\ldots,x_d]$.
We take a matrix $( \mathcal{Q}_{i,j}(u,v) )_{i,j\in I}$ in $\F[u,v]$ such that $Q_{i,j}(u,v) = Q_{j,i}(v,u)$ and $Q_{i,j}(u,v)$ has the form
$$ Q_{i,j}(u,v) = \left\{
                    \begin{array}{ll} \displaystyle
                      \sum_{p,q} t_{i,j;p,q}u^pv^q & \hbox{ if } i\ne j, \\
                      0 & \hbox{ if } i=j,
                    \end{array}
                  \right.
 $$
where the summation is taken over all $p,q\in \Z_{\ge0}$ such that $ (\alpha_i|\alpha_j)+s_ip+s_jq=0$ and $t_{i,j;p,q} \in \F$. In particular, $t_{i,j;-a_{ij},0} \in \F^{\times}$.
For each $i\in I$, choose a nonzero polynomial
$\mathcal{P}_i(u,v) \in \F[u,v]$ having the form
$$ \mathcal{P}_i(u,v) = \sum_{p,q } h_{i;p,q} u^pv^q , $$
where the summation is taken over all $p, q \in \Z_{\ge0}$ such that $2-a_{ii}-2p-2q=0$ and $h_{i;p,q} \in \F$. In particular,
$h_{i;1-\frac{a_{ii}}{2},0}, h_{i;0,1-\frac{a_{ii}}{2}} \in \F^\times$.

\begin{Def} \label{def:KLR}
Let $(A,P,\Pi,\Pi^\vee)$ be a Borcherds-Cartan datum.
For $\alpha\in Q^+$ with height $d$,
the {\em Khovanov-Lauda-Rouquier algebra $R(\alpha)$} of weight $\alpha$
 associated with the data $(A,P,\Pi,\Pi^\vee)$,
$(\mathcal{P}_i)_{i\in I}$ and $(\mathcal{Q}_{i,j})_{i,j\in I}$
is the associative graded $\F$-algebra
generated by $1_{\mathbf{i}}\ (\mathbf{i}\in \seq(\alpha))$, $x_k\
(1 \le k \le d)$, $\tau_t\ (1 \le t \le d-1)$  satisfying the
following defining relations:
\begin{equation} \label{Eq:def rel 1}
\begin{aligned}
& 1_{\mathbf{i}} 1_{\mathbf{j}} = \delta_{\mathbf{i},\mathbf{j}} 1_{\mathbf{i}},\ \sum_{\mathbf{i} \in \seq(\alpha)} 1_{\mathbf{i}}=1,\
x_k 1_{\mathbf{i}} =  1_{\mathbf{i}} x_k, \  x_k x_l = x_l x_k,\\
& \tau_t 1_{\mathbf{i}} = 1_{r_t( \mathbf{i})} \tau_t,\  \tau_t \tau_s = \tau_s \tau_t \text{ if } |t - s| > 1, \\
&  \tau_t^2 1_{\mathbf{i}} = \left\{
                                                \begin{array}{ll}
                                                  \partial_t\mathcal{P}_{i_t}(x_t,x_{t+1}) \tau_t 1_{\mathbf{i}} & \hbox{ if }  i_t = i_{t+1}, \\
                                                   \mathcal{Q}_{i_t, i_{t+1}}(x_t, x_{t+1}) 1_{\mathbf{i}} & \hbox{ if } i_t \ne i_{t+1},
                                                \end{array}
                                              \right. \\
&  (\tau_t x_k - x_{r_t(k)} \tau_t ) 1_{\mathbf{i}} = \left\{
                                                           \begin{array}{ll}
                                                             -  \mathcal{P}_{i_t }(x_t, x_{t+1}) 1_{\mathbf{i}} & \hbox{if } k=t \text{ and } i_t = i_{t+1}, \\
                                                               \mathcal{P}_{i_t }(x_t, x_{t+1}) 1_{\mathbf{i}} & \hbox{if } k = t+1 \text{ and } i_t = i_{t+1},  \\
                                                             0 & \hbox{otherwise,}
                                                           \end{array}
                                                         \right.
\end{aligned}
\end{equation}
\begin{equation}
\begin{aligned}  \label{Eq:def rel 2}
&( \tau_{t+1} \tau_{t} \tau_{t+1} - \tau_{t} \tau_{t+1} \tau_{t} )  1_{\mathbf{i}} \\
& \qquad \qquad = \left\{
                                                                                   \begin{array}{ll}
\mathcal{P}_{i_t }(x_t, x_{t+2})
\overline{\mathcal{Q}}_{i_t,i_{t+1}}(x_t, x_{t+1},
x_{t+2})1_{\mathbf{i}} & \hbox{if } i_t = i_{t+2} \ne i_{t+1}, \\
\overline{\mathcal{P}}_{i_t}'( x_{t}, x_{t+1}, x_{t+2})
\tau_{t}1_{\mathbf{i}} +
\overline{\mathcal{P}}_{i_t}''( x_{t}, x_{t+1}, x_{t+2}) \tau_{t+1}1_{\mathbf{i}} & \hbox{if } i_t = i_{t+1} = i_{t+2},\\
0 & \hbox{otherwise},
\end{array}
\right.
\end{aligned}
\end{equation}
where

\begin{equation}
\begin{aligned}
\overline{\mathcal{P}}'_i(u,v,w) &:=
\frac{\mathcal{P}_{i}(v,u)\mathcal{P}_{i}(u,w)}{(u-v)(u-w)}
+\frac{\mathcal{P}_{i}(u,w)\mathcal{P}_{i}(v,w)}{(u-w)(v-w)}
-\frac{\mathcal{P}_{i}(u,v)\mathcal{P}_{i}(v,w)}{(u-v)(v-w)}, \\
\overline{\mathcal{P}}''_i(u,v,w) &:=
- \frac{\mathcal{P}_{i}(u,v)\mathcal{P}_{i}(u,w)}{(u-v)(u-w)}
- \frac{\mathcal{P}_{i}(u,w)\mathcal{P}_{i}(w,v)}{(u-w)(v-w)}
+ \frac{\mathcal{P}_{i}(u,v)\mathcal{P}_{i}(v,w)}{(u-v)(v-w)}, \\
\overline{\mathcal{Q}}_{i,j}(u,v,w)
& := \frac{\mathcal{Q}_{i,j}(u,v) -
\mathcal{Q}_{i,j}(w,v)}{u-w}.
\end{aligned}
\end{equation}
\end{Def}

Let $R: = \bigoplus_{\alpha \in Q^{+}} R(\alpha)$.
The $\Z$-grading on
$R(\alpha)$ is given by
\begin{align} \label{Eq:degree}
\deg(1_\mathbf{i})=0, \quad \deg(x_k 1_\mathbf{i})= 2s_{i_k}, \quad  \deg(\tau_t 1_\mathbf{i})= -(\alpha_{i_{t}} | \alpha_{i_{t+1}}).
\end{align}
Note that $\overline{\mathcal{P}}_i'$, $\overline{\mathcal{P}}_i''$ and
$\overline{\mathcal{Q}}_{i,j}$ are polynomials. If $i \in
\Ire$, then $\mathcal{P}_{i}(u, v)$ is a nonzero constant, which
will be normalized to be 1 in this paper.  If $I$ is finite and
$a_{ii}=2$ for all $i\in I$, then the algebra $R$ coincides with the
Khovanov-Lauda-Rouquier algebra introduced in \cite{KL09, KL11,
R08}.

The algebra $R$ can be defined by using planar diagrams with dots
and strands. For simplicity,
we assume that $\mathcal{P}_i$ are symmetric and $t_{i,j;-a_{ij},0} = t_{i,j;0, -a_{ji}} = 1$ and
$t_{i,j;p,q} = 0$ for other $p,q$. Note that $\partial_t \mathcal{P}_{i_t}(x_t,x_{t+1}) = 0$.
We denote by $R$ the
$\F$-vector space spanned by braid-like diagrams, considered up to
planar isotropy, such that all strands are colored by $I$ and can
carry dots. The multiplication $D \cdot D'$ of two diagrams $D$ and
$D'$ is given by stacking of the diagram $D$ on the diagram $D'$ if
the color on the top of $D'$ matches with the color at the bottom of
$D$ and defined to be $0$ otherwise. It is obvious that the
following elements are generators of $R(\alpha)$ $(\alpha \in Q^{+},
\mathbf{i}= (i_1 \ldots i_d) \in \seq(\alpha))$:
\begin{align*}
1_\mathbf{i} \  := \  \genOne{i_1}{i_k}{i_d} , \quad x_k
1_\mathbf{i} \  := \  \genX{i_1}{i_k}{i_d} , \quad \tau_t
1_\mathbf{i} \  := \  \genTau{i_1}{i_t}{i_{t+1}}{i_d}.
\end{align*}
The local relations are given as follows:

\begin{align} \label{Eq:local rel 1}
 \dCross{i}{j} \quad =  \quad \begin{cases} \quad \quad \quad \quad \quad \  0 & \text{ if } i= j, \\
                    \quad  \quad \quad \quad \ \ \twoStrands{i}{j} & \text{ if } (\alpha_i|\alpha_j)=0, \\
                    \quad  \twoDotStrandsL{-a_{ij}}{i}{j} \ + \ \twoDotStrandsR{-a_{ji}}{i}{j}  & \text{ if } (\alpha_i|\alpha_j)\neq 0,
                  \end{cases}
\end{align}

\begin{equation} \label{Eq:local rel 2}
\begin{aligned}
& \CrossDR{{}}{i}{j} \ - \ \CrossUL{{}}{i}{j}\quad = \quad \begin{cases} \quad \mathcal{P}_i(x,y) \cdot \twoStrands{i}{i} & \text{ if } i = j,  \\
                                         \quad \quad \quad 0 & \text{ otherwise,  } \end{cases} \\
& \CrossUR{{}}{i}{j} \ - \ \CrossDL{{}}{i}{j}\quad = \quad \begin{cases} \quad \mathcal{P}_i(x,y) \cdot \twoStrands{i}{i} & \text{ if } i = j,  \\
                                         \quad \quad \quad 0 & \text{ otherwise,  } \end{cases} \\
&\  (  \text{ here, } x:=\smalltwoDotStrandsL{i}{i}\ \text{ and } \ y:=\smalltwoDotStrandsR{i}{i}\  )
\end{aligned}
\end{equation}

\begin{equation} \label{Eq:local rel 3}
\begin{aligned}
\quad & \BraidR{i}{j}{k} \  -\ \BraidL{i}{j}{k}\  =\ \left\{
                                          \begin{array}{ll}
                                            \mathcal{P}_i(x,z){\displaystyle \sum_{s=0}^{-a_{ij}-1}} \threeDotStrands{s}{-a_{ij}-1-s}{i}{j}{i} & \hbox{ if } i=k \ne j, a_{ij} \ne 0, \\
                                            \overline{\mathcal{P}}'_i(x,y,z) \left( \ \CrossL{i}{i}{i}\  -\  \CrossR{i}{i}{i} \ \right)  & \hbox{ if } i=j=k, \\
                                            0 & \hbox{otherwise.}
                                          \end{array}
                                        \right. \\
 &\  (  \text{ here, } x:=\smallthreeDotStrandsL{i}{j}{k} , \  \ y:=\smallthreeDotStrandsM{i}{j}{k}\ \text{ and }\  z:=\smallthreeDotStrandsR{i}{j}{k}\ )
\end{aligned}
\end{equation}

For $\mathbf{t}=(t_1 \ldots t_d) \in \Z_{\ge0}^d$ and a reduced expression $ w = r_{i_1}\cdots r_{i_t}  \in \sg_d$, set
$$  x^{\mathbf{t}} = x_{1}^{t_1}\cdots x_{d}^{t_d}   \  \text{ and } \ \tau_{w} = \tau_{i_1}\cdots\tau_{i_t}.$$
It follows from the defining relations that
$$ \{ \tau_{w} x^{\mathbf{t}} 1_\mathbf{i} \mid
\mathbf{t} \in \Z_{\ge0}^d,\ \mathbf{i} \in \seq(\alpha),
\ w: \text{reduced in } \sg_d  \} $$ is a spanning set of
$R(\alpha)$.

Consider the graded anti-involution $\psi: R(\alpha) \rightarrow
R(\alpha)$ which is the identity on generators. For a graded left
$R(\alpha)$-module $M$, let $M^\star$ be the graded right
$R(\alpha)$-module whose underlying space is $M$ with
$R(\alpha)$-action given by
$$ v \cdot r = \psi(r)v \quad \text{ for } v \in M^{\star}, \ r \in R(\alpha). $$

We will investigate the structure of $R(m \alpha_i)$ $(m \ge 0)$ in
more detail. If $a_{ii}=2$, then the defining relations for $R(m
\alpha_i)$ reduce to
\begin{align*}
& x_k x_l = x_l x_k,  \ \ \tau_t^2 = 0 , \\
& \tau_{t}\tau_{t+1}\tau_{t} = \tau_{t+1}\tau_{t}\tau_{t+1},\ \ \tau_t \tau_s = \tau_s \tau_t\  \text{ if } |t-s|>1,  \\
& \tau_{t} x_{t} = x_{t+1} \tau_{t}-1, \ \ \tau_{t} x_{t+1} = x_{t}
\tau_{t} +1, \\
& \tau_t x_k = x_k \tau_t \ \ \text{if} \ k \neq t, t+1.
\end{align*}
Hence the algebra $R(m \alpha_i)$ is isomorphic to the {\it nil
Hecke algebra} $NH_m$, which is the associative algebra generated by
$\mathbf{x}_k\ (1\le k \le m)$ and $\partial_t\ (1 \le t \le m-1)$
satisfying the following relations:
\begin{align*}
&\mathbf{x}_k \mathbf{x}_l = \mathbf{x}_l \mathbf{x}_k, \ \ \partial_t^2 = 0 , \\
& \partial_{t}\partial_{t+1}\partial_{t} = \partial_{t+1}\partial_{t}\partial_{t+1},\ \
\partial_t \partial_s = \partial_s \partial_t\  \text{ if } |t-s|>1,  \\
& \partial_t \mathbf{x}_t = \mathbf{x}_{t+1} \partial_t -1, \ \
\partial_t \mathbf{x}_{t+1}=  \mathbf{x}_{t}  \partial_t + 1, \\
& \partial_t \mathbf{x}_k = \mathbf{x}_k \partial_t \ \ \text{if} \
k \neq t, t+1.
\end{align*}
Therefore, as was shown in \cite{KL09}, the algebra $R(m \alpha_i)$
has a primitive idempotent $\tau_{w_0}x_1^{m-1} \cdots
x_{m-2}^2x_{m-1}$, where $w_{0}$ is the longest element in $\sg_m$,
and has a unique (up to isomorphism and degree shift) irreducible
module $L(i^m)$. The irreducible module $L(i^m)$ is isomorphic to
the one induced from the trivial $\F[x_1,\ldots, x_m]$-module of dimension 1 over $\bR$.

If $a_{ii}<0$, then $\mathcal{P}_i(u,v)$ is a homogeneous
polynomial with degree $1-\dfrac{a_{ii}}{2} > 1$, and
$\overline{\mathcal{P}}'_i(u,v,w)$ and $\overline{\mathcal{P}}''_i(u,v,w)$ have positive degree. By
\eqref{Eq:degree}, $R(m \alpha_i)$ has positive grading and hence
it has a unique idempotent $1_{(i\ldots i)}$. Thus there exists a unique irreducible $R(m \alpha_i)$-module $L(i^m)=\F v$
defined by
\begin{align} \label{Eq:def of L in Iim}
 1_{(i\ldots i)} \cdot v = v, \ \  x_k \cdot v = 0, \ \ \tau_t \cdot v = 0 .
\end{align}

If $a_{ii}=0$, then in general, $R(m\alpha_i)$ has many primitive idempotents, which means
that there are many irreducible $R(m\alpha_i)$-modules.
For example,
if $m=3$ and $\mathcal{P}_i(u,v) = u-v$, then $\tau_1\tau_2, \tau_2\tau_1$ and $1 -
\tau_1\tau_2  - \tau_2\tau_1$ are orthogonal primitive idempotents.
The algebra $R(m \alpha_i)$ itself, not principal
indecomposable modules, will serve as one of the projective modules
that give our categorification. The whole Grothendieck group of
the category of finitely generated projective $R(m \alpha_i)$-modules seems rather large and nontrivial.
We hope to investigate it in a later work.

We now construct a faithful polynomial representation of
$R(\alpha)$. First, we define an $R(m\alpha_i)$-module structure on
$\F[x_1, \ldots, x_m]$ by
\begin{align*}
x_k \cdot f(x_1, \ldots, x_m) &= x_k f(x_1, \ldots, x_m),   \\
\tau_t \cdot f(x_1, \ldots, x_m) &= \mathcal{P}_{i}(x_t, x_{t+1}) \partial_t ( f(x_1, \ldots, x_m))
\end{align*}
for $x_k, \tau_t \in R(m\alpha_i), f(x_1, \ldots, x_m) \in \F[x_1, \ldots, x_m]$.

\begin{Lem} \label{Lem:R(m_i) for i in Iim}
$\F[x_1, \ldots, x_m]$ is a faithful representation of $R(m\alpha_i)$.
\end{Lem}
\begin{proof}

If $i \in \Ire$, our assertion was shown in \cite[Example
2.2]{KL09}. Assume that $i\in \Iim$ and let $\mathbf{x}_k$ be the
endomorphism of $\F[x_1, \ldots, x_m]$ defined by
\begin{align*}
\mathbf{x}_k ( f(x_1, \ldots, x_m)) = x_k f(x_1, \ldots, x_m)
\end{align*}
for $f(x_1, \ldots, x_m) \in \F[x_1, \ldots, x_m]$.
Note that
$$\{ \partial_{j_1} \cdots \partial_{j_k} \mathbf{x}^\mathbf{t} \mid \mathbf{t} \in \Z_{\ge0}^m,
\ r_{j_1} \cdots r_{j_k}  \text{is a reduced expression in }  \sg_m
\, (k \ge 0)   \}$$ is a linearly independent subset of
$\End(\F[x_1, \ldots, x_m])$. Let
$$\iota: R(m \alpha_i) \longrightarrow \End(\F[x_1, \ldots, x_m])$$
be the map defined by $\iota(x_k) = \mathbf{x}_k$ and $\iota(\tau_t) = \mathcal{P}_i(  \mathbf{x}_t, \mathbf{x}_{t+1}) \cdot \partial_t$.

We first show that $\iota$ is well-defined. Since
$\mathcal{P}_i(u,v)$ is a homogeneous polynomial, it is
easy to verify that the relations $\eqref{Eq:def rel 1}$ hold. To
check the relations in  \eqref{Eq:def rel 2}, for simplicity, we
assume that $m=3$ and let $x = x_1, y = x_2, z = x_3$,
$\mathcal{P}(u,v) = \mathcal{P}_i(u,v)$.

Set
$$ \mathsf{P}(u,v) = \frac{\mathcal{P}(u,v)}{u-v}. $$
By a direct computation, we have
\begin{align*}
\iota(\tau_2 \tau_1 \tau_2) & = \mathsf{P}(x,y)\mathsf{P}(y,z)\mathsf{P}(x,z)(r_2r_1r_2-r_2r_1-r_1r_2+r_1)\\
& \ - \mathsf{P}(y,z)\mathsf{P}(z,y)\mathsf{P}(x,z)(1-r_2)+\mathsf{P}(x,y)\mathsf{P}(y,z)^2(r_2-1), \\
\iota(\tau_1 \tau_2 \tau_1) & = \mathsf{P}(x,y)\mathsf{P}(y,z)\mathsf{P}(x,z)(r_1r_2r_1-r_2r_1-r_1r_2+r_2)\\
& \ - \mathsf{P}(x,y)\mathsf{P}(y,x)\mathsf{P}(x,z)(1-r_1)+\mathsf{P}(x,y)^2\mathsf{P}(y,z)(r_1-1).
\end{align*}
As $\iota(\tau_k) = \mathsf{P}(x_k,x_{k+1}) (r_k-1) $ for $k=1,2,$
\begin{align*}
\iota(\tau_2 \tau_1 \tau_2) - \iota(\tau_1 \tau_2 \tau_1) &=
(-\mathsf{P}(y,x)\mathsf{P}(x,z)+\mathsf{P}(y,z)\mathsf{P}(x,z)-\mathsf{P}(x,y)\mathsf{P}(y,z))\iota(\tau_1) \\
& \quad
+ ( \mathsf{P}(x,y)\mathsf{P}(y,z) + \mathsf{P}(z,y)\mathsf{P}(x,z)-\mathsf{P}(x,y)\mathsf{P}(x,z))\iota(\tau_2),
\end{align*}
which shows that the relation $\eqref{Eq:def rel 2}$ holds.
It remains to show that $\iota$ is injective. Take a nonzero element
$$y = \tau_{w_1}f_1 + \cdots + \tau_{w_t}f_t \quad ( 0\ne f_k \in \F[x_1, \ldots, x_m],\ w_k \ \text{is a reduced expression in } \sg_m )$$
of $R(m\alpha_i)$ such that $w_i \ne w_j \text{ if } i \ne j$ and $
\ell( {w}_1 ) \ge \ell( {w}_k ) $ for $0 \le k \le t$. Write the reduced expression of $w_1$ as $w_1 =
r_{i_1}\cdots r_{i_l}$. Then, $\iota(y)$ can
be written as
$$ \iota(y) = \partial_{i_1}\cdots \partial_{i_l} f' + \cdots \text{lower terms} \cdots $$
for some nonzero polynomial $f'$, which implies that $\iota(y)$ is nonzero. Therefore $\iota$ is injective.
\end{proof}

Now we consider the general case $R(\alpha)$ with $\alpha \in Q^+$.
Take a total order $\prec$ on $I$. Let
$$ \mathfrak{Pol}(\alpha) = \bigoplus_{\mathbf{i} \in \seq(\alpha)} \F[x_1(\mathbf{i}), \ldots, x_d(\mathbf{i})].  $$
For any polynomial $f \in \F[u_1, \ldots, u_d]$, let $f(\mathbf{i})$
be the polynomial in $\F[x_1(\mathbf{i}), \ldots, x_d(\mathbf{i})]$
obtained from $f$ by replacing $u_k$ by $x_k(\mathbf{i})$. We define
an $R(\alpha)$-module structure on $\mathfrak{Pol}(\alpha)$ as
follows: for $\mathbf{i} \in \seq(\alpha)$ and $f \in \F[u_1,
\ldots, u_d]$, we define
\begin{equation} \label{Eq:def of faithful rep}
\begin{aligned}
1_\mathbf{j} \cdot f(\mathbf{i}) &= \delta_{\mathbf{ij}} f(\mathbf{i}) \quad \ (\ \mathbf{j} \in \seq(\alpha)\ ), \\
x_k \cdot f(\mathbf{i}) &= x_k(\mathbf{i})f(\mathbf{i}), \\
\tau_t \cdot f(\mathbf{i}) &=  \left\{
                      \begin{array}{ll}
                        \mathcal{P}_{i_t}(x_t(r_t \mathbf{i}), x_{t+1}(r_t \mathbf{i})) \partial_t f(r_t \mathbf{i})  & \hbox{ if } i_t = i_{t+1}, \\
                        \mathcal{Q}_{i_{t+1}, i_{t}}(x_t(r_t\mathbf{i}), x_{t+1}(r_t\mathbf{i})) r_t f(r_t \mathbf{i}) & \hbox{ if } i_t \ne i_{t+1},\ i_t \succ i_{t+1},  \\
                        r_t f(r_t \mathbf{i}) & \hbox{ if } i_t \ne i_{t+1},\ i_t \prec i_{t+1}.
                      \end{array}
                    \right.
\end{aligned}
\end{equation}

\begin{Lem} \label{Lem:faithful repn}
$\mathfrak{Pol}(\alpha)$ is a well-defined $R(\alpha)$-module.
\end{Lem}
\begin{proof}
We verify the defining relations of $R(\alpha)$. The relations $\eqref{Eq:def rel
1}$ can be verified in a straightforward manner. In the proof of
Lemma \ref{Lem:R(m_i) for i in Iim}, we already proved our assertion
when $i_t = i_{t+1} = i_{t+2}$. Thus it suffices to consider the
following three cases in $\eqref{Eq:def rel 2}$: (i) $i_t=i_{t+2}
\ne i_{t+1}$, \ (ii) $i_t, i_{t+1},i_{t+2}$ are distinct, \ (iii)
$i_t = i_{t+1} $ and $i_t \ne i_{t+2}$. For simplicity, let $d = 3$,
$\mathbf{i} = (i,j,k)$ and $f(u,v,w)= u^a v^b w^c $.
Set $x = x_1(\mathbf{i})$, $y = x_2(\mathbf{i})$ and $z = x_3(\mathbf{i})$.

Case (i): Let $\mathbf{i} = (i,j,i)$ with $i \ne j$. Without loss of
generality, we may assume $i \prec j$.
Then, by a direct
computation, we have
\begin{align*}
\tau_1 \tau_2 \tau_1 (x^a y^b z^c) \ &=\ \mathcal{P}_i(x, z) \mathcal{Q}_{ij}(x,y)  \frac{ x^c y^b z^a - x^a y^b z^c }{x-z}, \\
\tau_2 \tau_1 \tau_2 (x^a y^b z^c) \ &=\ \mathcal{P}_i(x, z) \frac{ \mathcal{Q}_{ij}(x,y)x^c y^b z^a - \mathcal{Q}_{ij}(z,y) x^a y^b z^c }{x-z},
\end{align*}
which yield
$$ (\tau_2 \tau_1 \tau_2 - \tau_1 \tau_2 \tau_1) (x^a y^b z^c) = \mathcal{P}_i(x, z) \frac{\mathcal{Q}_{ij}(x,y) - \mathcal{Q}_{ij}(z,y) }{x-z} x^ay^bz^c.  $$

Case (ii): Let $\mathbf{i} = (i,j,k)$ such that  $i, j,k$ are
distinct. Since the other cases are similar, we will only prove our
assertion when $i \succ j \succ k$. Then we have
\begin{align*}
\tau_1 \tau_2 \tau_1 (x^a y^b z^c) \ &=\ \mathcal{Q}_{ij}(y,z)\mathcal{Q}_{jk}(x,y) \mathcal{Q}_{ik}(x,z) x^cy^bz^a, \\
\tau_2 \tau_1 \tau_2 (x^a y^b z^c) \ &=\ \mathcal{Q}_{ij}(y,z)\mathcal{Q}_{jk}(x,y) \mathcal{Q}_{ik}(x,z) x^cy^bz^a,
\end{align*}
which implies that $ (\tau_2 \tau_1 \tau_2 - \tau_1 \tau_2 \tau_1) (x^a y^b z^c) = 0 $.

Case (iii): Similarly as above, we consider $\mathbf{i} = (i,i,j)$
with $i \succ j $ only.  Then
\begin{align*}
\tau_1 \tau_2 \tau_1 (x^a y^b z^c) \ &=\ \mathcal{Q}_{ij}(x,y) \mathcal{Q}_{ij}(x,z) \mathcal{P}_i(y,z) \frac{x^cy^bz^a - x^cy^az^b }{y-z}, \\
\tau_2 \tau_1 \tau_2 (x^a y^b z^c) \ &=\ \mathcal{Q}_{ij}(x,y) \mathcal{Q}_{ij}(x,z) \mathcal{P}_i(y,z) \frac{x^cy^bz^a - x^cy^az^b }{y-z}.
\end{align*}
Hence we have $ (\tau_2 \tau_1 \tau_2 - \tau_1 \tau_2 \tau_1) (x^a y^b z^c) = 0 $, which completes the proof.
\end{proof}

Note that $R(\alpha) =\bigoplus_{\mathbf{i},\mathbf{j} \in
\seq(\alpha)} {_\mathbf{j}}R(\alpha)_{\mathbf{i}}$, where
$_{\mathbf{j}}R(\alpha)_{\mathbf{i}} := 1_\mathbf{j} R(\alpha)
1_\mathbf{i}$.  Given each $w \in \sg_d$, fix a minimal
representative $\underline{w}$ of $w$. For $\mathbf{i},
\mathbf{j}\in \seq(\alpha)$, let
$$ _{\mathbf{j}}\sg_{\mathbf{i}} = \{ \underline{w} \mid w \in \sg_d,\ w(\mathbf{i})=\mathbf{j}  \}. $$
It follows from the defining relations that
$$  _{\mathbf{j}}B(\alpha)_{\mathbf{i}}:= \{ \tau_{\underline{w}} x^{\mathbf{t}} 1_\mathbf{i} \mid \mathbf{t} \in \Z_{\ge0}^d,\  \underline{w} \in
{_\mathbf{j} \sg_\mathbf{i}}  \} $$ is a spanning set of
$_\mathbf{j} R(\alpha)_\mathbf{i}$. Moreover, we have the following
proposition.

\begin{Prop} \label{Prop:basis of R(alpha)} \
\begin{enumerate}
\item The set $ _{\mathbf{j}}B(\alpha)_{\mathbf{i}}$ is a homogeneous basis of $_{\mathbf{j}} R(\alpha) _{\mathbf{i}}$.
\item $\mathfrak{Pol}(\alpha)$ is a faithful representation of $R(\alpha)$.
\end{enumerate}
\end{Prop}
\begin{proof}
 Let $<$ be the lexicographic order of $\seq(\alpha)$ arising from the order $\prec$ of $I$, and
let ${_{\mathbf{j}_2}}w_{\mathbf{j}_1}  $ be the minimal element in $ _{{\mathbf{j}_2}}\sg_{\mathbf{j}_1}$ for $\mathbf{j}_1$, $\mathbf{j}_2 \in \seq(\alpha)$.
Let
$$ \Upsilon : R(\alpha) \longrightarrow \End(\mathfrak{Pol}(\alpha))$$
be the algebra homomorphism given in $\eqref{Eq:def of faithful
rep}$. We will show that $ \Upsilon
({_{\mathbf{j}}B(\alpha)_{\mathbf{i}}})$ is linearly independent,
which would imply the set ${_{\mathbf{j}}B(\alpha)_{\mathbf{i}}}$ is
linearly independent. The injectivity of $\Upsilon$ would also
follow immediately. We prove our claim using induction on the
lexicographic order $<$ on $\seq(\alpha)$.

Let $\mathbf{i}\in \seq(\alpha)$, and let
$$\mathbf{j} = ( \underbrace{j_1 \ldots j_1 }_{d_1} \underbrace{j_2 \ldots j_2 }_{d_2} \cdots \underbrace{j_r \ldots j_r }_{d_r} ) \in \seq(\alpha)$$
such that $j_1 \succ j_2 \succ \cdots \succ j_r $. Note that
$\mathbf{j}$ is a maximal element in $\seq(\alpha)$.

Let $m$ be a linear combination of ${_{\mathbf{j}}B(\alpha)_{\mathbf{i}}}$ such that $ \Upsilon(m)=0 $.
Note that $m$ can be expressed as
$$ m = \sum_s \tau_{w_s}  \tau_{ {_\mathbf{j}}w_\mathbf{i}} x^{\mathbf{k}_s}1_{\mathbf{i}}  $$
for some $ \mathbf{k}_s \in \Z_{\ge0}^d$ and some $w_s \in \sg_{d_1} \times \cdots \times \sg_{d_r}$.
It follows from $\eqref{Eq:def of faithful rep}$ that $\Upsilon ( \tau_{ {_\mathbf{j}}w_\mathbf{i}} 1_{\mathbf{i}}  ) $ can be viewed as a linear map
from $\F[x_1(\mathbf{i}), \ldots, x_d(\mathbf{i})] $ to $ \F[x_1(\mathbf{j}), \ldots, x_d(\mathbf{j})] $
sending $1_{\mathbf{i}}$ to $1_{\mathbf{j}}$. Hence,
$$  \Upsilon(m) = 0 \quad \text{ if and only if }\quad  \Upsilon( \sum_s \tau_{w_s} x^{ {_\mathbf{j}}w_\mathbf{i} (\mathbf{k}_s) } 1_{\mathbf{j}} ) =0 .$$
Since $\Upsilon( \sum_s \tau_{w_s} x^{ {_\mathbf{j}}w_\mathbf{i} (\mathbf{k}_s) } 1_{\mathbf{j}}  )$ can be regarded as a linear map
in $\bigoplus_{k=1}^r \End( \F[x_1, \ldots, x_{d_k} ])$, by Lemma \ref{Lem:R(m_i) for i in Iim},
we have
$$ \Upsilon( \sum_s \tau_{w_s} x^{ {_\mathbf{j}}w_\mathbf{i} (\mathbf{k}_s) } 1_{\mathbf{j}}  ) =0  \quad \text{ if and only if }\quad
\sum_s \tau_{w_s} x^{ {_\mathbf{j}}w_\mathbf{i} (\mathbf{k}_s)   } 1_{\mathbf{j}} = 0, $$
which implies $ m = 0$. Therefore, $\Upsilon( {_{\mathbf{j}}B(\alpha)_{\mathbf{i}}} )$ is linearly independent.

We now consider the case when  $\mathbf{j} $ is an arbitrary
sequence in $ \seq(\alpha)$. This step can be proved by a similar
induction argument as in \cite[Theorem 2.5]{KL09}, which completes
the proof.
\end{proof}



For any $\alpha, \beta \in Q^+$, let
\begin{align*}
1_\alpha &= \sum_{\mathbf{i} \in \seq(\alpha)} 1_\mathbf{i}, \\
1_{\alpha, \beta} &= \sum_{\mathbf{i} \in \seq(\alpha),\ \mathbf{j} \in \seq(\beta)} 1_{\mathbf{i} * \mathbf{j}} .
\end{align*}
Then $1_{\alpha, \beta} R(\alpha+\beta)$ has a natural graded left
$R(\alpha) \otimes R(\beta)$-module structure.
\begin{Cor} \label{Cor:R(alpha+beta) is free}
$1_{\alpha, \beta} R(\alpha+\beta)$ is a free graded left $R(\alpha) \otimes R(\beta)$-module.
\end{Cor}
\begin{proof}
Let $d := | \alpha| $, $d' := |\beta|$, and $ \sg_d \times \sg_{d'} \backslash \sg_{d+d'} $ be the set of minimal right $\sg_{d} \times \sg_{d'}$-coset representatives of
$\sg_{d+d'}$. For $w \in \sg_d \times \sg_{d'} \backslash \sg_{d+d'} $,  set
$$ \hat{\tau}_{w} = \sum_{ \mathbf{i} \in \seq(\alpha),\ \mathbf{j} \in \seq(\beta)} 1_{\mathbf{i}*\mathbf{j}}\ \tau_{w}\  1_{w^{-1}( \mathbf{i}*\mathbf{j})}. $$
Then, it follows from Proposition \ref{Prop:basis of R(alpha)} that
$$ \{ \hat{\tau}_{w} \mid w\in \sg_d \times \sg_{d'} \backslash \sg_{d+d'}  \} $$
is a basis of $1_{\alpha, \beta} R(\alpha+\beta)$ as a left $R(\alpha) \otimes R(\beta)$-module.
\end{proof}

For a graded $R(\alpha)$-module $M=\bigoplus_{i\in \Z}M_i$, let
$M\langle k \rangle$ denote the graded $R(\alpha)$-module obtained
from $M$ by shifting the grading by $k$; i.e.,
$M \langle k \rangle :=
\bigoplus_{i\in \Z} M_{i+k}$.
Given $\alpha, \alpha', \beta, \beta' \in Q^+$ with $\alpha+\beta = \alpha' + \beta'$, let
$$ _{\alpha, \beta}R_{\alpha', \beta'} := 1_{\alpha, \beta} R(\alpha+\beta) 1_{\alpha', \beta'} . $$
We write $_{\alpha}R_{\alpha', \beta'}$ (resp.\ $_{\alpha,
\beta}R_{\alpha'}$) for $_{\alpha, \beta}R_{\alpha', \beta'}$ if
$\beta = 0$ (resp.\ $\beta' = 0$). Note that $_{\alpha,
\beta}R_{\alpha', \beta'}$ is a graded $(R(\alpha)\otimes R(\beta),
R(\alpha')\otimes R(\beta'))$-bimodule. Now we obtain the Mackey's
Theorem for Khovanov-Lauda-Rouquier algebras.

\begin{Prop} \label{Prop:Mackey}
The graded $(R(\alpha)\otimes R(\beta), R(\alpha')\otimes R(\beta'))$-bimodule $_{\alpha,\beta}R_{\alpha', \beta'}$ has
a graded filtration with graded subquotients isomorphic to
$$ {_{\alpha}R_{\alpha-\gamma, \gamma}} \otimes {_{\beta}R_{\beta+\gamma-\beta', \beta'-\gamma}} \otimes_{R'}
{_{\alpha-\gamma, \alpha'+\gamma-\alpha}R_{\alpha'}} \otimes
{_{\gamma, \beta-\gamma}R_{\beta}}  \langle (\gamma | \beta + \gamma - \beta') \rangle , $$
where $R' = R(\alpha-\gamma)\otimes
R(\gamma)\otimes R(\beta + \gamma - \beta') \otimes R(\beta' -
\gamma)$ for all $\gamma \in Q^+$ such that every term above lies in
$Q^{+}$.
\end{Prop}
\begin{proof}
The proof is almost identical to that of \cite[Proposition
2.18]{KL09}.
\end{proof}

For $\alpha = \sum_{i\in I} k_i \alpha_i \in Q^+$ with
$|\alpha|=d$, we define
$$ \pol(\alpha) = \prod_{\mathbf{i} \in \seq(\alpha)} \F[ x_{1,\mathbf{i}}, \ldots, x_{d,\mathbf{i}} ]. $$
Then the symmetric group $\sg_d$ acts on $\pol(\alpha)$ by $w \cdot x_{k,\mathbf{i}} := x_{w(k),w(\mathbf{i})}$ for $w\in \sg_d$.
Let
$$ \sym(\alpha) = \pol(\alpha)^{\sg_d}. $$
Note that $ \sym(\alpha) \simeq \bigotimes_{i\in I}
\F[x_1,\ldots,x_{k_i}]^{\sg_{k_i}} $. Considering $\sym(\alpha)$ as
a subalgebra of $R(\alpha)$ via the natural inclusion $\pol(\alpha)
\hookrightarrow R(\alpha)$ sending $x_{k,\mathbf{i}}$ to $x_{k}
1_{\mathbf{i}}$, we have the following lemma.
\begin{Lem}\ \label{Lem:center of R(alpha)}
\begin{enumerate}
\item $\sym(\alpha)$ is the center of $R(\alpha)$.
\item $R(\alpha)$ is a free module of rank $(d!)^{2}$ over its center $\sym(\alpha)$.
\end{enumerate}
\end{Lem}
\begin{proof}
We first consider the case when $\alpha = m\alpha_i$ for $i \in I$.
If $i \in \Ire$, it follows from $R(m\alpha_i) \simeq NH_m$ that
$\sym(\alpha)$ is the center of $R(m\alpha_i)$. Suppose that $i\in
\Iim$. By Lemma \ref{Lem:R(m_i) for i in Iim}, $R(\alpha)$ can be
considered as a subalgebra of $\End(\F[x_1, \ldots, x_d])$. Let
$\mathbf{x}_k$ be the endomorphism of $\F[x_1, \ldots, x_m]$ defined
by multiplication by $x_k$. It is obvious that $\sym(\alpha)$ is
contained in the center of $R(\alpha)$ and $ \F[ \mathbf{x}_1,
\ldots, \mathbf{x}_m ] \subset R(\alpha)$.

For $ f\in \F[ \mathbf{x}_1, \ldots, \mathbf{x}_m ]$, from the defining relations,
we have
$$ f \tau_{i_1} \cdots \tau_{i_k} = \tau_{i_1} \cdots \tau_{i_k}( r_{i_k} \cdots r_{i_1} f) + \cdots \text{ lower terms } \cdots $$
with respect to the Bruhat order.
Let $y = \sum_{i} \tau_{w_i} f_i $ be an element in the center of $R(\alpha)$. We assume $\ell(w_1) \ge \ell(w_{k})$ for all $k$.
Take $j$ such that $w_1(j) \ne j$. Then
$$y \mathbf{x}_j - \mathbf{x}_j y  = y ( \mathbf{x}_j - \mathbf{x}_{w_1(j)} ) + \cdots \text{ lower terms }\cdots, $$
which implies $\tau_{w_i} = 1$ for all $i$. Since $y$ commutes with all $\tau_i$, $y$ should be a symmetric polynomial.
Therefore, the center of $R(\alpha)$ is $\sym(\alpha)$.

We now deal with the general case when $\alpha \in Q^+$. In this
case, using the fact that $\sym(m\alpha_i)$ is the center of
$R(m\alpha_i)$ for $i\in I$, our assertion can be proved in the same
manner as in \cite[Thoerem 2.9, Corollary 2.10]{KL09}.
\end{proof}

\vskip 1em


\subsection{Quantum Serre relations} \

Let $R(\alpha)$-mod (resp.\ $R(\alpha)$-pmod, $R(\alpha)$-fmod) be the category of arbitrary
(resp.\ finitely generated projective, finite-dimensional) graded left $R(\alpha)$-modules.
The morphisms in these categories are homogeneous homomorphisms. 
Let
\begin{align*}
K_0(R) = \bigoplus_{\alpha \in Q^+} K_0(R(\alpha)\text{-pmod})\  \text{ and }\  G_0(R) = \bigoplus_{\alpha \in Q^+} G_0(R(\alpha)\text{-fmod}),
\end{align*}
where $K_0(R(\alpha)$-pmod) (resp.\ $G_0(R(\alpha)$-fmod)) is the
Grothendieck group of $R(\alpha)$-pmod (resp.\ $R(\alpha)$-fmod).
Then $K_0(R)$ and $G_0(R)$ have the
$\A$-module structure given by $q[M] = [M\langle -1 \rangle]$, where
$[M]$ is the isomorphism classes of an $R(\alpha)$-module $M$. For
$M,N \in R(\alpha)$-mod, let $\Hom(M,N)$ be the $\F$-vector space of
homogeneous homomorphisms of degree $0$, and let $\Hom(M \langle k
\rangle,N) = \Hom(M,N \langle -k \rangle)$ be the $\F$-vector space
of homogeneous homomorphisms of degree $k$. Define
$$ \HOM(M,N)= \bigoplus_{k\in \Z} \Hom(M,N \langle k \rangle). $$

Let $\sym^+(\alpha)$ be the maximal ideal of $\sym(\alpha)$. Since
$\sym^+(\alpha)$ acts on any irreducible graded $R(\alpha)$-module
trivially, the isomorphism classes of irreducible graded modules
over $R(\alpha)$ are in 1-1 correspondence with the isomorphism
classes of irreducible graded modules over the quotient
$R(\alpha)/\sym^+(\alpha) R(\alpha)$. It follows from Lemma
\ref{Lem:center of R(alpha)} that there are only finitely many
irreducible $R(\alpha)$-modules, and all irreducible
$R(\alpha)$-modules are finite-dimensional. Note that
$R(\alpha)$ has the Krull-Schmidt unique direct sum
decomposition property for finitely generated modules
 since each graded part of $R(\alpha)$ is
finite-dimensional. Hence irreducible $R(\alpha)$-modules form a
basis of $G_0(R(\alpha)\text{-fmod})$ as an $\A$-module, which implies
that the projective covers of irreducible $R(\alpha)$-modules form a
basis of $K_0(R(\alpha)\text{-pmod})$ as an $\A$-module.

Let us consider the $\A$-bilinear pairing $(\ ,\ ) : K_0(R(\alpha))
\times G_0(R(\alpha)) \longrightarrow \A $ defined by
\begin{equation}\label{eq:paring between K and G}
([P],[M]) =  \qdim(P^{\star} \otimes_{R(\alpha)} M),
\end{equation}
where  $\qdim(N) :=  \sum_{i\in \Z} (\dim_{\bR} N_i)q^i $ for a $\Z$-graded
module $N = \bigoplus_{i\in \Z} N_i$. 
Then, the paring $(\ ,\ )$
is perfect. Thus $K_0(R(\alpha))$ and $G_0(R(\alpha))$ are dual to
each other with respect to the pairing $(\ ,\ )$.
By Lemma
\ref{Lem:center of R(alpha)}, the pairing \eqref{eq:paring between K
and G} can be extended to an $\A$-bilinear form $(\ ,\ ) :
K_0(R(\alpha)) \times K_0(R(\alpha)) \longrightarrow \Q(q)$ given by
\begin{equation} \label{eq:paring of K}
([P], [Q]) =  \qdim(P^{\star} \otimes_{R(\alpha)} Q).
\end{equation} 
Since the pairing $\eqref{eq:paring between K and G}$ is perfect and
$ P^{\star} \otimes_{R(\alpha)} Q \simeq Q^{\star}
\otimes_{R(\alpha)} P $, we conclude that the pairing
\eqref{eq:paring of K} is a nondegenerate symmetric bilinear form on
$K_0(R(\alpha))$.

For a finite-dimensional $R(\alpha)$-module $M$, we define the {\em
character} $\ch_q(M)$ of $M$ to be
$$ \qch(M) =  \sum_{\mathbf{i}\in \seq(\alpha)} (\dim_q (1_\mathbf{i} M))\mathbf{i}.$$
For $\mathbf{i}=(i_1^{(d_1)} \ldots i_r^{(d_r)}) \in \seqd(\alpha)$, let
$$ \ie_{\mathbf{i}} :=  \ie_{i_1, d_1} \otimes  \cdots \otimes \ie_{i_r, d_r}, $$
where
$$ \ie_{i, d} := \left\{
                 \begin{array}{ll}
                   \tau_{w_0}x_1^{d-1}\cdots x_{d-1} 1_{(i\ldots i)} & \hbox{ if } i \in \Ire, \\
                   1_{(i\ldots i)} & \hbox{ if } i\in \Iim,
                 \end{array}
               \right.
$$
and $w_0 = r_1r_2r_1\cdots r_{d-1}\cdots r_{1} $
is the longest
element in $\sg_d$. Since each $\ie_{i_k, d_k}$ is an idempotent in
$R(d_k \alpha_{i_k})$ $(k=1, \ldots, r)$, $\ie_{\mathbf{i}}$ is an
idempotent.
Define an $R(\alpha)$-module
$P_{\mathbf{i}}$ corresponding to $\mathbf{i}=(i_1^{(d_1)} \ldots
i_r^{(d_r)}) \in \seqd(\alpha) $ by
\begin{align} \label{Eq:def of Pi}
P_{\mathbf{i}} := R(\alpha) \ie_\mathbf{i}\ \left\langle  \sum_{\substack{k=1,\ldots,r, \\ i_k \in \Ire}} \frac{d_k(d_k-1)(\alpha_{i_k}|\alpha_{i_k})}{4}  \right \rangle .
\end{align}
Note that $P_{\mathbf{i}}$ is a projective graded $R(\alpha)$-module. By construction, if $i \in \Iim$, then
$$ P_{(i^{(d)})} = P_{(\underbrace{i\ldots i}_{d})}. $$
For a finitely generated graded projective $R(\alpha)$-module $P$,
define
\begin{align} \label{Eq:bar involution}
\overline{P}= \HOM(P, R(\alpha))^\star.
\end{align}
Note that $\overline{P}$ is a graded projective left
$R(\alpha)$-module and that $\overline{P_\mathbf{i}\langle a
\rangle} \simeq P_\mathbf{i}\langle -a \rangle $ for $\mathbf{i} \in
\seqd(\alpha)$. Hence we get a $\Z$-linear involution $-: K_0(R)
\rightarrow K_0(R)$.

We now prove the quantum Serre relations on $K_0(R)$. Suppose that
$i\in \Ire, j\in I$ and $a_{ij} \ne 0$. Let $N = 1-a_{ij}$ and take
nonnegative integers $a,b \ge 0 $ with $a+b = N$. Define the
homogeneous elements
$$ \alpha_{a,b}^+ := \alphaABpl, \quad  \alpha_{a,b}^- := \alphaABmi. $$
Choose a pair of sequences $\mathbf{i}_1$ and $\mathbf{i}_2$ such
that $\mathbf{i}_1 * (i^{(a)}ji^{(b)}) * \mathbf{i}_2 \in
\seqd(\alpha)$, and write $P_{ (\cdots\ i^{(a)}ji^{(b)}\cdots )}$ for
$P_{  \mathbf{i}_1 * (i^{(a)}ji^{(b)}) * \mathbf{i}_2 }$.
Then these
elements give rise to homomorphisms of graded projective modules
\begin{equation}
\begin{aligned}
d_{a,b}^+ &: P_{ (\cdots i^{(a)}ji^{(b)} \cdots)} \longrightarrow P_{ ( \cdots i^{(a+1)}ji^{(b-1)} \cdots)  },\\
 & \qquad\qquad m  \quad \quad \longmapsto \quad m \cdot 1_{\mathbf{i}_1} \otimes \alpha_{a,b}^+ \otimes 1_{\mathbf{i}_2} , \\
d_{a,b}^- &: P_{( \cdots i^{(a)}ji^{(b)}\cdots)  } \longrightarrow P_{(\cdots i^{(a-1)}ji^{(b+1)}\cdots) },\\
 & \qquad\qquad m  \quad \quad \longmapsto \quad m \cdot 1_{\mathbf{i}_1} \otimes \alpha_{a,b}^- \otimes 1_{\mathbf{i}_2}.
\end{aligned}
\end{equation}
Set $d_{N,0}^+ = 0 $ and $d_{0,N}^- = 0 $. Then we have
$$
\xymatrix{    0\ \  \ar@<0.4ex>[r]^{} & \ar@<0.4ex>[l]^{} P_{ (\cdots i^{(0)}ji^{(N)} \cdots)}
\ar@<0.4ex>[r]^{\qquad d_{0,N}^+   } & \ar@<0.4ex>[l]^{\qquad  d_{1,N-1}^- }\ \ \cdots \ \ \ar@<0.4ex>[r]^{d_{a-1,b+1 }^+ \qquad } & \ar@<0.4ex>[l]^{ d_{a,b}^- \qquad} P_{ (\cdots i^{(a)}ji^{(b)} \cdots)}
\ar@<0.4ex>[r]^{\qquad d_{a,b }^+  } & \ar@<0.4ex>[l]^{\qquad d_{a+1,b-1 }^-}\ \ \cdots \ \ \ar@<0.4ex>[r]^{ d_{N-1,1 }^+ \qquad} & \ar@<0.4ex>[l]^{ d_{N,0 }^- \qquad}
P_{ (\cdots i^{(N)}ji^{(0)} \cdots)}  \ar@<0.4ex>[r]^{} & \ar@<0.4ex>[l]^{} \ \ 0
 }.
$$

\begin{Lem} \ \label{Lem:Serre}
\begin{enumerate}
\item $ d_{a, b}^+ \circ d_{a-1, b+1}^+ = 0,\ \   d_{a, b}^- \circ d_{a+1, b-1}^- = 0  $ for $a,b > 0$.
\item $ d_{N-1, 1}^+ \circ d_{N, 0}^- =  t_{i,j;-a_{ij},0} \id, \ \  d_{1,N-1}^- \circ d_{0,N}^+ =
(-1)^{N-1} t_{i,j;-a_{ij},0} \id $.
\item For $ 1 < a, b<N$, we have
$$  d_{a+1,b-1}^- \circ d_{a,b}^+ - d_{a-1,b+1}^+ \circ d_{a,b}^- = (-1)^{b-1} t_{i,j;-a_{ij},0} \id. $$
\end{enumerate}
\end{Lem}
\begin{proof} If $j \in \Ire$, this lemma was proved in \cite{KL11,
R08}. We will prove our lemma when $j \in \Iim$.

 Let $d = 2-a_{ij}$
and let $e_{a,b} = 1_{i,a} \otimes 1_{(j)} \otimes 1_{i, b}$ for
$a,b \ge 0$. Since $i \in \Ire$ and $\mathcal{P}_i(u,v)=1$, it
follows from \cite{KL11,R08} that
\begin{align*}
\alpha_{a,b}^+ &= \tau_{d-1} \cdots \tau_{a+1} e_{a+1,b-1} = e_{a,b} \tau_{d-1} \cdots \tau_{a+1} e_{a+1,b-1}, \\
\alpha_{a,b}^- &= \tau_{1} \cdots \tau_{a} e_{a-1,b+1} = e_{a,b} \tau_{1} \cdots \tau_{a} e_{a-1,b+1}.
\end{align*}
By a direct computation, we have
\begin{align*}
\alpha_{a-1,b+1}^+ \alpha_{a,b}^+  & =  e_{a-1,b+1} \tau_{d-1} \cdots \tau_{a} e_{a,b} e_{a,b} \tau_{d-1} \cdots \tau_{a+1} e_{a+1,b-1}\\
&=  e_{a-1,b+1} \tau_{d-1} \cdots \tau_{a}  \tau_{d-1} \cdots \tau_{a+1} e_{a+1,b-1} \\
& = 0.
\end{align*}
In the same manner, we get $ \alpha_{a+1,b-1}^- \alpha_{a,b}^- = 0 $.

On the other hand, using the same argument as in \cite{KL11}, for $a,b>0$, we obtain
\begin{align*}
\alpha_{a,b}^+ \alpha_{a+1,b-1}^- &= e_{a,b} \tau_{d-1} \cdots \tau_{a+1} e_{a+1,b-1} e_{a+1,b-1} \tau_{1} \cdots \tau_{a+1} e_{a,b} \\
&= \tau_{1} \cdots \tau_{a-1} \tau_{d-1} \cdots \tau_{a+1} \tau_{a} \tau_{a+1} e_{a,b}, \\
\alpha_{a,b}^- \alpha_{a-1,b+1}^+ &= e_{a,b} \tau_{1} \cdots \tau_{a} e_{a-1,b+1} e_{a-1,b+1} \tau_{d-1} \cdots \tau_{a} e_{a,b} \\
&= \tau_{1} \cdots \tau_{a-1} \tau_{d-1} \cdots \tau_{a} \tau_{a+1} \tau_{a} e_{a,b},
\end{align*}
which implies
$$\alpha_{N,0}^- \alpha_{N-1,1}^+ =  t_{i,j;-a_{ij},0} e_{N,0},\quad \alpha_{0,N}^+ \alpha_{1,N-1}^- = (-1)^{N-1}t_{i,j;-a_{ij},0} e_{0,N},$$
and
\begin{align*}
\alpha_{a,b}^+ \alpha_{a+1,b-1}^- - \alpha_{a,b}^- \alpha_{a-1,b+1}^+
&= \tau_{1} \cdots \tau_{a-1} \tau_{d-1} \cdots \tau_{a+2} (\tau_{a+1} \tau_{a} \tau_{a+1} - \tau_{a} \tau_{a+1} \tau_{a}) e_{a,b}\\
&= \tau_{1} \cdots \tau_{a-1} \tau_{d-1} \cdots \tau_{a+2} (\overline{Q}_{i,j}(x_{a},x_{a+1},x_{a+2})) e_{a,b}\\
&= (-1)^{b-1} t_{i,j;-a_{ij},0} e_{a,b}.
\end{align*}
Therefore, we obtain
\begin{align*}
& \alpha_{a-1,b+1}^+ \alpha_{a,b}^+ = 0, \quad  \alpha_{a+1,b-1}^- \alpha_{a,b}^- = 0,\\
& \alpha_{N,0}^- \alpha_{N-1,1}^+ =  t_{i,j;-a_{ij},0} e_{N,0},\quad \alpha_{0,N}^+ \alpha_{1,N-1}^- = (-1)^{N-1} t_{i,j;-a_{ij},0} e_{0,N}, \\
&  \alpha_{a,b}^+ \alpha_{a+1,b-1}^- - \alpha_{a,b}^- \alpha_{a-1,b+1}^+ = (-1)^{b-1} t_{i,j;-a_{ij},0} e_{a,b},
\end{align*}
as desired.
\end{proof}

\begin{Thm} \label{Thm:Serre} \
\begin{enumerate}
\item If $a_{ij}=0$, then
$[P_{ (\cdots\ ij\cdots )}] = [P_{ (\cdots\ ji\cdots )}].$
\item If $i \in \Ire$ and $j \in I$ with $i \neq j$, then
$$\sum_{k=0}^{1-a_{ij}} (-1)^k [ P_{ (\cdots\ i^{(k)}ji^{(1-a_{ij}-k)}\cdots )}  ] = 0.$$
\end{enumerate}
\end{Thm}
\begin{proof}
If $a_{ij}=0$, let $\tau^{-}$ (resp. $\tau^{+}$) be the element in
$R$ changing $(ij)$ to $(ji)$ (resp. $(ji)$ to $(ij)$) and define
$$d^-: P_{ (\cdots\ ij\cdots )} \to P_{ (\cdots\ ji\cdots )} \text{ (resp.\ } d^+: P_{ (\cdots\ ji\cdots )} \to P_{ (\cdots\ ij\cdots )} \text{)} $$
to be the map given by right multiplication by $t_{i,j;0,0} \tau^{-}$ (resp.
$t_{j,i;0,0} \tau^{+}$). From the defining relation $\eqref{Eq:def rel 1}$, we
see that $d^+$ and $d^-$ are inverses to each other. Hence
$$[P_{ (\cdots\ ij\cdots )}] = [P_{ (\cdots\ ji\cdots )}]. $$

Suppose that $a_{ij} \ne 0$ and $i\in \Ire$. By Lemma
\ref{Lem:Serre}, the complex $\left( P_{ (\cdots i^{(a)}ji^{(b)}
\cdots)}, d_{a,b}^+ \right)$ becomes an exact sequence with the
splitting maps  $(-1)^{b-1} t_{ij;-a_{ij},0} d_{a,b}^-$. Therefore, our assertion
follows from the Euler-Poincar\`{e} principle.
\end{proof}

\vskip 3em

\section{Categorification of $U_q^-(\g)$} \label{Sec:GKM}

In this section, we show that the Khovanov-Lauda-Rouquier algebra
$R$ gives a categorification of $U_\A^-(\g)$.

\subsection{Induction and restriction}\

For $\alpha, \beta \in Q^+$, consider the natural embedding
\begin{align*}
\iota_{\alpha,\beta}: R(\alpha)\otimes R(\beta) \hookrightarrow R(\alpha+\beta),
\end{align*}
which maps $1_\alpha \otimes 1_\beta$ to $1_{\alpha,\beta}$. For $M \in R(\alpha)\otimes R(\beta)$-mod and $N \in R(\alpha+\beta)$-mod, we define
\begin{align*}
\ind_{\alpha, \beta} M &= R(\alpha+\beta) \otimes_{R(\alpha)\otimes R(\beta)}M , \\
\res_{\alpha, \beta} N &= 1_{\alpha, \beta} N  .
\end{align*}
Then it is straightforward to verify that the Frobenius reciprocity
holds:
\begin{align} \label{Eq:reciprocity}
\HOM_{R(\alpha+\beta)}(\ind_{\alpha,\beta}M, N) \simeq \HOM_{R(\alpha)\otimes R(\beta)}(M, \res_{\alpha,\beta}N).
\end{align}
When there is no ambiguity, we will simply write $\ind$ and $\res $
for $\ind_{\alpha, \beta}$ and $\res_{\alpha,\beta} $, respectively.

Given $\mathbf{i} \in \seq(\alpha)$ and $\mathbf{j} \in
\seq(\beta)$, a sequence $\mathbf{k}\in \seq(\alpha+\beta)$ is
called a {\em shuffle} of $\mathbf{i}$ and $\mathbf{j}$ if
$\mathbf{k}$ is a permutation of $\mathbf{i} * \mathbf{j}$ such that
$\mathbf{i}$ and $\mathbf{j}$ are subsequences of $\mathbf{k}$. For
a shuffle $\mathbf{k}$ of $\mathbf{i} \in \seq(\alpha)$ and
$\mathbf{j} \in \seq(\beta)$, let
$$\deg(\mathbf{i},\mathbf{j},\mathbf{k}) = \deg(\tau_{w}1_{\mathbf{i}*\mathbf{j}}),$$
where $w$ is the element in $\sg_{|\alpha| + |\beta|} /
\sg_{|\alpha| } \times \sg_{|\beta|}$ corresponding to $\mathbf{k}$.
Given $X = \sum x_\mathbf{i}\ \mathbf{i}$ and $Y = \sum
y_\mathbf{j}\ \mathbf{j}$, the {\em shuffle product} $X \star Y$ of
$X$ and $Y$ is defined to be
$$ X \star Y = \sum_{\mathbf{k}}
\left(\sum_{\mathbf{i},\mathbf{j}}q^{\deg(\mathbf{i},\mathbf{j},\mathbf{k})}
x_\mathbf{i} y_\mathbf{j}  \right)  \ \mathbf{k}, $$ where
$\mathbf{k}$ runs over all the shuffles of $\mathbf{i}$ and
$\mathbf{j}$. Then, by Proposition \ref{Prop:basis of R(alpha)}, we
have
\begin{align} \label{Eq:shuffle}
\qch(\ind_{\alpha, \beta} M \boxtimes N) = \qch(M) \star \qch(N)
\end{align}
for $M\in R(\alpha)$-fmod and $N \in R(\beta)$-fmod.

By Corollary \ref{Cor:R(alpha+beta) is free},
$\ind_{\alpha,\beta}$ and $\res_{\alpha,\beta}$ take projective
modules to projective modules. Since $1_{\alpha, \beta}$ is an
idempotent, $\ind_{\alpha,\beta}$ and $\res_{\alpha,\beta}$ can be
viewed as exact functors between the categories of projective
modules. Hence we obtain the linear maps
\begin{align*}
& \ind_{\alpha,\beta}: K_0(R(\alpha)) \otimes K_0(R(\beta)) \longrightarrow K_0(R(\alpha+\beta)),\\
& \res_{\alpha,\beta}: K_0(R(\alpha+\beta)) \longrightarrow K_0(R(\alpha)) \otimes K_0(R(\beta)).
\end{align*}
It follows from Proposition \ref{Prop:Mackey} that
\begin{equation} \label{Eq:ind and res of Pi}
\begin{aligned}
&\ind_{\alpha, \beta}(P_\mathbf{i}\boxtimes P_\mathbf{j})  \simeq P_{\mathbf{i}*\mathbf{j}}  &\text{ for }&  \mathbf{i} \in \seq(\alpha),\ \mathbf{j} \in \seq(\beta),\\
&\res_{\alpha, \beta} P_\mathbf{k}  \simeq \bigoplus_{\mathbf{i},\mathbf{j}} P_\mathbf{i} \boxtimes P_\mathbf{j}\langle  -\deg(\mathbf{i},\mathbf{j},\mathbf{k} ) \rangle
 &\text{ for }& \mathbf{k} \in \seq(\alpha+\beta),
\end{aligned}
\end{equation}
where the sum is taken over all $\mathbf{i}\in \seq(\alpha)$,
$\mathbf{j}\in \seq(\beta)$ such that $\mathbf{k}$ can be
expressed as a shuffle of $\mathbf{i}$ and $\mathbf{j}$. We extend
the linear maps $\ind_{\alpha,\beta}$ and $\res_{\alpha,\beta}$ to
the whole space $K_0(R)$ by linearity:
\begin{align*}
\ind:& K_0(R) \otimes K_0(R) \longrightarrow K_0(R)\quad \text{given by }\quad ([M],[N]) \mapsto [\ind_{\alpha, \beta} M \boxtimes N],\\
\res:& K_0(R)  \longrightarrow K_0(R)\otimes K_0(R)\quad \text{given
by }\quad [L] \mapsto \sum_{\alpha',\beta' \in Q^+ }[\res_{\alpha',
\beta'} L].
\end{align*}
We denote by $[M][N]$ the product $\ind([M],[N])$ of $[M]$ and $[N]$ in $K_0(R)$.

\begin{Prop}\
\begin{enumerate}
\item The pair $(K_0(R), \ind)$ becomes an associative unital $\A$-algebra.
\item The pair $(K_0(R), \res)$ becomes a coassociative counital $\A$-coalgebra.
\end{enumerate}
\end{Prop}
\begin{proof}
Our assertions on associativity and coassociativity follow from the
transitivity of induction and restriction. Define
\begin{align*}
\iota :& \ \A \longrightarrow  K_0(R) \quad \text{ by } \ \iota( \sum_k a_k q^k ) =  \sum_k a_k q^k \mathbf{1},\\
\epsilon:&\ K_0(R) \longrightarrow \A  \quad \text{ by } \
\epsilon(M) = \qdim(M_0),
\end{align*}
where $M_0$ is the image of $M$ under the natural projection $K_0(R) \rightarrow K_0(R(0))$.
Then one can verify that $\iota$ (resp.\ $\epsilon$) is the unit (resp.\ counit) of $K_0(R)$.
\end{proof}

We define the algebra structure on $K_0(R)\otimes K_0(R)$ by
$$ \left( [M_1] \otimes [M_2] \right)  \cdot \left([N_1] \otimes [N_2] \right)
= q^{-(\beta_2|\gamma_1)}[M_1][N_1] \otimes [M_2][N_2] $$
for $M_i \in K_0(R(\beta_i))$, $N_i \in K_0(R(\gamma_i))$ $(i=1,2)$.
Using Proposition \ref{Prop:Mackey}, we prove:
\begin{Prop} $\res:K_0(R)  \longrightarrow K_0(R)\otimes K_0(R)$ is an algebra homomorphism.
\end{Prop}

Let us recall the bilinear paring $(\ ,\ ): K_0(R) \otimes K_0(R)
\longrightarrow \Q(q)$ given in $\eqref{eq:paring of K}$ and the
projective modules $P_\mathbf{i}$ for $\mathbf{i}\in \seqd(\alpha)$
defined in $\eqref{Eq:def of Pi}$. We denote by $\mathbf{1}$ the
1-dimensional $R(0)$-module of degree $0$.

\begin{Prop} \label{Prop:paring}
The bilinear pairing $(\ ,\ ): K_0(R) \otimes K_0(R) \rightarrow
\Q(q)$ satisfies the following properties:
\begin{enumerate}
\item $(\mathbf{1},\mathbf{1}) = 1$,
\item $([P_{(i)}], [P_{(j)}]) = \delta_{ij} (1-q_i^{2})^{-1}$ for $i,j\in
I$,
\item $([L], [M][N]) = (\res [L], [M] \otimes [N])$ for $[L],[M],[N] \in
K_0(R)$,
\item $([L][M],[N]) = ([L] \otimes [M], \res [N])$ for $[L],[M],[N] \in K_0(R)$.
\end{enumerate}
\end{Prop}
\begin{proof}
The assertions (1) and (2) follow from the $\Z$-grading
$\eqref{Eq:degree}$ on $R(\alpha)$. Suppose that $L \in
R(\alpha+\beta)$-pmod, $M \in R(\alpha)$-pmod and $N \in
R(\beta)$-pmod. Then we have
\begin{align*}
([L], [M][N]) &= \qdim(L^\star \otimes_{R(\alpha+\beta)} \ind_{\alpha,\beta} M \boxtimes N) \\
&=   \qdim( (\res_{\alpha, \beta}L)^\star \otimes_{R(\alpha)\otimes R(\beta)} M \boxtimes N) = (\res_{\alpha, \beta}L, M \boxtimes N),
\end{align*}
which yields that $([L], [M][N]) = (\res [L], [M] \otimes [N])$.

The assertion (4) can be proved in the same manner.
\end{proof}


Define a map $\Phi: U_{\A}^{-}(\g) \longrightarrow K_{0}(R)$ by
\begin{equation}
f_{i_1}^{(d_1)} \cdots f_{i_r}^{(d_r)} \longmapsto [P_{( i_1^{(d_1)}
\ldots i_r^{(d_r)}) }].
\end{equation}

\begin{Thm}\  \label{Thm:Phi is injective}
The map $\Phi$ is an injective algebra homomorphism.
\end{Thm}

\begin{proof}
By Theorem \ref{Thm:Serre}, $\Phi$ is an algebra homomorphism.
Since both of $\Delta_0$ and $\res$ are algebra homomorphisms and
$$\Delta_0 (f_i) = f_i \otimes \mathbf{1} + \mathbf{1} \otimes f_i,
\ \res (P_{(i)}) = P_{(i)} \otimes \mathbf{1} + \mathbf{1} \otimes
P_{(i)}\ (i\in I),$$ by \eqref{Eq:def of ()K and ()L} and
Proposition \ref{Prop:paring}, we have
$$(x,y)_L = (\Phi(x), \Phi(y)) \ \ \text{for all} \ x, y \in
U_{\A}^{-}(\g).$$ Hence $\text{Ker} \Phi$ is contained in the
radical of the bilinear form $(\ , \ )_{L}$, which is nondegenerate.
The assertion follows immediately.
\end{proof}

Therefore, $\text{Im} \Phi$ gives a categorification of
$U_{\A}^{-}(\g)$. In general, the homomorphism $\Phi$ is not
surjective. However, if $a_{ii} \neq 0$ for all $i \in I$, then
$\Phi$ is an isomorphism as will be shown in the next subsection.

\vskip 1em

\subsection{Surjectivity of $\Phi$}\

In this subsection, we assume that $a_{ii} \ne 0$ for all $i\in I$.
We have seen in Section \ref{Sec:KLR} that the algebra $R(m
\alpha_i)$ has a unique irreducible graded module $L(i^m)$. If $i\in
\Ire$, we have
$$L(i^m) \simeq \ind_{\F[x_1,\ldots,x_m]}^{R(m\alpha_i)} \mathbf{1},$$
where $\mathbf{1}$ is the trivial $\F[x_1,\ldots,x_m]$-module of dimension 1 over $\bR$. Note $\qdim(\mathbf{1})=1$. If
$i\in \Iim$, then $L(i^m)$ is isomorphic to the trivial
graded $R(m\alpha_i)$-module with defining relations
given in $\eqref{Eq:def of L in Iim}$. We know $\qch( L(i^m) )=
(i\ldots i)$.

For $M \in R(\alpha)$-mod and $i\in I$, define
\begin{equation*} 
\begin{aligned}
\Delta_{i^k} M &= 1_{k\alpha_i, \alpha- k\alpha_i} M
 \in R(k\alpha_i)\otimes R(\alpha-k\alpha_i)\text{-mod}, \\
\ep_i(M) &= \max\{ k \ge 0 \mid \Delta_{i^k} M \ne 0 \}, \\
\ke_i(M) &= \soc(\res_{\alpha-\alpha_i}^{\alpha_i, \alpha-\alpha_i} \circ \Delta_{i}(M))
\in R(\alpha-\alpha_i)\text{-mod}  , \\
\kf_i(M) &= \hd \ind_{\alpha_i, \alpha} ( L(i)\boxtimes M ) \in R(\alpha+\alpha_i)\text{-mod}.
\end{aligned}
\end{equation*}
Note that they are defined in the opposite manner to
\cite{KL09,LV09}. By the Frobenius reciprocity, we have
\begin{align} \label{Eq:reciprocity2}
\HOM_{R(\alpha)}(\ind_{m\alpha_i, \alpha-m\alpha_i} L(i^m) \boxtimes N, M) \simeq \HOM_{R(m\alpha_i) \otimes R(\alpha-m\alpha_i)}(L(i^m) \boxtimes N,\Delta_{i^m}M)
\end{align}
for $N \in R(\alpha - m\alpha_i)$-mod and $M \in R(\alpha)$-mod.

\begin{Lem} \label{Lem:Kato for i in Iim}
 For $i\in \Iim$, take $m_1, \ldots, m_k \in \Z_{> 0}$ and set $m = m_1 + \cdots +
 m_k$. Then the following statements hold.
\begin{enumerate}
\item $\res_{m_1 \alpha_i, \ldots, m_k \alpha_i} L(i^m) $ is isomorphic to $ L(i^{m_1}) \boxtimes \cdots \boxtimes  L(i^{m_k})$.
\item $\ind_{m_1 \alpha_i, \ldots, m_k \alpha_i}( L(i^{m_1}) \boxtimes \cdots \boxtimes  L(i^{m_k}))$ has an irreducible head, which is isomorphic to $L(i^m)$.
\end{enumerate}
\end{Lem}
\begin{proof}
The assertion (1) follows from the definition $\eqref{Eq:def of L in
Iim}$. To prove (2), for simplicity, we assume $k=2$. Let
$\mathbf{i} = (\underbrace{i \ldots i}_{m})$ and $L = \ind L_1
\boxtimes L_2, $ where $L_j := L(i^{m_j}) \ (j=1,2)$. Set
$$L' = \{ x \in L | \deg(x) > 0 \}. $$
Then, since $ 1 \otimes (L_1 \boxtimes L_2) \nsubseteq L'$, $L'$
is a unique maximal submodule of $L$; i.e., $ L / L'\simeq L(i^m)$
as a graded module. We will show that $\hd L$ is irreducible. By a
direct computation,
\begin{align*}
\qch(L) &= \sum_{w\in \sg_{m_1 + m_2}/ \sg_{m_1}\times \sg_{m_2}} q^{-\ell(w)(\alpha_i|\alpha_i)}\mathbf{i} \\
&=  \mathbf{i} + \text{( ...other terms with $q^{t}$...)} \ \ (t\in \Z_{>0}).
\end{align*}
Note that $\qch(L_1 \boxtimes L_2) = \mathbf{i}$. For any quotient
$Q$ of $L$, by the Frobenius reciprocity $\eqref{Eq:reciprocity}$,
we have an injective homomorphism of degree 0
$$ L_1 \boxtimes L_2 \hookrightarrow \res_{m_1 \alpha_i, m_2\alpha_i} Q ,$$
which yields
$$\qch(Q) = \mathbf{i} + \text{( ...other terms with
$q^{t}$... ) for $t \in \Z_{>0}$}.$$ Therefore, $\hd L$ has only
one summand, and hence it is irreducible.
\end{proof}

\begin{Lem} \label{Lem:epsilon and irr submodule}
Let $M$ be an irreducible $R(\alpha)$-module and let $L(i^m)
\boxtimes N$ be an irreducible submodule of the $R(m\alpha_i)
\otimes R(\alpha-m\alpha_i)$-module $\Delta_{i^m}M$. Then $\ep_i(N)
= \ep_i(M)-m$.
\end{Lem}
\begin{proof} If $i\in \Ire$, then the proof is the same as that of \cite[Lemma 3.6]{KL09}. If $i\in
\Iim$, by the definition, we have $\ep_i(N) \le \ep_i(M)-m$. From
the equation $\eqref{Eq:reciprocity2}$, we obtain
$$0 \rightarrow K \rightarrow \ind L(i^m)\boxtimes N \rightarrow M \rightarrow 0$$
for some submodule $K$ of $\ind L(i^m)\boxtimes N$. It follows from
$\eqref{Eq:shuffle}$ and the exactness of  $\Delta_{i^k}$ that
$\ep_i(N) \ge \ep_i(M)-m$, which yields our assertion.
\end{proof}

\begin{Lem} \label{Lem:properties of L(im) bt N}
Let $N$ be an irreducible $R(\alpha)$-module with $\ep_i(N)=0$ and
let $M = \ind L(i^m) \boxtimes N$. Then we have
\begin{enumerate}
\item $\Delta_{i^m}M \simeq L(i^m) \boxtimes  N$,
\item $\hd M$ is an irreducible module with $\ep_i(\hd M) = m$,
\item for all other composition factors $L$ of $M$, we have $\ep_i(L) < m$.
\end{enumerate}
\end{Lem}
\begin{proof}
Our assertion can be proved in the same manner as in \cite[Lemma 3.7]{KL09}.
\end{proof}

\begin{Lem} \label{Lem:Delta of M}
Let $M$ be an irreducible $R(\alpha)$-module and let $\ep = \ep_i(M)$. Then
$\Delta_{i^\ep}M$ is isomorphic to $L(i^\ep) \boxtimes N$ for some irreducible $R(\alpha-\varepsilon\alpha_i)$-module $N$ with $\ep_i(N)=0$.
\end{Lem}
\begin{proof}
Our assertion can be proved in the same manner as in \cite[Lemma 5.1.4]{K05} (cf.\ \cite[Lemma 3.8]{KL09}).
\end{proof}

\begin{Lem} \label{Lem:hd ind for i in Iim}
Suppose that $i \in \Iim$ and $N$ is an irreducible
$R(\alpha)$-module with $\ep_i(N)=0$. Let
$$ M = \ind L(i^{m_1}) \boxtimes \cdots \boxtimes L(i^{m_k}) \boxtimes N  $$
 for some positive integers $m_{1}, \ldots m_{k} \in \Z_{>0}$ and set
 $m= m_1 + \cdots + m_k$. Then
\begin{enumerate}
\item $\hd M$ is irreducible,
\item $\ep_i(\hd M) = m$.
\end{enumerate}
\end{Lem}
\begin{proof}
By the definition, we have
$$ \Delta_{i^m} M = ( \ind L(i^{m_1}) \boxtimes \cdots \boxtimes L(i^{m_k}) ) \boxtimes N. $$
In the Grothendieck group $G_0(R(m\alpha_i) \otimes R(\alpha -
m\alpha_i))$ of the category of finite-dimensional graded $R(m\alpha_i) \otimes R(\alpha - m\alpha_i)$-modules,
we have
\begin{align*}
[\Delta_{i^m} M] &= \sum_w q^{-\ell(w) (\alpha_i | \alpha_i)}[L(i^m) \boxtimes N], \\
 &= [L(i^m) \boxtimes N ] + ( \text{ ...other terms with $q^{t}$...
 }),
\end{align*}
where $w$ runs over all the elements in $\sg_m / \sg_{m_1} \times
\cdots \times \sg_{m_k}$. By the Frobenius reciprocity
$\eqref{Eq:reciprocity2}$, for any quotient $Q$ of $M$, there is a
nontrivial homomorphism of degree 0
$$ \Delta_{i^m} M = ( \ind L(i^{m_1}) \boxtimes \cdots \boxtimes L(i^{m_k}) ) \boxtimes N \rightarrow \Delta_{i^m}Q. $$
By Lemma \ref{Lem:Kato for i in Iim} (2), we have
$$ [\Delta_{i^m}Q] = [L(i^m) \boxtimes N ] + ( \text{ ...other terms with $q^{t}$... })$$
in the Grothendieck group $G_0(R(m\alpha_i) \otimes R(\alpha -
m\alpha_i))$. Therefore, by the same argument as in Lemma
\ref{Lem:Kato for i in Iim}, $ \hd M$ is irreducible and
$\ep_i(\hd M) = m$.
\end{proof}

\begin{Lem} \label{Lem:irr hd}
Let $N$ be an irreducible $R(\alpha)$-module and let $M = \ind
L(i^m) \boxtimes N$.
\begin{enumerate}
\item $\hd M$ is an irreducible module with $\ep_i( \hd M) = \ep_i(N) + m$.
\item If $i \in \Ire$, then for all other composition factors $L$ of $M$, we have $\ep_i(L) < \ep_i(N)+m$.
\end{enumerate}
\end{Lem}
\begin{proof}
If $i \in \Ire$, then the proof is identical with that of \cite[Lemma 5.1.5]{K05} (cf.\ \cite[Lemma 3.9]{KL09}).
Suppose that $i\in \Iim.$ Let $\ep = \ep_i(N)$. By Lemma \ref{Lem:Delta of M}, we have
$$ \Delta_{i^\ep} N = L(i^\ep) \boxtimes K $$
for some irreducible $R(\alpha-m\alpha_i)$-module $K$ with
$\ep_i(K)=0$. By the Frobenius reciprocity
$\eqref{Eq:reciprocity2}$, there is a surjective homomorphism
$$ \ind L(i^\ep) \boxtimes K  \twoheadrightarrow  N ,$$
which yields
$$ \ind L(i^m) \boxtimes L(i^\ep) \boxtimes K  \twoheadrightarrow \ind L(i^m) \boxtimes N . $$
Therefore, our assertion follows from Lemma \ref{Lem:hd ind for i in Iim}.
\end{proof}

\begin{Lem} \label{Lem:irr soc}
Let $M$ be an irreducible $R(\alpha)$-module. Then, for $0 \le m \le \ep_i(M) $,
the submodule $\soc \Delta_{i^m}M$ of $M$ is an irreducible module of the form $L(i^m) \boxtimes L$ with $\ep_i(L) = \ep_i(M)-m $
for some irreducible $R(\alpha-m\alpha_i)$-module $L$.
\end{Lem}
\begin{proof}
If $i \in \Ire$, then the proof is the same as that of \cite[Lemma
5.1.6]{K05} (cf.\ \cite[Lemma 3.10]{KL09}). If $i \in \Iim$, let
$\ep = \ep_i(M)$. Note that every summand of $\soc \Delta_{i^m}M$
has the form $L(i^m)\boxtimes L$ for some irreducible $R(\alpha -
m \alpha_i)$-module $L$. It follows from Lemma \ref{Lem:epsilon
and irr submodule} that
$$ \ep_i(L) = \ep - m,  $$
so $L(i^m) \boxtimes \Delta_{i^{\ep-m}}(L) \neq 0$. It is
clear that $\res^{ \ep \alpha_i,  \alpha - \ep \alpha_i}_{ m
\alpha_i, (\ep-m) \alpha_i,   \alpha - \ep \alpha_i  }
\Delta_{i^\ep}M$ has $ L(i^m) \boxtimes \Delta_{i^{\ep-m}}(L)$ as
a submodule. On the other hand, by Lemma \ref{Lem:Kato for i in
Iim} and Lemma \ref{Lem:Delta of M}, there exists an irreducible
$R(\alpha-\varepsilon \alpha_i)$-module $N$ such that
$$ \res^{ \ep \alpha_i,  \alpha - \ep \alpha_i}_{ m \alpha_i, (\ep-m) \alpha_i,   \alpha - \ep \alpha_i  } \Delta_{i^\ep}M
\simeq  L(i^m) \boxtimes L(i^{\ep - m}) \boxtimes N,$$ which is
irreducible. Hence $\soc \Delta_{i^m}M$ is irreducible and
isomorphic to $L(i^{m}) \boxtimes L$.
\end{proof}

By Lemma \ref{Lem:irr hd} and Lemma \ref{Lem:irr soc}, the operators $\ke_i$ and $\kf_i$ take irreducible modules to irreducible modules or $0$, and
$$ \ep_i(M) = \max \{ k \ge 0 \mid \ke_i^k M \ne 0 \}, \quad \ep_i(\kf_i M) = \ep_i(M)+1. $$

\begin{Lem} \label{Lem:soc and hd}
Let $M$ be an irreducible $R(\alpha)$-module. Then we have
\begin{enumerate}
\item $\soc \Delta_{i^m}M \simeq  L(i^m) \boxtimes (\ke_i^m M)$,
\item $\hd \ind (L(i^m) \boxtimes M) \simeq \kf_i^m M$.
\end{enumerate}
\end{Lem}
\begin{proof}
If $i \in \Ire$, then the proof is the same as in \cite[Lemma 5.2.1]{K05}.
Suppose that $i \in \Iim$. We first focus on the assertion (1). Since the case $m > \ep_i(M)$ is trivial,
we may assume that $m \le \ep_i(M)$. Since $L(i) \boxtimes \ke_iM \hookrightarrow \Delta_i M$, we have
$$ \underbrace{L(i)\boxtimes \cdots \boxtimes L(i)}_{m}  \boxtimes\ \ke_i^m M
\hookrightarrow \res_{\alpha_i, \ldots \alpha_i, \alpha-m \alpha_i
}^{m\alpha_i, \alpha - m\alpha_i} \Delta_{i^m}M,$$ which implies
there is a nontrivial homomorphism
$$ \ind (\underbrace{L(i)\boxtimes \cdots \boxtimes L(i)}_{m} ) \boxtimes\ \ke_i^m M
\longrightarrow \Delta_{i^m} M. $$ Since any quotient of $\ind (
L(i)\boxtimes \cdots \boxtimes L(i))$ has a 1-dimensional submodule,
$\Delta_{i^m} M$ has a submodule which is isomorphic to
$L(i^m)\boxtimes \ke_i^m M$. Hence the assertion (1) follows from
Lemma \ref{Lem:irr soc}.

For the assertion (2), by the definition of $\kf_i$, there is a
nontrivial homomorphism
$$ \ind ( \ind( \underbrace{L(i)\boxtimes \cdots \boxtimes L(i)}_{m} ) \boxtimes M) \twoheadrightarrow \kf_i^m M. $$
Using the same argument in the proof of Lemma \ref{Lem:irr hd},
we have
$$ \hd \ind ( \ind( L(i)\boxtimes \cdots \boxtimes L(i) ) \boxtimes M) \simeq \kf_i^m M. $$
On the other hand, the nontrivial homomorphism
$$ \ind( L(i)\boxtimes \cdots \boxtimes L(i) )   \longrightarrow L(i^m) $$
induces a nontrivial homomorphism
$$ \ind ( \ind( L(i)\boxtimes \cdots \boxtimes L(i) ) \boxtimes M)  \longrightarrow \ind L(i^m)\boxtimes M .$$
Therefore, we conclude $\hd \ind (L(i^m) \boxtimes M) \simeq \kf_i^m M$.
\end{proof}

\begin{Lem} \label{Lem:adjoint ke and kf}
Let $M $ be an irreducible $ R(\alpha)$-module and let $N$ be an irreducible $R(\alpha+\alpha_i)$-module.
Then we have
$$ \kf_i M \simeq N  \ \text{ if and only if } \  M \simeq \ke_i N.  $$
\end{Lem}
\begin{proof} Using Lemma \ref{Lem:soc and hd}, our assertion can be proved in the same manner as in \cite[Lemma 5.2.3]{K05}
\end{proof}

Let $\A  \seq(\alpha)$ (resp.\ $\Q(q)\seq(\alpha)$) be the free $\A$-module (resp.\ $\Q(q)$-module) generated by $\seq(\alpha)$.
For an irreducible $R(\alpha)$-module $M$, the character $\qch(M)$ can be viewed as an element in $\A\seq(\alpha)$.
Using the above lemmas, one can prove the following proposition in the same manner as in \cite[Theorem 5.3.1]{K05}.

\begin{Prop} \label{Prop:character map is injective}
The character map
$$\qch: G_0(R(\alpha)) \longrightarrow \A\seq(\alpha)$$
is injective.
\end{Prop}

Let $\mathcal{F}$ be the free associative algebra over $\Q(q)$
generated by $f_i$ $(i\in I)$ and consider the natural projection
$\pi: \mathcal{F} \to U_q^-(\g)$ given by $f_i \mapsto f_i$ $(i
\in I)$. Then the  vector space $\Q(q) \seq(\alpha)$ can be
regarded as the dual space of $\mathcal{F}_\alpha := \pi^{-1}( U_q^-(\g)_\alpha
) $ for $\alpha \in Q^+$. Set
\begin{equation*}
\begin{aligned}
& K_{0}(R)_{\Q(q)} = \Q(q) \otimes_{\A} K_{0}(R), \quad
K_{0}(R(\alpha))_{\Q(q)} = \Q(q) \otimes_{\A} K_{0}(R(\alpha)), \\
& G_0(R)_{\Q(q)} = \Q(q) \otimes_\A G_0(R), \quad
G_{0}(R(\alpha))_{\Q(q)} = \Q(q) \otimes_{\A} G_{0}(R(\alpha)),
\end{aligned}
\end{equation*}
and denote by $\Phi_{\Q(q)}: U_q^{-}(\g) \longrightarrow
K_{0}(R)_{\Q(q)}$ the algebra homomorphism induced by $\Phi:
U_{\A}^{-}(\g) \longrightarrow K_{0}(R)$. Then $\qch$ is the dual map of $ \Phi_{\Q(q)} \circ \pi $,
which yields the following diagram:
$$
\xymatrix{
\mathcal{F}_\alpha  \ar[rr]^{\pi} \ar[d]^{\text{dual}} & & U_q^-(\g)_\alpha \ar[rr]^{\Phi_{\Q(q)}} & & \ar[d]^{\text{dual w.r.t. } (\ ,\ )} K_0(R(\alpha))_{\Q(q)} \\
\ar[u] \Q(q) \seq(\alpha)& & & & \ar[llll]_{\qch} \ar[u] G_0(R(\alpha))_{\Q(q)}
}
$$
Combining Theorem \ref{Thm:Phi is injective} with Proposition
\ref{Prop:character map is injective}, we conclude
$$ \Phi_{\Q(q)}: U_q^-(\g) \longrightarrow  K_0(R)_{\Q(q)} $$
is an isomorphism.

\begin{Thm} \label{Thm:iso of K0 and Uq}
The map $\Phi: U_{\A}^-(\g) \longrightarrow K_0(R) $ is an isomorphism if $a_{ii} \ne 0$ for all $i\in I$.
\end{Thm}
\begin{proof}
It suffices to show that $\Phi_{\Q(q)}(U_{\A}^-(\g)) = K_0(R)$.
Choose a sequence $(i_k)_{k \ge 0}$ of $I$ such that, for each $i\in
I$, $i$ appears infinitely many times in $(i_k)_{k \ge 0}$. Let
$B_\alpha$ be the set of all isomorphism classes of irreducible
$R(\alpha)$-modules. We fix a representative $S_b$ for each $b \in
B_\alpha$. To each $b \in B_\alpha$, we assign the sequence $p_b :=
p_0 p_1 \cdots$ given as follows: if $M_0:= S_b$, define
$$p_k = \ep_{i_{k}} (M_k)\ \text{ and } \ M_{k+1} = \ke_{i_{k}}^{p_k}(M_k)\quad (k \ge 0) $$
inductively. For $b \in B_\alpha$, let
$$ P_b = P_{\mathbf{i}_b}, $$
where $\mathbf{i}_b := ( i_0 ^{(p_0)} i_1 ^{(p_1)} \ldots)$. Note
that $P_{\mathbf{i}_b}$ is well-defined since $\mathbf{i}_b$ has
only finitely many nonnegative integers. \ Define a total order
$\prec$ on $B_\alpha$ by
$$ b \prec c \text{ if and only if } p_{b} <_{\rm lex} p_{c} , $$
where $<_{\rm lex}$ is the lexicographic order. Then it follows from the definition of the pairing $\eqref{eq:paring between K and G}$
that
$$ (P_{b}, S_{c}) = 0 \text{ if } b \succ c \quad \text{ and }\quad (P_{b}, S_{b}) = q^t $$
for some $t \in \Z$. Hence, any projective module $[P]$ in
$K_0(R(\alpha))$ can be written as an $\A$-linear combination of $\{
P_b \mid b \in B_\alpha \}$, which implies
$\Phi_{\Q(q)}(U_{\A}^-(\g)) = K_0(R)$.
\end{proof}

\vskip 3em

\section{Crystals and Perfect bases} \label{Sec:Crystals}

In this section, we develop the theory of perfect bases for
$U_q^-(\g)$ as a $B_q(\g)$-modules. We prove that the negative part
$U_q^-(\g)$ has a perfect basis by constructing the upper global
basis. We also show that the crystals arising from perfect bases of
$U_q^{-}(\g)$ are all isomorphic to the crystal $B(\infty)$.

\subsection{Crystals}\

We review the basic theory of abstract crystals for quantum
generalized Kac-Moody algebras introduced in \cite{JKKS07}.

\begin{Def}\label{Def: abstract crystal} An {\em abstract crystal} is a set $B$ together with the maps
$\  {\rm wt} : B \to P,\ \varphi_{i},\varepsilon_{i} : B \to \Z \sqcup \{-\infty \} \mbox{ and } \
\tilde{e}_i,\tilde{f}_i : B \to B \sqcup \{0\} \ (i \in I) $ satisfying the following conditions:
\begin{enumerate}
\item $\varphi_{i}(b) = \varepsilon_{i}(b) + \langle h_i,{\rm wt}(b) \rangle,$
\item ${\rm wt}(\tilde{e}_i b)={\rm wt}(b)+\alpha_i, {\rm wt}(\tilde{f}_i b)={\rm wt}(b)-\alpha_i \mbox{ if } \tilde{e}_i b, \tilde{f}_i b \in B,$
\item for $ b,b^{\prime} \in B $ and $ i \in I,$ $ b'=\tilde{e}_i b$ if and only if $b = \tilde{f}_i b',$
\item for $ b \in B$, if $ \varphi_{i}(b) = -\infty$, then $ \tilde{e}_i b = \tilde{f}_i b=0,$
\item if $ b \in B $ and $\tilde{e}_i b \in B$, then
         $$\varepsilon_{i}(\tilde{e}_i b)=\begin{cases} \varepsilon_{i}(b)-1 & \mbox{if} \ i \in \Ire, \\ \varepsilon_{i}(b) & \mbox{if} \ i \in \Iim, \end{cases}
       \ \  \varphi_{i}(\tilde{e}_i b)=\begin{cases} \varphi_{i}(b)+1 & \mbox{if} \ i \in \Ire, \\ \varphi_{i}(b)+a_{ii} & \mbox{if} \ i \in \Iim, \end{cases}$$
\item if $ b \in B$ and $\tilde{f}_i b \in B$, then
         $$\varepsilon_{i}(\tilde{f}_i b)=\begin{cases} \varepsilon_{i}(b)+1 & \mbox{if} \ i \in \Ire, \\ \varepsilon_{i}(b) & \mbox{if} \ i \in \Iim, \end{cases}
       \ \  \varphi_{i}(\tilde{f}_i b)=\begin{cases} \varphi_{i}(b)-1 & \mbox{if} \ i \in \Ire, \\ \varphi_{i}(b)-a_{ii} & \mbox{if} \ i \in \Iim. \end{cases}$$
\end{enumerate}
\end{Def}


\begin{Exa} \label{Exa: natural abstract crystals} \
\begin{enumerate}
\item For $b \in B(\infty)$, define
$ \wt, \varepsilon_i$, and $\varphi_i$ as follows:
\begin{align*}
\ \ & \wt(b)  =  -(\alpha_{i_1}+ \cdots + \alpha_{i_r}) \text{ for } b = \tilde{f}_{i_1} \cdots \tilde{f}_{i_r} {\bf 1}+ q L(\infty), \\
\ \ & \varepsilon_i(b)  =  \begin{cases} \max \{ k \ge 0 \ | \ \tilde{e}_i^k b\neq 0 \} & \text{ for } i \in \Ire, \\
                                       \qquad   0 & \text{ for } i \in \Iim, \end{cases} \\
\ \ & \varphi_i(b)  =   \varepsilon_i(b)+ \langle h_i, {\rm wt}(b)\rangle.
\end{align*}
Then $(B(\infty), \wt, \tilde{e}_i, \tilde{f}_i, \varepsilon_i,
\varphi_i)$ becomes an abstract crystal.
\item
For $b \in B(\lambda)$, define $\wt, \varepsilon_i$, and $\varphi_i$ as follows:
\begin{align*}
\ \ & \wt(b)  =  \lambda-(\alpha_{i_1}+ \cdots + \alpha_{i_r}) \text{ for } b = \tilde{f}_{i_1} \cdots \tilde{f}_{i_r}v_{\lambda}+ q L(\lambda), \\
\ \ & \varepsilon_i(b) =  \begin{cases} \max \{ k \ge 0 \ | \ \tilde{e}_i^k b\neq 0 \} & \text{ for } i \in \Ire,  \\
                                       \qquad   0  & \text{ for } i \in \Iim, \end{cases} \\
\ \ &  \varphi_i(b)  =  \begin{cases} \max \{ k \ge 0 \ | \ \tilde{f}_i^k b\neq 0 \} & \text{ for } i \in \Ire, \\
                                     \quad  \langle h_i, {\rm wt}(b)\rangle & \text{ for } i \in \Iim. \end{cases}
\end{align*}
Then $(B(\lambda),{\rm wt},\tilde{e}_i,\tilde{f}_i, \varepsilon_i,
\varphi_i)$ becomes an abstract crystal.

\item For $\lambda \in P$, let $T_{\lambda}=\{ t_{\lambda} \}$ and define
\begin{align*}
\ \ & {\rm wt}(t_{\lambda}) = \lambda, \quad \tilde{e}_i t_{\lambda} =\tilde{f}_i t_{\lambda} =0 \quad
 \varepsilon_i(t_{\lambda})=\varphi_i(t_{\lambda}) = - \infty \text{ for all } i \in I.
\end{align*}
Then $(T_{\lambda},{\rm wt},\tilde{e}_i,\tilde{f}_i, \varepsilon_i, \varphi_i)$ is an abstract crystal.
\item  Let $C= \{ c \}$ and define
\begin{align*}
\ \ &{\rm wt}(c) = 0, \quad \tilde{e}_i c =\tilde{f}_i c =0 \quad \varepsilon_i(c)=\varphi_i(c) = 0 \text{ for all } i \in I.
\end{align*}
Then $(C,{\rm wt},\tilde{e}_i,\tilde{f}_i, \varepsilon_i, \varphi_i)$ is an abstract crystal.
\end{enumerate}
\end{Exa}

\begin{Def}\label{Def: crystal morphism} \
 \begin{enumerate}
 \item A {\it crystal morphism} $\phi$ between abstract crystals $B_1$ and $B_2$ is a map
 from $ B_1 $ to $ B_2 \sqcup \{0\}$ satisfying the following conditions:
\begin{enumerate}
\item if $b \in B_1$ and $\phi(b) \in B_2$, then $\wt(\phi(b))= \wt(b),\ \varepsilon_{i}(\phi(b))=\varepsilon_{i}(b)$ and $\varphi_{i}(\phi(b))=\varphi_{i}(b)$,
\item if $b \in B_1$ and $ i\in I$ with $\tilde{f}_i b \in B_1$, then we have $\tilde{f}_i \phi(b) = \phi(\tilde{f}_i b).$
\end{enumerate}
\item A crystal morphism $\phi: B_1 \to B_2$ is called {\it strict} if
$$ \phi(\tilde{e}_i b) = \tilde{e}_i \phi(b)\quad \text{ and }\quad \phi(\tilde{f}_i b) = \tilde{f}_i \phi(b) $$
for all $i\in I$ and $b \in B_1$.
\end{enumerate}
\end{Def}

The tensor product of two crystals is defined as follows: for given two crystals $B_1$ and $B_2$, their tensor product
 $B_1 \otimes B_2$ is the set $\{ b_1 \otimes b_2 \mid b_1 \in B_1,  b_2 \in B_2 \}$ with the maps ${\rm wt}, \varepsilon_i
 ,\varphi_i,\tilde{e}_i$ and $\tilde{f}_i$ given by
\begin{equation} \label{Eq:def of tensor product}
\begin{aligned}
& {\rm wt}(b_1 \otimes b_2) = {\rm wt}(b_1) \otimes {\rm wt}(b_2), \\
&\varepsilon_i(b_1 \otimes b_2)= \max \{ \varepsilon_i(b_1), \varepsilon_i(b_2)-\langle h_i, {\rm wt}(b_1) \rangle \}, \\
&\varphi_i(b_1 \otimes b_2)= \max \{ \varphi_i(b_1)+\langle h_i, {\rm wt}(b_2) \rangle, \varphi_i(b_2) \}, \\
& \tilde{f}_i(b_1 \otimes b_2)= \begin{cases} \tilde{f}_i(b_1) \otimes b_2  & \text{ if } \varphi_i(b_1)> \varepsilon_i(b_2), \\
                                            b_1 \otimes \tilde{f}_i(b_2) & \text{ if } \varphi_i(b_1)\le \varepsilon_i(b_2), \end{cases} \\
& \text{ for } i \in \Ire,\ \ \tilde{e}_i(b_1 \otimes b_2) =\begin{cases} \tilde{e}_i(b_1) \otimes b_2  & \text{ if } \varphi_i(b_1) \ge \varepsilon_i(b_2), \\
                                            b_1 \otimes \tilde{e}_i(b_2) & \text{ if } \varphi_i(b_1) < \varepsilon_i(b_2), \end{cases} \\
& \text{ for } i \in \Iim,\ \
      \tilde{e}_i(b_1 \otimes b_2) =\begin{cases} \tilde{e}_i(b_1) \otimes b_2  & \text{ if } \varphi_i(b_1) > \varepsilon_i(b_2)-a_{ii}, \\
                                                  0 & \text{ if } \varepsilon_i(b_2) < \varphi_i(b_1) \le \varepsilon_i(b_2)-a_{ii},\\
                                                  b_1 \otimes \tilde{e}_i(b_2) & \text{ if } \varphi_i(b_1) \le \varepsilon_i(b_2). \end{cases}
\end{aligned}
\end{equation}

It was proved in \cite[Lemma 3.10]{JKKS07} that $B_1 \otimes B_2$
becomes an abstract crystal. Moreover, they proved the {\it
recognition theorem} of $B(\lambda)$ $(\lambda \in P^+)$ using the
abstract crystal structure of $B(\infty)$.

\begin{Prop} \cite[Theorem 5.2]{JKKS07} \label{Prop: recognition theorem of B-lambda}
For $\lambda \in P^+$, the crystal $B(\lambda)$ is isomorphic to the connected component of
$B(\infty) \otimes T_{\lambda} \otimes C$ containing ${\bf 1} \otimes t_{\lambda} \otimes c$.
\end{Prop}

\vskip 1em

\subsection{Perfect bases}\ \label{Sec:perfect bases}

We revisit the algebra $U^{-}_q(\g)$. We analyze $U^{-}_q(\g)$ as a
$B_q(\g)$-module and develop the perfect basis theory for
$U_q^{-}(\g)$. The crystal structure is revealed when $e_i'$ acts on
a perfect basis.

Let
\begin{equation*}
 \begin{aligned}
 \ \ \bse_i'^{(n)} =
 \begin{cases}
 (\bse_i')^{n}  & \text{if } i \in \Ire, \\
 \dfrac{(\bse_i')^{n}}{\{n\}_{i}!} & \text{if } i \in \Iim.
 \end{cases}
 \end{aligned}
\end{equation*}
Then we obtain the following commutation relations:
\begin{equation} \label{eq:commutation relation}
 \begin{aligned}
 \bse_i'^{(n)}\bsf_j^{(m)} =
  \begin{cases}
  \displaystyle \sum_{k=0}^{n} q_i^{-2nm+(n+m)k-k(k-1)/2} \left[\begin{matrix} n \\ k\\ \end{matrix} \right]_i
                 \bsf_i^{(m-k)} \bse_i'^{(n-k)} & \text{if } i=j \text{ and } i \in \Ire,\\
  \displaystyle \sum_{k=0}^{m} q_i^{-c_i(-2nm+(n+m)k-k(k-1)/2)} \left \{ \begin{matrix} m \\ k\\ \end{matrix} \right \}_i
                 \bsf_i^{(m-k)} \bse_i'^{(n-k)} & \text{if } i=j \text{ and } i \in \Iim,\\
  q_i^{-nm a_{ij}} \bsf_j^{(m)}\bse_i'^{(n)} & \text{if } i \neq j.
  \end{cases}
 \end{aligned}
\end{equation}

For $i \in I$ and $v \in U^{-}_q(\g)$, let
$$\ell_i(v) = \min \{n \in \Z_{\ge 0} \ \mid \ \bse_i'^{n+1}v=0 \}.$$
Note that $\ell_i$ is well-defined since $\bse_i'$ is locally nilpotent (see $\eqref{eq: special commute}$).
Then, for $i \in I$ and $k \in \Z_{\ge 0}$,
$$U^{-}_q(\g)_{i}^{< k} :=\{ v \in U^{-}_q(\g) \mid \ell_i(v) < k \}$$
becomes a $\Q(q)$-vector space.

\begin{Def} \label{Def:perfect bases}
A basis $B$ of $U^{-}_q(\g)$ is said to be {\it perfect} if
\begin{enumerate}
\item $B=\displaystyle\bigsqcup_{\mu\in Q^-} B_{\mu}$, where $B_{\mu}:= B \cap U^{-}_q(\g)_{\mu}$,
\item for any $b \in B$ and $i \in I$ with $\bse_i'(b)\neq 0$, there exists a unique $\mathsf{e}_i(b) \in B$ such that
\begin{equation} \label{eq: perfect basis}
\bse_i'(b) \in c \ \mathsf{e}_i(b) + U^{-}_q(\g)_{i}^{< \ell_i(b)-1} \text{ for some } c \in \Q(q)^{\times},
\end{equation}
\item if $\mathsf{e}_i(b)=\mathsf{e}_i(b')$ for $b,b' \in B$, then $b = b'\ ( i \in I)$.
\end{enumerate}
\end{Def}

Now, we define the {\it upper Kashiwara operators} for the $B_q(\g)$-module $U^{-}_q(\g)$.
Let $u \in U^{-}_q(\g)$ such that $\bse_i' u=0$.
Then, for $n \in \Z_{\ge0}$, we define the upper Kashiwara operators
$ \tilde{E}_i, \tilde{F}_i $ by
\begin{equation*}
 \begin{aligned}
 \ \ &\tilde{E}_i(f_i^{(n)}u)=
 \begin{cases}
 \displaystyle\frac{q_i^{-(n-1)}}{[n]_i} f_i^{(n-1)}u  & \text{ if } i \in \Ire, \\
 \{n\}_i q_i^{c_i(n-1)} f_i^{(n-1)}u & \text{ if } i \in \Iim ,
 \end{cases} \\
 \ \ &\tilde{F}_i(f_i^{(n)}u)=
 \begin{cases}
 \displaystyle q_i^{n}[n+1]_i f_i^{(n+1)}u & \text{ if } i \in \Ire , \\
 \displaystyle\frac{1}{\{n+1\}_i q_i^{c_i n}} f_i^{(n+1)}u & \text{ if } i \in \Iim.
 \end{cases}
 \end{aligned}
 \end{equation*}
From the $i$-string decomposition $\eqref{eq: lowerpart i-string decomposition}$,
the upper Kashiwara operators $ \tilde{E}_i$ and $\tilde{F}_i $
can be extended to the whole space $U^{-}_q(\g)$ by linearity.

\begin{Def} \label{Def:upper crystal}
An {\it upper crystal basis} of $U^{-}_q(\g)$ is a pair
$(L^{\vee},B^{\vee})$ satisfying the following conditions:
\begin{enumerate}
\item $L^{\vee}$ is a free $\A_0$-module of $U^{-}_q(\g)$ such that $U^{-}_q(\g)=\Q(q) \otimes_{\A_0}L^{\vee}$ and
      $L^{\vee} = \bigoplus_{\alpha \in Q^+} L^{\vee}_{-\alpha}$, where $L^{\vee}_{-\alpha} := L^{\vee} \cap
      U_q^{-}(\g)_{-\alpha}$,
\item $B^{\vee}$ is a $\Q$-basis of $L^{\vee}/ q L^{\vee}$ such that $B^{\vee} = \bigsqcup_{\alpha \in Q^+} B^{\vee}_{-\alpha}$,
            where $B^{\vee}_{-\alpha} := B^{\vee} \cap (L^{\vee}_{-\alpha}/ q
            L^{\vee}_{-\alpha})$,
\item $\tilde{E}_i B^{\vee} \subset B^{\vee} \sqcup \{0\}$, \ \
             $\tilde{F}_i B^{\vee} \subset B^{\vee} $ for all $i \in I$,
\item For $ b,b'\in B^{\vee}$ and $  i \in I$, $b'^{\vee} = \tilde{F}_i b^{\vee}$
            if and only if $b^{\vee} = \tilde{E}_i b'^{\vee}$.
\end{enumerate}
\end{Def}

We have the following lemma
which is the $U^{-}_q(\g)$-version of \cite[Lemma 4.3]{KOP09}.
\begin{Lem} \label{Lem: properties of the bilnear form}
For any  $u, v \in U^{-}_q(\g)$, we have
$$(\tilde{f}_i u,  v)_K=(u,\tilde{E}_i v)_K, \quad
( \tilde{e}_i u, v)_K =(u, \tilde{F}_i v)_K.$$
\end{Lem}

\begin{Lem} \label{Lem: relatoin between eiip and EEi}
Let $u \in U^{-}_q(\g)$, and $n$ be the smallest integer such that $\bse_i'^{n+1}u=0$.
Then we have
\begin{equation*}
\bse_i'^{n}u =
 \begin{cases} [n]_i! \tilde{E}_i^{n} u  & \text{ if } i \in \Ire, \\
               \tilde{E}_i^{n}u & \text{ if } i \in \Iim. \end{cases}
\end{equation*}
\end{Lem}

\begin{proof}
For $u \in U_q^{-}(\g)$ and $i \in I$, consider the $i$-string
decomposition: $u = \sum_{l=0}^{n}f_i^{(l)}u_l$, where $\bse_i'
u_l=0$. If $i \in \Ire$, then by $\eqref{eq:commutation relation}$
and the definition of $\tilde{E}_i$, we have
$$\bse_i'^{n}u=q_i^{- n(n-1)/2}u_n, \ \  \tilde{E}_i^{n}u=\dfrac{q_i^{- n(n-1)/2}}{[n]_i!}u_n.$$
Similarly, if $i \in \Iim$, we obtain
$$\bse_i'^{(n)}u=q_i^{c_i n(n-1)/2}u_n, \ \  \tilde{E}_i^{n}u=\{n\}_i! q_i^{c_i n(n-1)/2}u_n,$$
which proves our assertion.
\end{proof}

Let $(L(\infty), B(\infty))$ be the lower crystal basis of $U^{-}_q(\g)$. Set
$$L(\infty)^{\vee} = \{ u \in U^{-}_q(\g) \mid (u, L(\infty))_K \subset \A_{0} \}.$$
We also denote by $( \ , \ )_K : L(\infty)^{\vee} / q L(\infty)^{\vee} \times L(\infty) /
qL(\infty) \rightarrow \Q$ the nondegenerate bilinear form induced by the
bilinear form $( \ , \ )_K$ on $U^{-}_q(\g)$. Let
 $$B(\infty)^{\vee} = \{b^{\vee} \mid b \in B(\infty) \}$$
be the $\Q$-basis of $L(\infty)^{\vee} / q L(\infty)^{\vee}$ which is dual
to $B(\infty)$ with respect to $( \ , \ )_K$.

\begin{Prop} \label{prop:upper}
The pair $(L(\infty)^{\vee}, B(\infty)^{\vee})$ is an upper crystal basis of $U^{-}_q(\g)$.
\end{Prop}
\begin{proof}
The proof is almost the same as in \cite{Kash93b}.
\end{proof}

Let $\A_{\infty}$ be the subring of $\Q(q)$ consisting of regular
functions at $\infty$. Let $U_{\A}$ (resp.\ $L$ and $L^{-}$) be an
$\A$-subalgebra (resp.\ $\A_0$-subalgebra and
$\A_{\infty}$-subalgebra) of $U^{-}_q(\g)$.

\begin{Def}
A triple $(U_{\A}, L, L^{-})$ is a {\it balanced triple} if
\begin{enumerate}
\item $U^{-}_q(\g) \cong \Q(q) \otimes_{\A} U_{\A} \cong \Q(q) \otimes_{\A_0}
L \cong \Q(q) \otimes_{\A_{\infty}} L^{-}$ as $\Q(q)$-vector
spaces,
\item the natural $\Q$-linear map $E \to L/qL$ is an
isomorphism, where $ E := U_{\A} \cap L \cap L^{-}$.
\end{enumerate}
\end{Def}
It was shown in \cite{Kash91} that the condition (2) is equivalent
to saying that there are natural isomorphisms $U_\A \cong \A
\otimes_{\Q} E$, \ $L \cong \A_0 \otimes_{\Q} E$, \ $L^{-} \cong
\A_{\infty} \otimes_{\Q} E$.

Let $U^{0}_{\A}(\g)$ be the $\A$-subalgebra of $U_q(\g)$ generated
by $q^{h}$, $\prod^{m}_{k=1}  \dfrac{1-q^k q^h}{1-q^k}$ for all $m \in
\Z_{\ge 0}$, $h \in P^{\vee}$
and let $U_{\A}(\g)$ be the $\A$-algebra generated by
$U^{0}_{\A}(\g)$, $U^{+}_{\A}(\g)$ and $U^{-}_{\A}(\g)$.

\begin{Prop} [\cite{JKK05}]
$(U_{\A}^{-}(\g),L(\infty),L(\infty)^{-})$ is a balanced triple for $U^{-}_q(\g)$.
\end{Prop}

Recall the $\Q(q)$-algebra automorphism $\bar {}  : U^{-}_q(\g) \to U^{-}_q(\g)$ given in $\eqref{Eq:bar involution}$.
Define
\begin{equation*}
\begin{aligned}
U^{-}_{\A}(\g)^{\vee} & = \{ u \in U^{-}_q(\g) \mid (u,U^{-}_{\A}(\g))_K \subset \A \}, \\
L(\infty)^{\vee} & = \{ u \in U^{-}_q(\g)  \mid (u,L(\infty))_K \subset \A_{0} \}, \\
\overline{L(\infty)}^{\vee} & = \{ u \in U^{-}_q(\g)  \mid (u,L(\infty)^{-})_K \subset \A_{\infty} \}.
\end{aligned}
\end{equation*}
By the same argument as in \cite{Kash93b}, one can verify that
$(U^{-}_{\A}(\g)^{\vee},L(\infty)^{\vee},\overline{L(\infty)}^{\vee})$ is a balanced triple for
$U^{-}_q(\g)$. Hence there is a natural isomorphism
$$E^{\vee}:=U^{-}_{\A}(\g)^{\vee} \cap L(\infty)^{\vee} \cap
\overline{L(\infty)}^{\vee} \overset{\sim} \longrightarrow L(\infty)^{\vee} / q L(\infty)^{\vee}.$$
Let $G^{\vee}$ denote the inverse of this isomorphism and set
$${\mathbb B}(\infty) =\{ G^{\vee}(b^{\vee}) \mid b^{\vee} \in B(\infty)^{\vee} \}.$$

\begin{Lem} \label{lem: Key formula}
Let $b^{\vee} \in L(\infty)^{\vee} / q L(\infty)^{\vee}$ and $n \in \Z_{\ge 0}$.
\begin{enumerate}

\item If ${\tilde{E}_i}^{n+1}b^{\vee}=0$, then
${\bse_i'}^{n}G(b^{\vee})=
\begin{cases}
[n]_i!G^{\vee}({\tilde{E}_i}^{n}b^{\vee}) & \text{ if } i \in \Ire,\\
 G^{\vee}({\tilde{E}_i}^{n}b^{\vee}) & \text{ if }i \in \Iim. \end{cases}$

\item ${\bse_i'}^{n+1}G^{\vee}(b^{\vee})=0$ if and only if ${\tilde{E}_i}^{n+1} b^{\vee}=0$.

\end{enumerate}
\end{Lem}

\begin{proof}

We first prove the assertion (1). Let $i \in \Ire$. Since
$\varphi(\dfrac{1}{[n]_i!} {\bse_i'}^{n})=\bsf_i^{(n)}$, by Lemma
\ref{Lem: relatoin between eiip and EEi}, we obtain
$$\dfrac{1}{[n]_i!} {\bse_i'}^{n}G^{\vee}(b^{\vee})={\tilde{E}_i}^{n}G^{\vee}(b^{\vee})\in U^{-}_{\A}(\g)^{\vee} \cap L(\infty)^{\vee} \cap
\overline{L(\infty)}^{\vee},$$ which yields $\dfrac{1}{[n]_i!}
{\bse_i'}^{n}G^{\vee}(b^{\vee})=G^{\vee}({\tilde{E}_i}^{n}b^{\vee})$.

Similarly, for $i \in \Iim$, it follows from $\varphi({\bse_i'}^{n})=\bsf_i^{(n)}$ that
$${\bse_i'}^{n}G^{\vee}(b^{\vee})={\tilde{E}_i}^{n}G^{\vee}(b^{\vee})\in U^{-}_{\A}(\g)^{\vee} \cap L(\infty)^{\vee} \cap
\overline{L(\infty)}^{\vee}.$$ Thus we have
${\bse_i'}^{n}G^{\vee}(b^{\vee})=G^{\vee}({\tilde{E}_i}^{n}b^{\vee})$.

For the assertion (2), it is obvious that
${\bse_i'}^{n+1}G^{\vee}(b^{\vee})=0$ implies
${\tilde{E}_i}^{n+1}b^{\vee}=0$. To prove the converse, suppose
${\bse_i'}^{n+1}G^{\vee}(b^{\vee})\neq 0$ and take the smallest $m
> n$ such that ${\bse_i'}^{m+1}G^{\vee}(b^{\vee})=0$. By (1), we
have
\begin{equation*}
\begin{aligned}
{\bse_i'}^{m}G^{\vee}(b^{\vee})=
\begin{cases}
[m]_i!G^{\vee}({\tilde{E}_i}^{m}b^{\vee})=0,& \ \text{ if } i \in \Ire, \\
G^{\vee}({\tilde{E}_i}^{m}b^{\vee})=0, & \ \text{ if } i \in \Iim,
\end{cases}
\end{aligned}
\end{equation*}
which is a contradiction to the choice of $m$. Hence we conclude
${\bse_i'}^{n+1}G^{\vee}(b^{\vee})=0$.
\end{proof}

For $b^{\vee}\in B(\infty)^{\vee}$, we define
\begin{align*}
\varepsilon_i^{\rm or}(b^{\vee}) &= \min \{ n \in \Z_{\ge 0}  \mid  {\tilde{E}_i}^{n+1}b^{\vee}=0 \}, \\
 \varphi_i^{\rm or}(b^{\vee}) &= \min \{ n \in \Z_{\ge 0}  \mid  {\tilde{F}_i}^{n+1}b^{\vee}=0 \}.
\end{align*}

\begin{Prop} \label{Prop:the e_i action global basis}
For $b^{\vee} \in B(\infty)^{\vee}$, we have
\begin{equation*}
\begin{aligned}
{\bse_i'} G^{\vee}(b^{\vee}) & =
\begin{cases}  \displaystyle [\varepsilon_i^{{\rm or}}(b^{\vee})]_i G^{\vee} ({\tilde{E}_i} b^{\vee})
 +\sum_{\varepsilon_i^{{\rm or}}(b'^{\vee}) < \varepsilon_i^{{\rm or}}(b^{\vee})-1} E_{b^{\vee},b'^{\vee}}^{i} G^{\vee} (b'^{\vee}) & \
                \text{if} \ i \in \Ire,   \\
               \displaystyle G^{\vee}({\tilde{E}_i} b^{\vee})
               +\sum_{\varepsilon_i^{\rm or}(b'^{\vee}) < \varepsilon_i^{\rm or}(b^{\vee})-1} E_{b^{\vee}, b'^{\vee}}^{i} G^{\vee} (b'^{\vee}) & \ \text{if} \ i \in \Iim,
\end{cases} \\
{\bsf_i} G^{\vee}(b^{\vee}) & =
\begin{cases}
              \displaystyle q_i^{-\varepsilon_i^{{\rm or}}(b^{\vee})} G^{\vee}({\tilde{F}_i} b^{\vee})+ \sum_{\varepsilon_i^{{\rm or}}(b'^{\vee}) \le \varepsilon_i^{{\rm or}}(b^{\vee})} F^{i}_{b^{\vee},b'^{\vee}}
              G^{\vee}(b'^{\vee}) & \ \text{if} \ i \in \Ire, \\
              \displaystyle \{\varepsilon_i^{\rm or}(b^{\vee})+1\}_i q_i^{c_i(\varepsilon_i^{\rm or}(b^{\vee})+1)}G^{\vee}({\tilde{F}_i} b^{\vee})
                + \sum_{\varepsilon_i^{\rm or}(b'^{\vee}) \le \varepsilon_i^{\rm or}(b^{\vee})} F^{i}_{b^{\vee},b'^{\vee}} G^{\vee} (b'^{\vee}) & \ \text{if} \ i \in \Iim. \end{cases}
\end{aligned}
\end{equation*}
for some $E_{b^{\vee}, b'^{\vee}}^{i},  F^{i}_{b^{\vee},b'^{\vee}} \in \Q(q)$.
\end{Prop}

\begin{proof}
If $i \in \Ire$, our assertions were proved in \cite{Kash93b}. We
will prove the case when $i \in \Iim$. Set $n=\varepsilon_i^{{\rm
or}}(b^{\vee})$. By Lemma \ref{lem: Key formula} and Definition
\ref{Def:upper crystal} (4), we have
$$ {\bse_i'}^{n} G^{\vee}(b^{\vee}) = G^{\vee}({\tilde{E}_i}^{n}b^{\vee})
= G^{\vee}({\tilde{E}_i}^{n-1}{\tilde{E}_i}
b^{\vee})={\bse_i'}^{n-1}G^{\vee}({\tilde{E}_i} b^{\vee}),$$ which
implies
$${\bse_i'} G^{\vee}(b^{\vee})-G^{\vee}({\tilde{E}_i} b^{\vee}) \in {\rm Ker}({\bse_i'}^{n-1}).$$

Using the equation $\eqref{eq:commutation relation}$, we get
$${\bse_i'}^{(n+1)}\bsf_i G^{\vee}(b^{\vee})=(q_i^{2c_i(n+1)}\bsf_i{\bse_i'}^{(n+1)}+q_i^{c_i(n+1)}{\bse_i'}^{(n)})G^{\vee}(b^{\vee}).$$
Hence Lemma \ref{lem: Key formula} yields
$${\bse_i'}^{(n+1)}\bsf_i G^{\vee}(b^{\vee})=\dfrac{1}{\{n\}_i!}q_i^{c_i(n+1)}{\bse_i'}^{n}G^{\vee}(b^{\vee})=\dfrac{1}{\{n\}_i!}q_i^{c_i(n+1)}G^{\vee}({\tilde{E}_i}^{n} b^{\vee}).$$
Using Lemma \ref{lem: Key formula} again, we obtain
$$\dfrac{1}{\{n\}_i!}q_i^{c_i(n+1)}G^{\vee}({\tilde{E}_i}^{n+1}{\tilde{F}_i} b^{\vee})=\dfrac{1}{\{n\}_i!}q_i^{c_i(n+1)}{\bse_i'}^{n+1}G^{\vee}({\tilde{F}_i} b^{\vee})
=\{n+1\}_i q_i^{c_i(n+1)}{\bse_i'}^{(n+1)}G^{\vee}({\tilde{F}_i}
b^{\vee}).$$ Thus we have
$$\bsf_iG^{\vee}(b^{\vee})-\{n+1\}_i q_i^{c_i(n+1)}G^{\vee}({\tilde{F}_i} b^{\vee}) \in {\rm Ker}({\bse_i'}^{n+1})$$
as desired.
\end{proof}

Combining Proposition \ref{prop:upper} and Proposition
\ref{Prop:the e_i action global basis}, we obtain the existence of
perfect basis for $U_q^{-}(\g)$.

\begin{Prop} \label{Cor perfect basis B(infty)}
${\mathbb B}(\infty)$ is a perfect basis of the $B_q(\g)$-module
$U^{-}_q(\g) $.
\end{Prop}

Let $B$ be a perfect basis of $U^{-}_q(\g)$. For $b \in B$,
define ${\rm wt}(b)= \mu$ if $b \in B_\mu$  and
\begin{equation*}
\begin{aligned}
& \mathsf{f}_i (b) = \begin{cases}  b' \ \ & \text{if} \
\mathsf{e}_i (b') = b, \\
0 \ \ & \text{otherwise}, \end{cases} \quad
\varepsilon_i(b) =
\begin{cases}
\ell_i(b) \ \ & \text{if} \ i\in \Ire, \\
0 \ \ & \text{if} \ i \in \Iim,
\end{cases}   \\
& \varphi_i(b) = \varepsilon_i(b) + \langle h_i, \wt(b) \rangle.
\end{aligned}
\end{equation*}
Then it is straightforward to verify that $(B, \wt,\mathsf{e}_i,
\mathsf{f}_i, \varepsilon_i, \varphi_i)$ is an abstract crystal.
The graph obtained from the crystal $(B, \wt,\mathsf{e}_i,
\mathsf{f}_i, \varepsilon_i, \varphi_i)$ is called a {\it perfect
graph} of $U^{-}_q(\g)$. The following proposition asserts that
the perfect basis ${\mathbb B}(\infty)$ yields the crystal
$B(\infty)$.

\begin{Prop} \label{prop:B-infinities}
There exist crystal isomorphisms
$${\mathbb B}(\infty) \cong B(\infty)^{\vee} \cong B(\infty).$$
\end{Prop}

\begin{proof}
Let $\vee : B(\infty) \to B(\infty)^{\vee}$ defined by $b \mapsto b^{\vee}$.
Then
$$ \tilde{f}_i b =b' \iff (\tilde{f}_i b,b'^{\vee})_K =1 \iff (b,\tilde{E}_i b'^{\vee})_K=1
\iff b^{\vee} = \tilde{E}_ib'^{\vee} \iff \tilde{F}_i b^{\vee} =b'^{\vee}.$$
Hence we have $B(\infty)^{\vee} \cong B(\infty)$ from Lemma \ref{Lem:nondegenerate pairing in GKM} and Lemma \ref{Lem: properties of the bilnear form}.

By Proposition \ref{Prop:the e_i action global basis}, we have
$$ \tilde{E}_i b^{\vee} =b'^{\vee} \iff
\mathsf{e}_i G^{\vee}(b^{\vee}) = G^{\vee}(\tilde{E}_i b^{\vee}) =
G^{\vee}(b'^{\vee}).$$ Hence the map $G^{\vee}$ gives a crystal
isomorphism between ${\mathbb B}(\infty)$ and $B(\infty)^{\vee}$.
\end{proof}

In the rest of this section, we will show that the perfect graph
arising from any perfect basis of $U_q^{-}(\g)$ is isomorphic to
the crystal $B(\infty)$. Our argument follows the outline given in
\cite[Section 6]{KOP09}.

Let $B$ be a perfect basis of $U^{-}_q(\g)$.
For each sequence ${\bf i}=(i_1,\dots,i_m) \in I^{m}$ $(m \ge 1)$, we
define a binary relation $\preceq_{{\bf i}}$ on $U^{-}_q(\g) \setminus \{0\} $ as
follows:
 \begin{equation*}
 \begin{aligned}
   \ \ & \mbox{ if } {\bf i}=(i), \mbox{ then } \quad  v \preceq_{{\bf i}} v' \Leftrightarrow \ell_{i}(v) \le \ell_{i}(v^{\prime}),\\
   \ \ & \mbox{ if } {\bf i}=(i;{\bf i}'), \mbox{ then }\quad v \preceq_{{\bf i}} v' \Leftrightarrow  \begin{cases} \ell_{i}(v) < \ell_{i}(v^{\prime}) \mbox{ or} \\
             \ell_{i}(v)=\ell_{i}(v^{\prime}), \bse_i'^{\ell_{i}(v)}(v) \preceq_{{\bf i}'} \bse_i'^{\ell_{i}(v^{\prime})}(v^{\prime}). \end{cases}
 \end{aligned}
  \end{equation*}
We write $v \equiv_{{\bf i}} v'$ if $v \preceq_{{\bf i}} v'$ and $v' \preceq_{{\bf i}} v$.
 For a given ${\bf i}=(i_1,\dots,i_m) \in I^{m}$, define the maps $\bse'^{ top}_{\bf i}:U^{-}_q(\g) \to U^{-}_q(\g)$ and
 ${{\mathsf{e}_{{\bf i}}}^{top}}:B \to B\sqcup \{ 0 \} $ as follows:
 \begin{eqnarray*}
 \bse'^{top}_i(v)  = \bse_i'^{\ell_{i}(v)}(v) \mbox{ for } m=1\quad & \mbox{ and } & \quad
 \bse'^{top}_{{\bf i}} = \bse'^{top}_{i_m}\circ \cdots \circ \bse'^{top}_{i_1} \mbox { for } m >1, \\
 \mathsf{e}_i^{top}(b)  = \mathsf{e}_i^{\ell_{i}(b)}(b) \mbox{ for } m=1 \quad & \mbox{ and } &
 \quad {{\mathsf{e}_{{\bf i}}}^{top}} = \mathsf{e}^{top}_{i_m}\circ \cdots \circ
 \mathsf{e}^{top}_{i_1}\mbox { for } m >1.
 \end{eqnarray*}
By Proposition \ref{Prop:Highest vector 1}, we identify $\Q(q) $
with $ \{ v \in U^{-}_q(\g) \mid \bse_i'(v)=0 \mbox{ for all } i
\in I\}$. Note that $ \Q(q) \cap B =\{ {\bf 1} \}$. For each $v
\in U^{-}_q(\g)$, there exists a sequence ${\bf i}$ such that
$\bse'^{top}_{\bf i}(v) \in \Q(q)$. From $\eqref{eq: perfect
basis}$, one can check that the following statements hold.

\begin{Lem} \label{Lem: realtions between top operators}
For any sequence ${\bf i}=(i_1,\dots,i_m) \in I^{m}\ ( m\ge 1)$, we have
\begin{itemize}
  \item[(1)] $\bse'^{top}_{{\bf i}}(b) \in \Q(q)^{\times} {{\mathsf{e}_{{\bf i}}}^{top}}(b)$ for any $b\in B$,
  \item[(2)] if $\bse'^{top}_{{\bf i}}(b) \in \Q(q)^{\times}$ for some $b\in B$, then ${{\mathsf{e}_{{\bf i}}}^{top}}(b)\in \Q(q)^{\times}$,
  \item[(3)] if $b \equiv_{\bf i} b'$ and ${{\mathsf{e}_{{\bf i}}}^{top}}(b)={{\mathsf{e}_{{\bf i}}}^{top}}(b')$, then $b=b'$ for all $ b,b' \in B$.
\end{itemize}
\end{Lem}

\begin{Def}
Let $B,B'$ be perfect bases of $U^{-}_q(\g)$. A {\em perfect
morphism} $[\phi,\tilde{\phi},c]:(U^{-}_q(\g),B) \to
(U^{-}_q(\g),B')$ is a triple $(\phi,\tilde{\phi},c)$, where
\begin{itemize}
\item[(1)] $\phi:U^{-}_q(\g) \to U^{-}_q(\g)$ is a $B(\g)$-module endomorphism such that $0 \notin \phi(B)$,
\item[(2)] $\tilde{\phi}: B \to B^{\prime}$ is a map satisfying  $\tilde{\phi}({\bf 1})=\phi({\bf 1})$,
\item[(3)] $c: B\setminus \{ {\bf 1} \} \to \Q(q)^{\times}$ is a map satisfying
\begin{align*}
  \phi(b)-c(b)\tilde{\phi}(b) \prec_{{\bf i}} \phi(b)
\end{align*}
for $b\in B \setminus \{ {\bf 1} \}$ and ${\bf i}=(i_1,\dots,i_m)$ such that $\bse'^{top}_{\bf i}(b)\in \Q(q)$.
\end{itemize}
\end{Def}

\begin{Lem} Let $\phi$ be a $B_q(\g)$-endomorphism of $U_q^{-}(\g)$.
\begin{enumerate}
\item If a perfect morphism $[\phi,\tilde{\phi},c]$ exists, then $\tilde{\phi}$ and $c$ are uniquely determined.
\item For a given perfect morphism $[\phi,\tilde{\phi},c]: (U_q^{-}(\g),B) \to (U_q^{-}(\g),B')$, the map $\tilde{\phi}$ is a crystal morphism.
\end{enumerate}
\end{Lem}

\begin{proof}
This lemma is essentially the same as \cite[Lemma 6.3, Lemma
6.4]{KOP09}. However, since our algebra $U_q^{-}(\g)$ is
considered as a $B_q(\g)$-module, Proposition \ref{Prop:Highest
vector 1} plays a key role in proving this lemma. Then our
assertions follow by a similar argument in \cite{KOP09}.
\end{proof}

Now we state and prove the main result of this section.

\begin{Thm} \label{Thm: uniqueness of perfect graphs}
 Let $B$ and $B^{\prime}$ be two perfect bases of $U^{-}_q(\g)$.
 Then the identity map ${\rm id} : U^{-}_q(\g) \to U^{-}_q(\g) $ induces a perfect
 isomorphism from $(U^{-}_q(\g) ,B)$ to $(U^{-}_q(\g),B')$. That is,
there exists a unique crystal isomorphism $\tilde{\phi}: B \to B^{\prime}$
and a unique map $c: B\setminus \{ {\bf 1} \} \to \Q(q)^{\times}$ satisfying
$\tilde{\phi}( {\bf 1} )= {\bf 1} $ and $$ b-c(b)\tilde{\phi}(b)\prec_{{\bf i}} b $$
for each $b \in B\setminus \{ {\bf 1} \}$ and any sequence ${\bf i}=(i_1,\dots,i_m)$ with
${{\mathsf{e}_{{\bf i}}}^{top}}(b)= {\bf 1} $.
\end{Thm}

\begin{proof}
Since the proof is almost the same as \cite[Theorem 6.6]{KOP09},
we only give a sketch of proof. By a similar argument in
\cite[Lemma 6.5]{KOP09}, for a given $b \in B \setminus \{ {\bf 1}
\}$, one can show that there exist unique $b' \in B'$, $v \in
U_q^{-}(\g)$ and $k \in \Q(q)^{\times}$ satisfying
\begin{align*}
\ \ & (1) \ b \equiv_{{\bf i}} b', \quad (2) \ b = v + kb', \quad
(3) \ v =0 \text{ or } v \prec_{{\bf i}} b, \ v \prec_{{\bf i}} b'
\end{align*}
for any sequence ${\bf i}$ with $\bse'^{top}_{{\bf i}}(b) \in
\Q(q)^{\times}$. Then the maps ${\rm id}: U_q^{-}(\g) \to
U_q^{-}(\g)$,  $\tilde{\phi}:B \to B'$ and $c: B \setminus \{ {\bf
1} \}\to \Q(q)^{\times}$ defined by $b \mapsto b'$ and $b \mapsto
k$ give rise to a perfect isomorphism.
\end{proof}

\vskip 3em

\section{Construction of crystals $\Bklr{\infty}$ and $\Bklr{\lambda}$}\

In this section, we investigate the crystal structures on the sets
of isomorphism classes of irreducible graded modules over $R$ and
its cyclotomic quotient $R^\lambda$. We assume that $a_{ii} \ne 0$
for all $i\in I$.

\subsection{The crystal $\Bklr{\infty}$}\

Let $\Bklr{\infty}$ be the set of isomorphism classes of
irreducible graded $R$-modules. In this subsection, we define a
crystal structure on $\Bklr{\infty}$ and show that it is
isomorphic to the crystal $B(\infty)$ using the perfect basis
theory given in Section \ref{Sec:perfect bases}.

Let $\alpha \in Q^+$. For any $P \in R(\alpha)$-pmod and $M \in R(\alpha)$-fmod, we define
\begin{equation}\label{Eq:def of f,e',F,E'}
\begin{aligned}
f_i(P) &= \ind_{\alpha_i, \alpha} (P_{(i)} \boxtimes P),  & e'_i(P)& = P^\star \otimes_{R'(\alpha_i)}L(i), \\
F_i(M) &= \ind_{\alpha_i, \alpha} (L(i) \boxtimes M),  & E'_i(M)& = \res_{\alpha-\alpha_i}^{\alpha_i, \alpha-\alpha_i} \circ \Delta_i M,
\end{aligned}
\end{equation}
where $R'(\alpha_i):= R(\alpha_i) \otimes 1_{\alpha-\alpha_i} \hookrightarrow R(\alpha_i) \otimes R(\alpha-\alpha_i) \subset R(\alpha)$.
Here, the $(R(\alpha_i), R(\alpha - \alpha_i))$-bimodule structure of $P^\star$ is given as follows:
for $ v \in P^\star, \ r \in R(\alpha - \alpha_i) $ and $ s \in R(\alpha_i)$,
\begin{align*}
r \cdot v  := ( 1_{\alpha_i}\otimes r) \ v , \quad  v \cdot s  := \psi(s  \otimes 1_{\alpha - \alpha_i} ) \ v   .
\end{align*}
Since $f_i$ and $e'_i$ (resp.\ $F_i$ and $E'_i$) take projective
modules to projective modules (resp.\ finite-dimensional modules
to finite-dimensional modules), they induce the linear maps
\begin{align*}
f_i&:K_0(R) \longrightarrow K_0(R), \qquad e'_i:K_0(R) \longrightarrow K_0(R), \\
F_i&:G_0(R) \longrightarrow G_0(R), \qquad  E'_i:G_0(R) \longrightarrow G_0(R).
\end{align*}
Then we have the following lemma, which is the Khovanov-Lauda-Rouquier algebra version of the equation
$\eqref{eq: special commute}$.
\begin{Lem}\ \label{Lem:relation of boson}
\begin{enumerate}
\item $e'_if_j = \delta_{ij} + q_i^{-a_{ij}} f_j e'_i $ on $K_0(R)$.
\item $E'_iF_j = \delta_{ij} + q_i^{-a_{ij}} F_j E'_i $ on $G_0(R)$.
\end{enumerate}
\end{Lem}
\begin{proof}

(1)
 Fix $\mathbf{i} \in \seq(\alpha)$ and let $ \mathbf{i}' = (j) * \mathbf{i} \in \seq(\alpha+\alpha_j)$. By the equation $\eqref{Eq:ind and res of Pi}$,
\begin{align*}
\Delta_i P_{\mathbf{i}'}  &\simeq \sum_{\mathbf{i}': \text{shuffles of $(i)$ and $\mathbf{j}$}} P_{(i)} \boxtimes P_{\mathbf{j}}
\langle \deg( (i), \mathbf{j}, \mathbf{i}' ) \rangle \\
& \simeq \delta_{ij}P_{(i)}\boxtimes P_{\mathbf{i}} +  \sum_{\mathbf{i}: \text{shuffles of $(i)$ and $\mathbf{k}$}} P_{(i)} \boxtimes P_{(j)*\mathbf{k}}
\langle  -\deg( (i), \mathbf{k}, \mathbf{i} ) + (\alpha_i| \alpha_j) \rangle ,
\end{align*}
which yields
\begin{align*}
e_i' f_j [P_\mathbf{i}] &= e_i'[P_{\mathbf{i}'}] \\
& = [ P_{\mathbf{i}'}^\star \otimes_{R'(\alpha_j)} L(i) ]\\
&= [ (\Delta_i P_{\mathbf{i}'})^\star \otimes_{R'(\alpha_i)} L(i)]\\
&= \delta_{ij} [P_{\mathbf{i}}] + q^{-(\alpha_i|\alpha_j)} f_j[(\res_{\alpha-\alpha_i}^{\alpha_i, \alpha-\alpha_i}( P_{\mathbf{i}}^\star \otimes_{R'(\alpha_i)} L(i) ))] \\
&= \delta_{ij} [P_{\mathbf{i}}] + q^{-(\alpha_i|\alpha_j)} f_j e_i'[P_\mathbf{i}].
\end{align*}

(2) For an irreducible $R(\alpha)$-module $M$, it follows from Proposition \ref{Prop:Mackey} that
\begin{align*}
E_i' F_j [M] &= E_i'([\ind_{\alpha_j, \alpha}L(j) \boxtimes M])\\
&= [ E_i'L(j)] [M] + [\ind_{\alpha_j, \alpha-\alpha_i}L(j) \boxtimes E_i'(M) \langle  (\alpha_j| \alpha_i)  \rangle]\\
&= \delta_{ij}[M] + q^{-(\alpha_i|\alpha_j)} F_j E_i'[M].
\end{align*}
\end{proof}

We also have analogues of the equation $\eqref{Eq:def of ()K and ()L}$ and
Lemma \ref{Lem:nondegenerate pairing in GKM} (3).

\begin{Lem} \ \label{Lem:duality e,f and E,F}
\begin{enumerate}
\item For $[P], [Q] \in K_0(R)$, we have
$$ (e_i'[P], [Q]) = (1-q_i^2)([P], f_i[Q]).  $$
\item For $[P] \in K_0(R)$ and $[M] \in G_0(R)$, we have
$$ (f_i[P], [M]) = ([P], E'_i [M]), \quad (e'_i[P], [M]) = ([P], F_i [M]). $$
\end{enumerate}
\end{Lem}
\begin{proof}
(1) Let $P, Q \in R(\alpha)$-mod. Then 
\begin{align*}
([P], f_i[Q]) & = \qdim(   P^\star \otimes_{R(\alpha+\alpha_i)} (\ind P_{(i)} \boxtimes Q) )\\
&=  \qdim( (\Delta_i P)^\star \otimes_{R(\alpha_i)\otimes R(\alpha)}  ( P_{(i)} \boxtimes Q) ) \\
&= (1-q_i^2)^{-1} \qdim(  (e_i' P)^\star \otimes_{R(\alpha)} Q  )\\
&= (1-q_i^2)^{-1}(e_i'[P], [Q]).
\end{align*}

(2) Let $P \in R(\alpha)$-pmod and $M \in R(\alpha+\alpha_i)$-fmod.
By definition, we have the first assertion:
\begin{align*}
(f_i[P], [M]) &= \qdim( (\ind P_{(i)} \boxtimes P)^\star \otimes_{R(\alpha + \alpha_i)} M ) \\
&= \qdim( ( P_{(i)} \boxtimes P)^\star \otimes_{R(\alpha_i)\otimes R(\alpha)} \Delta_i M ) \\
&= \qdim( P^\star \otimes_{R(\alpha)} \res_{\alpha}^{\alpha_i, \alpha} \Delta_i M ) \\
&= ( [P],  E_i'[M] ).
\end{align*}

In a similar manner, we have
\begin{align*}
(e_i'[P], [M]) &= \qdim \left( ( P^\star \otimes_{R'(\alpha_i)}L(i) ) \otimes_{R(\alpha)} M  \right) \\
&= \qdim \left(  (\Delta_i P)^\star \otimes_{R(\alpha_i) \otimes R(\alpha)}  L(i) \boxtimes M  \right) \\
&= \qdim \left(  P^\star \otimes_{R(\alpha+\alpha_i)}  \ind L(i) \boxtimes M  \right) \\
&= ( [P], F_i [M] ).
\end{align*}
\end{proof}

We now define a $B_q(\g)$-module structure on $K_0(R)_{\Q(q)}$ and
$G_0(R)_{\Q(q)}$ as follows:
\begin{align*}
\bse'_i \cdot [P] &:= e_i'[P], \ \  \bsf_i \cdot [P] := f_i[P] \quad \text{ for } [P] \in K_0(R)_{\Q(q)}, \\
\bse'_i \cdot [M] &:= E_i'[M], \ \ \bsf_i \cdot [M] := F_i[M]
\quad \text{ for } [M] \in G_0(R)_{\Q(q)}.
\end{align*}
By the same argument as in the proof of \cite[Lemma 3.4.2]{Kash91}, it follows from Lemma \ref{Lem:relation of boson}, Lemma \ref{Lem:duality e,f and E,F} and Theorem \ref{Thm:Serre}
that $K_0(R)_{\Q(q)}$ and $G_0(R)_{\Q(q)}$ are well-defined $B_q(\g)$-modules.
Consider the $B_q(\g)$-module homomorphism
$$\Phi^{\vee}_{\Q(q)}: U_q^-(\g) \longrightarrow G_0(R)_{\Q(q)} $$
given by $$\Phi^{\vee}_{\Q(q)}(f_i) = L(i) \ \ \text{for} \  i\in
I.$$ Then, by Theorem \ref{Thm:iso of K0 and Uq}, we obtain the
following diagram.
$$
\xymatrix{
\Phi_{\Q(q)}\ : \ U_q^-(\g)  \ar[rrr]^{\sim} \ar[d]^{\text{dual w.r.t. } (\ ,\ )_K}  &  & & \ar[d]^{\text{dual w.r.t. } (\ ,\ )} K_0(R)_{\Q(q)} \\
\Phi_{\Q(q)}^\vee\ : \ U_q^-(\g) \ar[u]  \ar[rrr]^{\sim} & & &   \ar[u] G_0(R)_{\Q(q)}
}
$$
Therefore, $K_0(R)_{\Q(q)}$ and $G_0(R)_{\Q(q)}$ are well-defined
$B_q(\g)$-modules, which are isomorphic to $U_q^-(\g)$.


The following lemma is the Khovanov-Lauda-Rouquier algebra version of Proposition \ref{Prop:the e_i action global basis}.
\begin{Lem} \label{Lem:perfect condition of U-}
Let $M$ be an irreducible $R(\alpha)$-module and $\ep = \ep_i(M)$. Then we have

\begin{align*}
E_i'[M] = \left\{
            \begin{array}{ll}
              q^{-\ep + 1}_i [\ep]_i[\ke_i M] + \sum_{k}c_k [N_k]  & \hbox{ if } i\in \Ire, \\
              {[\ke_i M]} + \sum_{k}c_k' [N_k'] & \hbox{ if } i \in
              \Iim,
            \end{array}
          \right.
\end{align*}
where $c_k, c_k' \in \Q(q)$ and $\ep_i(N_k),\ \ep_i(N_k') < \ep -
1$.
\end{Lem}
\begin{proof}
If $i= \Ire$, then the assertion can be proved in the same manner as \cite[Lemma 3.9]{KP10}.
Suppose that $i\in \Iim$. By Lemma \ref{Lem:Delta of M},
$$ \Delta_{i^\ep}M \simeq L(i^\ep) \boxtimes N $$
for some irreducible module $N$ with $\ep_i(N)=0$. Then, from
$\eqref{Eq:reciprocity2}$, we have an exact sequence
\begin{align} \label{Eq:perfect for KLR eq1}
0 \rightarrow K \rightarrow \ind_{\ep \alpha_i, \alpha-\ep \alpha_i} L(i^\ep)\boxtimes N
\rightarrow M \rightarrow 0
\end{align}
for some $R(\alpha)$-module $K$. Note that $\ep_i(K) < \ep$.

On the other hand, it follows from $\ep_i(N)=0$ and Lemma \ref{Lem:Kato for i in Iim} that
$$ [\Delta_i \ind_{\ep \alpha_i, \alpha-\ep \alpha_i} L(i^\ep)\boxtimes N ]
 = [\ind_{\alpha_i, (\ep-1) \alpha_i, \alpha-\ep \alpha_i}^{\alpha_i, \alpha- \alpha_i} L(i)\boxtimes L(i^{\ep-1})\boxtimes N]. $$
By Lemma \ref{Lem:properties of L(im) bt N}, Lemma \ref{Lem:soc and hd} and Lemma \ref{Lem:adjoint ke and kf}, we have
$$\hd (\ind_{\alpha_i, (\ep-1) \alpha_i, \alpha-\ep \alpha_i}^{\alpha_i, \alpha- \alpha_i} L(i)\boxtimes L(i^{\ep-1})\boxtimes N) \simeq
L(i)\boxtimes(\kf_i^{\ep -1}N) \simeq L(i)\boxtimes \ke_iM  $$
and all the other composition factors of $\ind_{\alpha_i, (\ep-1) \alpha_i, \alpha-\ep \alpha_i}^{\alpha_i, \alpha- \alpha_i} L(i)\boxtimes L(i^{\ep-1})\boxtimes N$ are
of the form $L(i)\boxtimes L$ with $\ep_i(L) < \ep-1$. Moreover, since $\ep_i(K)<\ep$, all composition factors of $\Delta_i(K)$ are of the form $L(i)\boxtimes L'$ with
$\ep_i(L')$ with $\ep_i(L') < \ep - 1$. Therefore, applying the exact functor $\Delta_i$ to $\eqref{Eq:perfect for KLR eq1}$, we have
$$  E_i'[M] = {[\ke_i M]} + \sum_{k}c_k' [N_k'] $$
for some $R(\alpha)$-modules $N_k'$ with $ \ep_i(N_k') < \ep-1 $.
\end{proof}

For an element $[M] \in \Bklr{\infty}$, we define
\begin{align*}
\wt([M]) &= -\alpha \ \ \text{ if } M \in R(\alpha)\text{-fmod}, \\
\varepsilon_i([M]) &= \left\{
                         \begin{array}{ll}
                            \max\{ k \ge 0 \mid \ke_i^k M \ne 0 \}  & \hbox{ if } i\in \Ire, \\
                           0 & \hbox{ if } i \in \Iim,
                         \end{array}
                       \right.\\
\varphi_i([M]) &= \varepsilon_i(b) + \langle h_i, \wt( [M]) \rangle.
\end{align*}
Then we have the following theorem.

\begin{Thm} \label{Thm: B(infty)}
The sextuple $(\mathfrak{B}(\infty),\wt, \ke_i, \kf_i, \varepsilon_i,
\varphi_i)$ becomes an abstract crystal, which is isomorphic to the crystal
$B(\infty)$ of $U_q^{-}(\g)$.
\end{Thm}
\begin{proof}
It follows from Lemma \ref{Lem:adjoint ke and kf} and Lemma
\ref{Lem:perfect condition of U-} that the pair $(\Bklr{\infty},
\{\ke_i\}_{i\in I})$ is a perfect basis for the $B_q(\g)$-module
$G_{0}(R)_{\Q(q)}$. Hence by Theorem \ref{Thm: uniqueness of
perfect graphs}, $\mathfrak{B}(\infty)$ is isomorphic to
$B(\infty)$.
\end{proof}

\vskip 1em

\subsection{Cyclotomic quotients $R^{\lambda}$ and their crystals $\Bklr{\lambda}$ }\

In this subsection, we define the cyclotomic quotient
$R^{\lambda}$ of $R$ for $\lambda \in P^+$, and investigate the
crystal structure on the set of isomorphism classes of irreducible
$R^{\lambda}$-modules.

For $\alpha \in Q^+$ with $|\alpha|=d$ and $\lambda \in P^+ $, let
$I^{\lambda}(\alpha)$ denote the two-side ideal of $R(\alpha)$
generated by
\begin{equation} \label{Eq:def of cyclotomic ideal}
\begin{aligned}
& \{x_d^{\langle h_{i_d}, \lambda \rangle} 1_\mathbf{i} \mid \mathbf{i}=(i_1, \ldots, i_d) \in \seq(\alpha)  \}. 
\end{aligned}
\end{equation}
Note that it is defined in the opposite manner to \cite[Section
3.4]{KL09}. We define
$$ R^\lambda(\alpha) = R(\alpha) / I^\lambda(\alpha). $$
The algebra $R^{\lambda}:= \bigoplus_{\alpha
\in Q^{+}} R^{\lambda}(\alpha)$ is called the {\it cyclotomic
Khovanov-Lauda-Rouquier algebra} of weight $\lambda$.  For an
irreducible $R(\alpha)$-module $M$, let
$$ \ep^{\vee}_i (M) = \max \{ k \ge 0 \mid 1_{\alpha - k \alpha_i, k \alpha_i } M  \ne 0 \}. $$
This definition is also the opposite to \cite[(5.6)]{LV09}. Combining Lemma \ref{Lem:properties of L(im) bt N} and Lemma \ref{Lem:soc and hd}
with $\eqref{Eq:def of L in Iim}$ and the fact that $x_m^k L(i^m)=0$ for $k \ge m,\ i \in \Ire$, we obtain
\begin{align} \label{Eq:eq condition of being in B(lamda)}
   I^\lambda(\alpha) \cdot M = 0   \  \text{ if and only if }\
 \left\{
   \begin{array}{ll}
     \ep^{\vee}_i (M) \le \langle h_i, \lambda \rangle & \hbox{ for } i \in \Ire, \\
     \ep^{\vee}_i (M) = 0 & \hbox{ for } i \in \Iim \text{ with } \langle h_i, \lambda \rangle=0,
   \end{array}
 \right.
\end{align}
where $M$ is an irreducible $R(\alpha)$-module.
\begin{Lem}
Let $M$ be an irreducible $R(\alpha)$-module.
\begin{enumerate}
\item For $i \in I$, either $\ep_i^{\vee}(\kf_i M) = \ep_i^{\vee}(M)$ or $\ep_i^{\vee}(M)+1$.
\item For $i,j \in I$ with $i \ne j$, we have $\ep_i^{\vee}(\kf_j M) = \ep_i^{\vee}(M)$.
\end{enumerate}
\end{Lem}
\begin{proof}
The proof is the same as that of \cite[Proposition 6.2]{LV09}.
\end{proof}

For $M \in R^\lambda(\alpha)$-$\fMod$ and $N \in R(\alpha)$-$\fMod$,
let $\infl^\lambda M$ be the inflation of $M$,
and $\pr^\lambda N $ be the quotient of $ N $ by $ I^\lambda(\alpha) N$.
Let $\Bklr{\lambda}$ denote the set of isomorphism classes of irreducible graded $R^\lambda$-modules.
For $M \in R^\lambda(\alpha)$-$\fMod$, define
\begin{equation} \label{Eq:def of B(lambda)}
\begin{aligned}
\wt^\lambda(M) &= \lambda - \alpha, \\
\ke_i^\lambda M &=  \pr^\lambda \circ \ke_i \circ \infl^\lambda M, \\
\kf_i^\lambda M &= \pr^\lambda \circ \kf_i \circ \infl^\lambda M, \\
\varepsilon_i^\lambda(M) &= \left\{
                               \begin{array}{ll}
                                 \max\{ k \ge 0 \mid (\ke_i^\lambda)^k M \ne 0 \} & \hbox{ for } i\in \Ire, \\
                                 0 & \hbox{ for } i\in \Iim,
                               \end{array}
                             \right. \\
\varphi_i^\lambda(M) &= \left\{
                           \begin{array}{ll}
                             \max\{ k \ge 0 \mid (\kf_i^\lambda)^k M \ne 0 \} & \hbox{ for } i\in \Ire, \\
                            \langle h_i , \wt^\lambda(M) \rangle & \hbox{ for } i\in \Iim.
                           \end{array}
                         \right.
\end{aligned}
\end{equation}

We will show that $(\mathfrak{B}(\lambda),\wt^\lambda,
\ke_i^\lambda, \kf_i^\lambda, \varepsilon_i^\lambda,
\varphi_i^\lambda)$ is an abstract crystal. For this purpose, we
need several lemmas.

\begin{Lem} \label{Lem:difference of phi}
Let $i \in \Ire$ and $\lambda, \mu \in P^+$.
For $[M], [N] \in \Bklr{\infty}$ with $ \pr^\lambda M \ne \emptyset,\ \pr^{\lambda} N \ne \emptyset,\ \pr^\mu M \ne \emptyset,\ \pr^{\mu} N \ne \emptyset $,
we have
$$ \ph_i^\lambda (M) - \ph_i^\lambda (N) = \ph_i^\mu (M) - \ph_i^\mu (N)  .$$
\end{Lem}
\begin{proof}
The assertion can be proved in the same manner as in \cite[Proposition 6.6, Remark 6.7]{LV09}.
\end{proof}

\begin{Lem} \label{Lem:structure lem for R(mi+j)}
 Let $i \in \Ire$ and $ j \in I$ with $a_{ij}<0$.
\begin{enumerate}
\item If $ m \le -a_{ij}$, then for each $0 \le k \le m $, there exists a unique irreducible
$R(m \alpha_i + \alpha_j)$-module $L(i^{k} j i^{m-k})$ with
$$ \ep_i(L(i^k j i^{m-k})) = k  \ \text{ and }\  \ep_i^{\vee}(L(i^k j i^{m-k})) = m-k .$$
\item If $0 \le k \le -a_{ij}$, then the module
$$ \ind L(i^s) \boxtimes L( i^k j i^{-a_{ij} - k}) \simeq \ind L( i^k j i^{-a_{ij} - k}  ) \boxtimes L(i^s) $$
is irreducible for all $s \ge 0$.
\item If $0 \le k \le -a_{ij} \le c$ and $N$ is an irreducible $R(c\alpha_i + \alpha_j)$-module with $\ep_i(N)=k$, then
we have $c+a_{ij} \le k \le c$ and
$$ N \simeq \ind  L(i^{c+a_{ij}}) \boxtimes L( i^{k-c-a_{ij}} j i^{c-k}) .  $$
\end{enumerate}
\end{Lem}
\begin{proof}
To prove (1), we consider the induced module $\ind L(i^k) \boxtimes L(j) \boxtimes L(i^{m-k})$ for $0 \le k \le m$. Let
$$ K = \Span_\F \{ \tau_w \otimes (t \otimes u \otimes v) \mid  w \in \sg_{m+1}, \ell(w)>0, t \in L(i^k), u \in L(j), v \in L(i^{m-k})  \} . $$
By the same argument as in \cite[Proposition 6.11]{LV09}, we deduce
that $K$ is a proper maximal submodule of $\ind L(i^k) \boxtimes
L(j) \boxtimes L(i^{m-k})$, and that $\hd \ind L(i^k) \boxtimes L(j)
\boxtimes L(i^{m-k})$ is the quotient module  $\ind L(i^k) \boxtimes
L(j) \boxtimes L(i^{m-k}) /K $ which is irreducible. We denote it by
$L(i^{k} j i^{m-k})$. By the Frobenius reciprocity
$\eqref{Eq:reciprocity}$, we have
\begin{align} \label{Eq:epsilon of irr repn}
 \ep_i(L(i^k j i^{m-k})) = k  \ \text{ and }\  \ep_i^{\vee}(L(i^k j i^{m-k})) = m-k .
\end{align}

On the other hand, there is a surjective homomorphism of degree 0
$$ \ind L(i^k) \boxtimes L(j) \boxtimes L(i^{m-k}) \twoheadrightarrow \ke_i^k \ke_j \ke_i^{m-k} \mathbf{1}, $$
which implies that $ L(i^k j i^{m-k}) \simeq \ke_i^k \ke_j \ke_i^{m-k} \mathbf{1}$. By Theorem \ref{Thm: B(infty)},
$$\{ \ke_i^k \ke_j \ke_i^{m-k} \mathbf{1} \mid 0 \le k \le m  \}  $$
is a complete set of irreducible $R(m\alpha_i + \alpha_j)$-module.
Therefore, $L(i^k j i^{m-k})$ is a unique irreducible $R(m\alpha_i +
\alpha_j)$-module satisfying $\eqref{Eq:epsilon of irr repn}$.

The assertion (2), (3) can be proved by the same argument as in \cite[Theorem 6.10]{LV09}.
\end{proof}

Fix $i \in \Ire$ and $ j \in I$ with $i \ne j, a_{ij} \ne 0$ and
let
$$\mathfrak{L}(k) = L(i^{k} j i^{-a_{ij}-k}) \ \ \text{for} \ 0 \le k \le -a_{ij}.$$

\begin{Lem} \label{Lem:surjection for R(ci+dj)}
Let $c,d \in \Z_{\ge 0}$ with $c+d \le -a_{ij}$.
\begin{enumerate}
\item We have
\begin{align*}
\hd \ind L(i^m) \boxtimes L(i^c j i^{d} ) & \simeq \kf_i^m L(i^cji^{d}) \simeq \kf_i^{m+c}L(ji^{d})\\
& \simeq \left\{
           \begin{array}{ll}
             \ind L(i^{m+a_{ij} + c + d}) \boxtimes \mathfrak{L}(-a_{ij}-d) & \hbox{ if } m \ge -a_{ij} - c -d, \\
             \mathfrak{L}(i^{m+c} j i^{d}) & \hbox{ if } m < -a_{ij} - c -d.
           \end{array}
         \right.
\end{align*}
\item Suppose that there is a nonzero homomorphism
$$ \ind L(i^m) \boxtimes \mathfrak{L}(c_1) \boxtimes \cdots  \boxtimes \mathfrak{L}(c_r) \longrightarrow Q $$
where $Q$ is irreducible. Then
$$ \ep_i(Q) = m+ \sum_{t=1}^r c_t\ \text{ and }\ \ep_i^{\vee}(Q) = m + \sum_{t=1}^r (-a_{ij} - c_t) .$$
\item Let $M$ and $Q$ be irreducible. Suppose that there is a nonzero homomorphism $\ind \mathfrak{L}(k) \boxtimes M \rightarrow Q$.
Then $\ep_i(Q) = \ep_i(M)+k$.
\end{enumerate}
\end{Lem}
\begin{proof}
The proof is identical to that of \cite[Lemma 6.13]{LV09}.
\end{proof}

\begin{Lem}\ \label{Lem:surjection2 for R(ci+dj)}
\begin{enumerate}
\item If $N$ is an irreducible $R(c\alpha_i + d\alpha_j)$-module with
$\ep_i(N)=0$, then there exist $r \in \Z_{> 0} $ and $b_t \le
-a_{ij}$ for $1 \le t \le r$ such that
$$ \ind L(j i^{b_1}) \boxtimes \cdots \boxtimes L(j i^{b_r}) \twoheadrightarrow N . $$
\item Let $a := -a_{ij}$. Suppose that we have a surjective
homomorphism
$$ \ind L(i^h) \boxtimes L(j i^{b_1}) \boxtimes \cdots \boxtimes L(j i^{b_r}) \twoheadrightarrow Q, $$
where $Q$ is irreducible.
\begin{enumerate}
\item  If $h \ge  \sum_{t=1}^r (a- b_t)$, then we have a surjective homomorphism
$$ \ind L(i^g) \boxtimes \mathfrak{L}(a-b_1) \boxtimes \cdots \boxtimes \mathfrak{L}(a-b_r) \twoheadrightarrow Q,$$
where $g := h- \sum_{t=1}^r (a-b_t) $.
\item Otherwise, we have
$$ \ind \mathfrak{L}(a-b_1) \boxtimes \cdots \boxtimes \ind \mathfrak{L}(a-b_{s-1}) \boxtimes L(i^{g'}ji^{b_s}) \boxtimes L(j i^{s+1}) \boxtimes \cdots \boxtimes L(ji^{b_1})  \twoheadrightarrow Q,$$
where $g' = h - \sum_{t=1}^{s-1} (a-b_t)$ and $s$ is such that
$$  \sum_{t=1}^{s-1} (a-b_t) \le h < \sum_{t=1}^s (a-b_t). $$
\end{enumerate}
\end{enumerate}
\end{Lem}
\begin{proof}
The assertions can be proved in the same manner as in \cite[Lemma 6.14, Lemma 6.15]{LV09}.
\end{proof}

\begin{Prop} \label{Prop:phi-ep=wt for R(ci+dj)}
Let $i \in \Ire$ and $j \in I$ with $i\ne j$.
Let $M$ be an irreducible $R(c\alpha_i + d\alpha_j)$-module, and $\lambda \in P^+$ such that
$\pr^\lambda(M) \ne 0$ and $\pr^{\lambda}( \kf_j M) \ne 0$. Then we have
$$ \ep_i^\lambda (\kf_jM) = \ep_i^\lambda(M) + a_{ij} + k, \ \ \ph_i^{\lambda}(\kf_j M) = \ph_i^\lambda(M)+k $$
for some $ 0 \le k \le -a_{ij}$.
\end{Prop}
\begin{proof}
Using the argument in \cite[Theorem 6.19]{LV09} with Lemma
\ref{Lem:difference of phi}, Lemma \ref{Lem:structure lem for
R(mi+j)}, Lemma \ref{Lem:surjection for R(ci+dj)} and Lemma
\ref{Lem:surjection2 for R(ci+dj)}, our assertion follows.
\end{proof}

\begin{Prop} \label{Prop:ph - ep = wt}
Let $i \in \Ire$, and $M$ be an irreducible $R(\alpha)$-module with $\pr^\lambda(M) \ne 0 $.
\begin{enumerate}
\item For $j\in I$ with $i \ne j$, we have
$$ \ph_i^\lambda(\kf_j M) - \ep_i^\lambda(\kf_j M) = -\langle h_i, \alpha_j \rangle + \ph_i^\lambda(M) - \ep_i^\lambda(M). $$
\item Moreover, we have
$$ \ph_i^\lambda(M) = \ep_i^\lambda(M) + \langle h_i, \wt^\lambda(M)\rangle. $$
\end{enumerate}
\end{Prop}
\begin{proof}
Combining \cite[Proposition 6.20]{LV09} with Proposition
\ref{Prop:phi-ep=wt for R(ci+dj)}, we obtain the assertion (1).
Since $\ph_i^\lambda(\mathbf{1}) = \ep_i^\lambda(\mathbf{1}) +
\langle h_i, \lambda \rangle$, the assertion (2) follows by
induction on $|\alpha|$ combined with the assertion (1).
\end{proof}

Combining Proposition \ref{Prop:ph - ep = wt} with $\eqref{Eq:def of B(lambda)}$, we obtain the following proposition.
\begin{Prop} \label{Prop:B(lambda) for KLR is crystal}
The sextuple $(\Bklr{\lambda},\wt^\lambda, \ke_i^\lambda,
\kf_i^\lambda, \varepsilon_i^\lambda, \varphi_i^\lambda)$ is an abstract crystal.
\end{Prop}

We would like to show that $\Bklr{\lambda}$ is isomorphic to the
crystal $B(\lambda)$.
For this purpose, we first prove the following lemma.
%

\begin{Lem} \label{Lem: Blambda}
Let $i \in \Iim$ and $M$ be an irreducible $R^\lambda(\alpha)$-module. Then
$$ \langle h_i, \wt^\lambda(M)\rangle \le 0 \ \text{ if and only if }\ \kf^\lambda_i M=0. $$
\end{Lem}
\begin{proof}
Let $\alpha = \sum_{j \in I} k_j \alpha_j $ with $|\alpha|=d$. For
simplicity, we identify $M$ with $\infl^\lambda M$.

We first assume that $\langle h_i, \wt^\lambda(M)\rangle \le 0$. Since $\langle h_i, \lambda \rangle \ge 0$ and $\langle h_i, -\alpha_j\rangle \ge 0$ for all $j \in I$,
we have
$$ \langle h_i, \lambda \rangle = 0 \ \  \text{ and } \ \  k_j=0 \text{ for } j\in I \text{ with } a_{ij}  \ne 0. $$
Take an element $\mathbf{j} = (j_1 \ldots j_d) \in \seq(\alpha)$
such that $ 1_\mathbf{j} M \ne 0$. Note that $a_{i j_k}=0$ for all
$k=1,\ldots, d$. By the Frobenius reciprocity
$\eqref{Eq:reciprocity2}$, we have an embedding
$$ L(i) \boxtimes M \hookrightarrow \Delta_i \kf_i M,  $$
which implies that $1_{(i)*\mathbf{j}} \ (\kf_i M) \ne 0$. Since
$a_{i j_1} =0 $, it follows from the quantum Serre relations that
$$1_{(j_1 i j_2 \ldots j_d) } \ (\kf_i M) \ne 0. $$
Repeating this process, we have
$$ 1_{ \mathbf{j}*(i) } \ (\kf_i M) \ne 0,  $$
which yields that $I^\lambda(\alpha + \alpha_i) \kf_i M \ne 0$ since $1_{ \mathbf{j}*(i) } \in I^\lambda(\alpha + \alpha_i)$.
Therefore, we have the only if part of our assertion.

We now prove the converse. We will actually prove the
contrapositive:
$$\langle h_i, \wt^\lambda(M)\rangle > 0 \ \ \Longrightarrow \ \ \kf^\lambda_i M \ne 0 .$$
Assume that $\langle h_i, \wt^\lambda(M)\rangle > 0$.

First consider the case $\langle h_i, -\alpha \rangle = 0$. In this case,
$\langle h_i, \lambda \rangle > 0$ and $k_j=0$ for $ j\in I$ with $ a_{ij}  \ne 0$.
Take a nonzero element
$v \in L(i)$. By definition, we have
$$ \ind L(i) \boxtimes M = \Span_\F\{ \tau_t \cdots \tau_1 \otimes (v \otimes m ) \mid m \in M,\ 0 \le t \le d  \}. $$
Since $I^\lambda(\alpha) M = 0 $ and $k_j=0$ for $ j\in I$ with $ a_{ij}  \ne 0$, it follows from the definition $\eqref{Eq:def of cyclotomic ideal}$ that
$$ I^\lambda(\alpha+\alpha_i) (\ind L(i) \boxtimes M) = 0. $$
Hence we have $\kf^\lambda_i M \ne 0$.

Now we suppose that $\langle h_i, -\alpha \rangle > 0$.
Take a nonzero element $v$ in $L(i)$, and define $N$ to be the submodule of $\ind L(i) \boxtimes M$ generated by
$$ \mathsf{N} = \{ x_{d+1}^{\langle h_i, \lambda \rangle} \tau_d \cdots \tau_1 1_{(i)*\mathbf{k}} \otimes (v \otimes m) \mid  0 \le t \le d,\ m \in M,\  \mathbf{k} \in \seq(\alpha) \}.$$
As $\langle h_i, -\alpha \rangle > 0$, we have
$$\deg( x_{d+1}^{\langle h_i, \lambda \rangle} \tau_d \cdots \tau_1 1_{(i)*\mathbf{k}} \otimes (v \otimes m)) > \deg(1 \otimes v \otimes m) .$$
Then, as $M$ is $R^\lambda(\alpha)$-module, we have $I^\lambda(\alpha+\alpha_i) (\ind L(i) \boxtimes M)\subset N $.
Hence $(\ind L(i) \boxtimes M )/ N$ is $R^\lambda(\alpha+\alpha^i)$-module.
To prove $\kf^\lambda_i M \ne 0$, it suffices to show that $(\ind L(i) \boxtimes M )/ N$ is nontrivial; i.e., $N$ is proper.

Take $m_0 \in M$ such that $\deg(m_0) \le \deg(m)$ for all $m \in M$.
We claim that $ 1 \otimes (v\otimes m_0)  \notin N$. Suppose that $ 1 \otimes (v\otimes m_0)  \in N$.
Since $I^\lambda(\alpha)M=0$,
it follows from the defining relations
$\eqref{Eq:def rel 1}$ and $\eqref{Eq:def rel 2}$ that
\begin{align*}
x_r ( x_{d+1}^{\langle h_i, \lambda \rangle} \tau_d \cdots \tau_1 1_{(i)*\mathbf{k}} \otimes (v\otimes m)) &= x_{d+1}^{\langle h_i, \lambda \rangle} \tau_d \cdots \tau_1  ( x_{r+1} 1_{(i)*\mathbf{k}} \otimes (v\otimes m) ), \\
\tau_s ( x_{d+1}^{\langle h_i, \lambda \rangle} \tau_d \cdots \tau_1 1_{(i)*\mathbf{k}} \otimes (v\otimes m)) &=  x_{d+1}^{\langle h_i, \lambda \rangle} \tau_d \cdots \tau_1 ( \tau_{s+1} 1_{(i)*\mathbf{k}} \otimes (v\otimes m))
\end{align*}
for $m \in M$, $1 \le r \le d$ and $ 1 \le s \le d-1$. So, the element $1 \otimes (v\otimes m_0)$ can be written as
$$ 1 \otimes (v\otimes m_0) = \sum_j \tau_{t_j} \tau_{t_j+1}  \cdots \tau_d x_{d+1}^k n_j, $$
for some $n_j\in \mathsf{N}$, $ t_j, k \in \Z_{\ge0} $. Since $\langle h_i, -\alpha \rangle > 0$ and $m_0$ is minimal, we have
$$  \deg (1 \otimes (v\otimes m_0)) < \deg(n_j) \le \deg(\tau_{t_j} \tau_{t_j+1}  \cdots \tau_d x_{d+1}^k n_j),$$
which gives a contradiction. Therefore, $1 \otimes (v\otimes m_0)$ is not contained in $N$ and $N$ is proper.
\end{proof}

We are now ready to state and prove the crystal version of
categorification of $V(\lambda)$.  Define a map $\Psi_\lambda: \
\Bklr{\lambda} \ \longrightarrow \ \Bklr{\infty} \otimes
T_{\lambda} \otimes C$ by
$$[M] \longmapsto [\infl^\lambda M] \otimes t_\lambda \otimes c.$$

\begin{Thm}\ \label{Thm: B(lambda)}
\begin{enumerate}
\item $\Psi_\lambda$ is a strict crystal embedding.
\item The crystal $\mathfrak{B}(\lambda)$ is isomorphic to the crystal $B(\lambda)$.
\end{enumerate}
\end{Thm}
\begin{proof}
To prove (1), let $M $ be an irreducible
$R^\lambda(\alpha)$-module and let $M_0 = \infl^\lambda M$. Note
that
\begin{align*}
\ep^{\lambda}_i(M) =\ep_i(M_0), \quad
\varphi_i(M_0)+ \langle h_i,\lambda \rangle =\ep_i(M_0) + \langle h_i,\lambda -\alpha \rangle
 = \varphi^{\lambda}_i(M) \ge 0.
\end{align*}
By the tensor product rule $\eqref{Eq:def of tensor product}$ and
Proposition \ref{Prop:B(lambda) for KLR is crystal}, we have
\begin{align*}
\wt(\Psi_\lambda(M)) &=  \wt( M_0 \otimes t_\lambda \otimes c ) = \lambda - \alpha = \wt^\lambda(M),\\
\ep_i(\Psi_\lambda(M)) & =\ep_i( M_0 \otimes t_{\lambda} \otimes c) =\max \{ \ep_i(M_0), -\langle h_i, \lambda-\alpha \rangle \} = \ep^{\lambda}_i(M),\\
\varphi_i(\Psi_\lambda(M)) &=\ph_i( M_0 \otimes t_{\lambda} \otimes c) =\max \{ \varphi_i(M_0)+ \langle h_i, \lambda \rangle, 0 \} = \varphi_i^\lambda(M).
\end{align*}

On the other hand, it follows from Lemma \ref{Lem: Blambda} that
\begin{align} \label{Eq:eq1 in B(lambda) for GKM}
\langle h_i, \lambda - \alpha+\alpha_i \rangle \le 0\ \Longrightarrow \ \ke_i^\lambda M = 0.
\end{align}
By a direct computation, we have
 \begin{equation*}
 \begin{aligned}
 \ \ &\tilde{f_i}(M_0 \otimes t_\lambda \otimes c)  = \begin{cases} (\tilde{f_i} M_0) \otimes t_{\lambda} \otimes c & \text{ if } \varphi^{\lambda}_i(M) > 0,   \\
                                \quad \quad \quad 0 & \text{ if } \varphi^{\lambda}_i(M) \le 0, \end{cases} \\
 \ \ &\tilde{e_i}(M_0 \otimes t_\lambda \otimes c)  =
 \begin{cases} (\tilde{e_i} M_0) \otimes t_{\lambda} \otimes c & \text{ if } i \in \Ire, \ \varphi^{\lambda}_i(M) \ge 0,    \\
  (\tilde{e_i} M_0) \otimes t_{\lambda} \otimes c & \text{ if } i \in \Iim, \langle h_i, \lambda-\alpha+\alpha_i \rangle  >0,  \\
  \quad \quad \quad 0  & \text{ if } i \in \Iim, \langle h_i, \lambda-\alpha+\alpha_i \rangle \le 0. \end{cases}
  \end{aligned}
 \end{equation*}
By $\eqref{Eq:eq1 in B(lambda) for GKM}$ and Lemma \ref{Lem:
Blambda}, we  get
\begin{align*}
\ke_i(\Psi_\lambda(M)) = \Psi_\lambda(\ke_i^\lambda(M))\ \text{
and }\ \kf_i(\Psi_\lambda(M)) = \Psi_\lambda(\kf_i^\lambda(M)),
\end{align*}
which completes the proof of (1).

Since $\Psi_\lambda$ takes $\mathbf{1}$ to $\mathbf{1} \otimes
t_\lambda \otimes c$, the assertion (2) follows from (1) and
Proposition \ref{Prop: recognition theorem of B-lambda}.
\end{proof}

\vskip 3em


\bibliographystyle{amsplain}


\end{document}